\documentclass[reqno]{amsart}

\usepackage{varioref}
\usepackage[
    eientei
    , +maths
    , +tikz
    , scriptStyle = euler
    , margins = reduced
    , numberTheoremsWithin = subsection
]{kappak} 

\usepackage{enumitem}
\usepackage{mathbbol} 
\usepackage{MnSymbol} 
\usepackage{tikz-cd}
\usepackage{todonotes}

\makeindex

\DeclareSymbolFontAlphabet{\mathbb}{AMSb}
\DeclareSymbolFontAlphabet{\mathbbold}{bbold}

\usetikzlibrary{calc, cd, decorations.pathmorphing, positioning}
\tikzset{treeedge/.style = {above, sloped}}
\tikzset{treenode/.style = {circle, minimum size = 3pt, inner sep = 0, draw, fill}}

\input{macros.tex}

\renewcommand{\ulcorner}{\lefthalfcap}
\renewcommand{\lrcorner}{\righthalfcup}

\newcategory{\qCat}{q}{C}{at}
\newcategory{\Kan}{}{K}{an}

\renewcommand{\bbDelta}{\mathbbold{\Delta}}
\renewcommand{\bbLambda}{\mathbbold{\Lambda}}
\renewcommand{\bbOmega}{\mathbbold{\Omega}}

\newcommand{\presheaf}{\mathrm{\catPP sh}}
\newcommand{\Psh}[1]{{\presheaf (#1)}}
\renewcommand{\PshO}{{\Psh\bbOO}}
\newcommand{\PshDelta}{{\Psh\bbDelta}}

\newcategory{\simplicialpresheaves}{}{S}{p}
\newcommand{\Sp}[1]{\simplicialpresheaves (#1)}
\newcommand{\SpDelta}{{\Sp\bbDelta}}

\newcommand{\SpLambda}{{\Sp\bbLambda}}

\newcommand{\PshLambda}{{\Psh\bbLambda}}

\newcategory{\homotopyCategory}{}{H}{o}
\DeclareMathOperator{\Ho}{\homotopyCategory}

\newcategory{\IAlg}{\rmII}{A}{lg}
\newcategory{\ICat}{\rmII}{C}{at}

\newcategory{\elTree}{el}{T}{ree}

\usepackage{pict2e}
\makeatletter
\newcommand{\adjunction}{
    \mathrel{\vcenter{
        \offinterlineskip\m@th
        \ialign{
            \hfil$##$\hfil\cr
            \longrightarrow\cr
            \noalign{\kern-.2ex}
            \begingroup\setlength\unitlength{.15em}
            \begin{picture} (1, 1)
                \roundcap
                \polyline (0, 0) (1, 0)
                \polyline (0.5, 0) (0.5, 1)
            \end{picture}
            \endgroup\cr
            \longleftarrow\cr
        }
    }}
}
\newcommand{\longadjunction}{
    \mathrel{\vcenter{
        \offinterlineskip\m@th
        \ialign{
            \hfil$##$\hfil\cr
            \relbar\joinrel\relbar\joinrel\relbar\joinrel\relbar\joinrel\relbar\joinrel\rightarrow\cr
            \noalign{\kern-.2ex}
            \begingroup\setlength\unitlength{.15em}
            \begin{picture} (1, 1)
                \roundcap
                \polyline (0, 0) (1, 0)
                \polyline (0.5, 0) (0.5, 1)
            \end{picture}
            \endgroup\cr
            \leftarrow\joinrel\relbar\joinrel\relbar\joinrel\relbar\joinrel\relbar\joinrel\relbar\cr
        }
    }}
}
\makeatother

\newcommand{\Cell}[1]{\mathsf{Cell} \left( #1 \right)}

\newcommand{\longhookrightarrow}{
    \mathrel{
        \hookrightarrow
        \!\!\!\!
        \tikz \fill[white] (0, 0) rectangle (.1,.4);
        \!\!\!\!\!\!
        \relbar
        \!\!
        \rightarrow
    }
}

\newcommand{\subto}{\longhookrightarrow}

\renewcommand{\bot}{\perp}

\newcommand{\boxproduct}{\mathbin{\boxtimes}}
\newcommand{\lboxproduct}{\mathbin{\hat{\boxtimes}}}

\newcommand{\qqquad}{\qquad\qquad}

\newcommand{\noop}[1]{}

\title{Opetopic algebras III: Presheaf models of homotopy-coherent opetopic algebras}
\date{January 2020}

\author[C. Ho Thanh]{C\'edric Ho Thanh}
\address{
    Institut de Recherche en Informatique Fondamentale (IRIF),
    Universit\'e Paris Diderot,
    France
}
\email{cedric.hothanh@irif.fr}
\urladdr{hothanh.fr/cedric}

\author[C. Leena Subramaniam]{Chaitanya Leena Subramaniam}
\address{
    Institut de Recherche en Informatique Fondamentale (IRIF),
    Universit\'e Paris Diderot,
    France
}
\email{chaitanya@irif.fr}
\urladdr{www.irif.fr/~chaitanya}


\begin{document}

\begin{abstract}
    For $\Lambda$ the category of free opetopic algebras, we construct a model
    structure \emph{\`a la Cisinski} on the category of presheaves over
    $\Lambda$, and show that is is equivalent to opetopic complete Segal
    spaces. This generalizes the results of Joyal and Tierney in the simplicial
    case, and of Cisinski and Moerdijk in the planar dendroidal case.
\end{abstract}

\maketitle

\begin{small}
    \tableofcontents
\end{small}

\section{Introduction}

This work is a continuation of \cite{HoThanh2020} and \cite{HoThanh2020a}. In
the former, we introduced the notion of opetopic algebra, which can be thought
of as algebras whose operations have \emph{higher dimensional arities}. In the
latter, we lifted the notion of Segal and complete Segal space from the
simplicial and dendroidal setting \cite{Rezk2001,Cisinski2011} to our opetopic
framework. In this paper, we provide a presheaf model for complete Segal
spaces, namely \emph{opetopic $\infty$-algebras}. Much like in the simplicial
case, where they are called \emph{quasi-categories} \cite{Boardman1973,Joyal},
$\infty$-algebras are opetopic algebras where laws only hold up to homotopy,
instead of strictly. Our strategy to prove this equivalence is largely adapted
from that of Joyal and Tierney \cite{Joyal2007}.

\subsection{Related work}

Complete Segal spaces, introduced in \cite{Rezk2001}, are simplicial spaces
that have compositions and identities up to homotopy. In \cite{Joyal2007},
Joyal and Tierney showed that they can be modeled by quasi-categories, that is,
fibrant objects in a suitable model structure on the category $\PshDelta$ of
simplicial sets. Specifically, they show that the discrete space functor is a
left Quillen equivalence
\[
    (-)^\mathrm{disc} : \PshDelta_{\mathrm{Joyal}}
        \adjunction \SpDelta_{\mathrm{Rezk}} : (-)_{-, 0}
\]
where the right adjoint is the ``first row'' functor.

The notion of complete Segal space has been lifted to the dendroidal setting by
Cisinski and Moerdijk in \cite{Cisinski2013}. In this setting, they are
dendroidal spaces having operadic compositions and identities up to homotopy.
In \cite{Cisinski2011}, they show that $\infty$-operads, model complete
dendroidal Segal spaces in the same sense as before. Furthermode, since the
category of simplicial sets (resp. spaces) can be recovered as a slice of the
category of dendroidal sets (resp. spaces), the equivalence in the simplicial
case can be recovered as a special case.

\subsection{Plan}

In section \cref{sec:preliminaries}, we recall elements of Cisinski's
homotopical machinery from \cite{Cisinski2006}, and specify the so-called
Joyal--Tierney calculus \cite{Joyal2007} to an arbitrary category of simplicial
presheaves. In \cref{sec:opetopic-algebras}, we survey some results of
\cite{HoThanh2020}, and construct the folk model structure for opetopic
algebras. In \cref{sec:infty-algebra}, we consider presheaves over $\bbLambda$,
the category of free opetopic algebras, we construct the model structure for
$\infty$-algebras following the methods of Cisinski presented in
\cref{sec:cisinski-homotopy-theory}. Lastly, in
\cref{sec:infty-algebras-vs-complete-segal-spaces}, we establish a Quillen
equivalence between $\infty$-algebras and complete Segal spaces of
\cite{HoThanh2020a}.

\subsection{Acknowledgments}

We are grateful to Pierre-Louis Curien for his patient guidance and reviews.
The first named author has received funding from the European Union's Horizon
2020 research and innovation program under the Marie Sklodowska--Curie grant
agreement No. 665850.


\section{Preliminaries}
\label{sec:preliminaries}

\subsection{Cisinski homotopy theory}
\label{sec:cisinski-homotopy-theory}

We recall some results and constructions of \cite{Cisinski2006}.

\begin{definition}
    [Lifting property]
    \label{def:lifting-property}
    Let $\catCC$ be a category, and $l,r \in \catCC^\rightarrow$. We say that
    $l$ has the \emph{left lifting property against}\index{lifting property}
    $r$ (equivalently, $r$ has the \emph{right lifting property against} $l$),
    written $l \pitchfork r$\index{$\pitchfork$|see {lifting property}}, if for
    any solid commutative square as follows, there exists a (non necessarily
    unique) dashed arrow making the two triangles commute:
    \begin{equation}
        \label{eq:lifting-property}
        \begin{tikzcd}
            \cdot \ar[d, "l" left] \ar[r] & \cdot \ar[d, "r"] \\
            \cdot \ar[r] \ar[ur, dashed] & \cdot
        \end{tikzcd}
    \end{equation}
    If $\catCC$ has a terminal object $1$, then for all $X \in \catCC$, we
    write $l \pitchfork X$ for $l \pitchfork (X \twoheadrightarrow 1)$. Let
    $\sfLL$ and $\sfRR$ be two classes of morphisms of $\catCC$. We write
    $\sfLL \pitchfork \sfRR$ if for all $l \in \sfLL$ and $r \in \sfRR$ we have
    $l \pitchfork r$. The class of all morphisms $r$ (resp. $l$) such that
    $\sfLL \pitchfork r$ (resp. $l \pitchfork \sfRR$) is denoted
    $\sfLL^\pitchfork$ (resp ${}^\pitchfork \sfRR$).
\end{definition}

\begin{definition}
    [Orthogonality]
    \label{def:orthogonality}
    We say that $l$ is \emph{left orthogonal to}\index{orthogonal} $r$
    (equivalently, $r$ is \emph{right orthogonal} to $l$), written $l \perp
    r$\index{$\perp$|see {orthogonal}}, if for any solid commutative square as
    in \cref{eq:lifting-property}, there exists a \emph{unique} dashed arrow
    making the two triangles commute. The relation $\perp$ is also known as the
    \emph{unique lifting property}\index{unique lifting property|see
    {orthogonal}}. Note that notations of \cref{def:lifting-property} also make
    sense when $\pitchfork$ is replaced by $\perp$.
\end{definition}

\begin{notation}
    For $\catCC$ a small category, let $\Psh\catCC \eqdef \Cat (\catCC^\op,
    \Set)$\index{$\Psh\catCC$} be the category of $\Set$-presheaves over
    $\catCC$, and $\Sp\catCC \eqdef \Cat (\catCC^\op,
    \PshDelta)$\index{$\Sp\catCC$} be the category of simplicial presheaves
    over $\catCC$.
\end{notation}

\begin{definition}
    [{Cylinder object, \cite[definition 1.3.1]{Cisinski2006}}]
    \label{def:cylinder}
    Let $\catCC$ be a small category, and $X \in \Psh\catCC$ be a presheaf over
    $\catCC$. A \emph{cylinder}\index{cylinder} of $X$ is a factorization of the fold map
    \[
        \triangleDdiagram
            {X + X}{X}{\rmII X ,}{\nabla}{(i_0, i_1)}{\nabla}
    \]
    such that $(i_0, i_1) : X + X \longrightarrow \rmII X$ is a monomorphism. We
    write $X^{(e)}$ for the image of $i_e : X \longrightarrow \rmII X$.
\end{definition}

\begin{definition}
    [{$\rmII$-homotopy, \cite[definition 1.3.3, remark 1.3.4]{Cisinski2006}}]
    \label{def:homotopy}
    Let $\catCC$ be a small category, $f, g : X \longrightarrow Y$ be two
    parallel maps in $\Psh\catCC$, and $\rmII X$ be a cylinder of $X$
    (\cref{def:cylinder}). An \emph{elementary
    $\rmII$-homotopy}\index{elementary homotopy|see {homotopy}} from $f$ to $g$
    is a morphism $H : \rmII A \longrightarrow B$ such that the following
    triangle commutes:
    \[
        \triangleDLdiagram
            {A + A}{\rmII A}{B .}
            {(i_0, i_1)}{(f, g)}{H}
    \]
    Let $\simeq_\rmII$ (or just $\simeq$ is the context is clear), the
    \emph{$\rmII$-homotopy
    relation}\index{homotopy}\index{I-homotopy@$\frakII$-homotopy|see
    {homotopy}}, be the equivalence relation spanned by this relation on
    $\Psh\catCC (A, B)$.

    On readily checks that $\simeq$ is a congruence on the category
    $\Psh\catCC$, and let $\Ho\Psh\catCC$ be the quotient category. A morphism
    $f : X \longrightarrow Y$ is a \emph{$\rmII$-homotopy
    equivalence}\index{homotopy equivalence} (or just \emph{homotopy
    equivalence} if the context is clear) if it is invertible in
    $\Ho\Psh\catCC$.
\end{definition}

\begin{definition}
    [{Elementary homotopical data, \cite[definition 1.3.6]{Cisinski2006}}]
    \label{def:elementary-homotopical-data}
    Let $\catCC$ be a small category. An \emph{elementary homotopical
    data}\index{elementary homotopical date} on $\Psh\catCC$ is a functorial
    cylinder $\rmII : \Psh\catCC \longrightarrow \Psh\catCC$
    (\cref{def:cylinder}) such that
    \begin{enumerate}
        \item \condition{DH1} $\rmII$ preserves small colimits and monomorphisms;
        \item \condition{DH2} for all monomorphism $f : X \longrightarrow Y$ in
        $\Psh\catCC$, and for $e = 0, 1$, the following square is a pullback:
        \[
            \squarediagram
                {X}{Y}{\rmII X}{IY .}
                {f}{i_e}{i_e}{\rmII f}
        \]
    \end{enumerate}
\end{definition}

\begin{definition}
    [{Anodyne extension}]
    \label{def:anodyne-extension}
    Let $\catCC$ be a small category, and $\rmII : \Psh\catCC \longrightarrow
    \Psh\catCC$ be an elementary homotopical data on $\Psh\catCC$
    (\cref{def:elementary-homotopical-data}). A
    \emph{class of anodyne extensions}\index{anodyne extension} relative to $\rmII$ is a class $\mathsf{An}
    \subseteq \Psh\catCC^\rightarrow$ such that
    \begin{enumerate}
        \item \condition{An0} there exists a \emph{set} $A \subseteq
        \Psh\catCC^\rightarrow$ of monomorphisms such that $\mathsf{An} =
        {}^\pitchfork \left( A^\pitchfork \right)$
        (\cref{def:lifting-property});
        \item \condition{An1} for all monomorphism $m : X \longhookrightarrow
        Y$ in $\Psh\catCC$, and $e = 0, 1$, the cocartesian gap map $g$ is in
        $\mathsf{An}$:
        \[
            \begin{tikzcd}
                X
                    \ar[r, hook, "m"] \ar[d, "i_e" left]
                    \ar[dr, phantom, "\ulcorner", very near end] &
                Y
                    \ar[d, "i_e"] \ar[ddr, bend left] \\
                \rmII X
                    \ar[r, hook] \ar[drr, bend right] &
                \rmII X \cup Y^{(e)}
                    \ar[dr, dashed, "g" description] \\
                &
                &
                \rmII Y ;
            \end{tikzcd}
        \]
        \item \condition{An2} for all $m : X \longhookrightarrow Y$ in
        $\mathsf{An}$, the cocartesian gap map $g$ is in $\mathsf{An}$:
        \[
            \begin{tikzcd}
                X + X
                    \ar[r, hook, "m + m"] \ar[d]
                    \ar[dr, phantom, "\ulcorner", very near end] &
                Y + Y
                    \ar[d] \ar[ddr, bend left] \\
                \rmII X
                    \ar[r, hook] \ar[drr, bend right] &
                \rmII X \cup (Y + Y)
                    \ar[dr, dashed, "g" description] \\
                &
                &
                \rmII Y .
            \end{tikzcd}
        \]
    \end{enumerate}
\end{definition}

\begin{definition}
    [{Homotopical structure, \cite[definition 1.3.14]{Cisinski2006}}] Let
    \label{def:homotopical-structure}
    $\catCC$ be a small category. A \emph{homotopical
    structure}\index{homotopical structure} on $\Psh\catCC$ is a pair $(\rmII,
    \mathsf{An})$, where $\rmII$ is a functorial cylinder on $\Psh\catCC$
    (\cref{def:cylinder}), and $\mathsf{An}$ is a class of anodyne extension
    relative to $\rmII$ \cref{def:anodyne-extension}.
\end{definition}

\begin{definition}
    [{Cisinski model category}]
    \label{def:cisinski-model-structure}
    A \emph{Cisinski model category}\index{Cisinski model category} is a model
    structure on a presheaf category over a small category, that is cofibrantly
    generated, and where cofibrations are the monomorphisms.
\end{definition}

A notable source of such structures is the following theorem:

\begin{theorem}
    [{\cite[definition 1.3.21, theorem 1.3.22]{Cisinski2006}}]
    \label{th:cisinsli-model-structure}
    Let $\catCC$ be a small category, and $(\rmII, \mathsf{An})$ be a homotopical
    structure on $\Psh\catCC$ (\cref{def:homotopical-structure}). There is a
    model structure on $\Psh\catCC$ such that:
    \begin{enumerate}
        \item a morphism $f$ is a \emph{naive fibration}\index{fibration!naive}
        if $\mathsf{An} \pitchfork f$; a presheaf $X \in \PshLambda$ is
        \emph{fibrant}\index{fibrant object} if the terminal morphism $X
        \longrightarrow 1$ is a naive fibration;
        \item a morphism $f : X \longrightarrow Y$ is a \emph{weak
        equivalence}\index{weak equivalence} if for all fibrant object $P \in
        \PshLambda$, the induced map $f^* : \Ho\PshLambda (Y, P)
        \longrightarrow \Ho\PshLambda (Y, P)$ is a bijection
        (\cref{def:homotopy});
        \item a morphism $f$ is a \emph{cofibrations}\index{cofibration} if it
        is a monomorphisms, it is a \emph{acyclic
        cofibrations}\index{cofibration!acyclic} if it is a cofibration and a
        weak equivalence;
        \item a morphism $f$ is a \emph{fibration}\index{fibration} if it has
        the right lifting property with respect to acyclic cofibrations, it is
        an \emph{acyclic fibration}\index{fibration!acyclic} if it has the
        right lifting property with respect to all cofibrations.
    \end{enumerate}
    This model structure is of Cisinski type, cellular, and proper.
\end{theorem}

\begin{lemma}
    [{\cite[proposition 1.3.36]{Cisinski2006}}]
    \label{lemma:cisinski:1.3.36}
    Let $\catCC$ be a small category, and $\Psh\catCC$ be endowed with a model
    structure as in \cref{th:cisinsli-model-structure}. Let $f : X
    \longrightarrow Y$ be a morphism in $\Psh\catCC$, and assume that $Y$ is
    fibrant. Then $f$ is a fibration if and only if it is a naive fibration.
\end{lemma}


\begin{definition}
    [Skeletal category]
    \label{def:skeletal-category}
    A \emph{skeletal category}\index{skeletal category} \cite[definition
    8.1.1]{Cisinski2006} is a small category $\catCC$ endowed with a map $\deg
    : \ob \catCC \longrightarrow \bbNN$ and two wide subcategories $\catCC_+$
    and $\catCC_-$ such that the following axioms are satisfied.
    \begin{itemize}
        \item \condition{Sq0} (Invariance) Isomorphisms are in $\catCC_+$ and
        $\catCC_-$. Furthermore, if $c, c' \in \catCC$ are isomorphic, then
        $\deg c = \deg c'$.
        \item \condition{Sq1} (Dimension) If $f : c \longrightarrow c'$ is an
        arrow in $\catCC_+$ (resp. $\catCC_-$) that is not an isomorphism, then
        $\deg c < \deg c'$ (resp. $\deg c > \deg c'$).
        \item \condition{Sq2} (Factorization) Every arrow $f$ of $\catCC$ can
        be essentially uniquely factored as $f = f_+ f_-$, with $f_+ \in
        \catCC_+$ and $f_- \in \catCC_-$.
        \item \condition{Sq3} (Section) Every arrow in $\catCC_-$ admits a
        section. Two arrows $f, f' \in \catCC_-$ are equal if and only if they
        have the same sections.
    \end{itemize}
    The skeletal category $\catCC$ is \emph{normal}\index{normal skeletal
    category} \cite[definition 8.1.36, proposition 8.1.37]{Cisinski2006} if
    $\catCC$ is rigid, i.e. does not have non-trivial automorphisms. Note that
    in this case, $\catCC$ is a Reedy category \cite[defintion
    15.1.12]{Hirschhorn2009}.
\end{definition}

\begin{definition}
    [Boundary]
    \label{def:boundary}
    Let $\catCC$ be a skeletal category. The \emph{boundary} \cite[paragraph
    8.1.30]{Cisinski2006} $\partial c \in \Psh\catCC$ of an object $c \in
    \catCC$ is defined as
    \[
        \partial c
        \eqdef
        \colim_{\substack{
            f : d \rightarrow c \\
            f \in \catCC_+ \text{ not an iso}
        }} d .
    \]
    The canonical map $\sfB_c : \partial c \longrightarrow c$ is a monomorphism
    and is called the \emph{boundary inclusion} of $c$. Write $\sfBB_\catCC$
    (or just $\sfBB$ if the context is clear) for the set of boundary
    inclusions of $\catCC$.
\end{definition}

\begin{proposition}
    [{\cite[propositions 8.1.35 and 8.1.37]{Cisinski2006}}]
    \label{prop:cisinski:boundaries-generate-monomorphisms}
    Assume $\catCC$ is a normal skeletal category. Then the class of
    monomorphisms of $\Psh\catCC$ is $\Cell {\sfBB_\catCC}$.
\end{proposition}

\subsection{Joyal--Tierney calculus}
\label{sec:joyal-tierney-calculus}

Consider the fully faithful functor $i_\catCC : \catCC \subto \catCC \times
\bbDelta$ mapping $c \in \catCC$ to the tuple $(c, [0])$. It induces an
adjunction
\[
    (-)^{\mathrm{disc}} : \Psh\catCC \adjunction \Sp\catCC : (-)_{-, 0} ,
\]
where $(-)^{\mathrm{disc}}$\index{$(-)^{\mathrm{disc}}$|see {discrete space}}
is the left Kan extension of the composite $\catCC \longrightarrow \catCC
\times \bbDelta \longrightarrow \Sp\catCC$ along the Yoneda embedding, and
where $(-)_{-, 0}$ is the precomposition by $i_\catCC$, i.e. the ``evaluation
at $0$''. We call $(-)^{\mathrm{disc}}$ the \emph{discrete space
functor}\index{discrete space}, as for $X \in \Psh\catCC$ and $c \in \catCC$,
the space $X^{\mathrm{disc}}_c$ is discrete at $X_c$.

Dually, the projection $p_\bbDelta : \catCC \times \bbDelta \longrightarrow
\bbDelta$ induces an adjunction
\[
    (-)^{\mathrm{const}} : \PshDelta \adjunction \Sp\catCC : r ,
\]\index{$(-)^{\mathrm{const}}$|see {constant space}}
where $r$ is the right Kan extension of $\catCC \times \bbDelta \longrightarrow
\bbDelta \longrightarrow \PshDelta$ along the Yoneda embedding, and
$(-)^{\mathrm{const}}$ is the precomposition by $p_\bbDelta$. We call
$(-)^{\mathrm{const}}$ the \emph{constant space functor}\index{constant space},
as for $Y \in \PshDelta$, the functor $Y^{\mathrm{const}} : \catCC^\op
\longrightarrow \PshDelta$ is constant at $Y$. The functor $r : \Sp\catCC
\longrightarrow \PshDelta$ provides a simplicial enrichment on $\Sp\catCC$ as
follows:
\[
    \Map (X, Y) \eqdef r (Y^X) .
\]
Note that $\Map (X, Y)_n = \Sp\catCC (X \times \Delta [n]^{\mathrm{const}},
Y)$.

\begin{proposition}[{Generalization of \cite[proposition 2.4]{Joyal2007}}]
    \label{prop:joyal-tierney:2.4}
    The simplicially enriched category $\Sp\catCC$ is tensored and
    cotensored over $\PshDelta$: for $K \in \PshDelta$ and $X \in \Sp \catCC$,
    we define
    \[
        K \otimes X \eqdef K^\mathrm{const} \times X ,
        \qqquad
        X^K \eqdef X^{K^\mathrm{const}} .
    \]
\end{proposition}
\begin{proof}
    Let $Y \in \Sp\catCC$.
    \begin{enumerate}
        \item We have
        \begin{align*}
            \Map (K \otimes X, Y)
            &= \Map (K^\mathrm{const} \times X, Y)
            = \int_{c \in \catCC} \Map (K \times X_c, Y_c)
            \cong \int_{c \in \catCC} \Map \left(K, \Map (X_c, Y_c) \right) \\
            &\cong \Map (K, \int_{c \in \catCC} \Map (X_c, Y_c))
            = \Map \left( K, \Map (X, Y) \right) ,
        \end{align*}
        naturally in all variables, and thus $- \otimes X$ is an enriched left
        adjoint to $\Map (X, -)$.
        \item We have
        \[
            \Map (X, Y^K)
            = \Map (X, Y^{K^\mathrm{const}})
            \cong \Map (K^{\mathrm{const}} \times X, Y)
            = \Map (K \otimes X, Y)
        \]
        naturally in all variables, and thus $(-)^K$ is an enriched right
        adjoint to $K \otimes -$. \qedhere
    \end{enumerate}
\end{proof}

We now specify the so-called Joyal--Tierney calculus \cite[section
2]{Joyal2007} and \cite[section 4]{Riehl2014} to the category $\Sp\catCC$ of
simplicial presheaves over $\catCC$. This convenient formalism will be heavily
used throughout this work.

\begin{definition}
    [Box product]
    \label{def:box-product}
    The \emph{box product}\footnote{The box product is denoted by $\medsquare$
    in \cite[section 2]{Joyal2007}. In \cite[notation 4.1]{Riehl2014}, it is
    called the \emph{exterior product}\index{exterior product} and written
    $\underline{\times}$.}\index{box product} functor $\boxproduct : \Psh\catCC
    \times \PshDelta \longrightarrow \Sp\catCC$ is defined as
    \[
        (X \boxproduct Y)_{c, n} \eqdef X_c \times Y_n ,
    \]
    for $X \in \Psh\catCC$, $Y \in \PshDelta$, $c \in \catCC$, and $[n] \in
    \bbDelta$.

    This functor is divisible on both sides, meaning that is is left adjoint in
    each variable. The right adjoint to $X \boxproduct - : \PshDelta
    \longrightarrow \Sp\catCC$ will be denoted by $X \backslash - : \Sp\catCC
    \longrightarrow \PshDelta$, and the right adjoint to $- \boxproduct Y :
    \Psh\catCC \longrightarrow \Sp\catCC$ will be denoted by $- / Y : \Sp
    (\catCC) \longrightarrow \Psh\catCC$. Note that $X \backslash -$ and $- /
    Y$ are contravariant in $X$ and $Y$ respectively. Consequently, for $W \in
    \Sp (\catCC)$, the functors $- \backslash W : \Psh\catCC \longrightarrow
    \PshDelta$ and $W / - : \PshDelta \longrightarrow \Psh\catCC$ are mutually
    right adjoint.
\end{definition}

\begin{lemma}
    \label{lemma:joyal-tierney:easy-computations}
    \begin{enumerate}
        \item For $X \in \Psh\catCC$, we have $X^{\mathrm{disc}} = X
        \boxproduct \Delta [0]$. Dually, for $Y \in \PshDelta$, we have
        $Y^{\mathrm{const}} = 1 \boxproduct Y$, where $1$ is the terminal
        presheaf in $\Psh\catCC$.
        \item Let $W \in \Sp\catCC$. For $c \in \catCC$, we have $c
        \backslash W = W_c \in \PshDelta$. Dually, for $[n] \in \bbDelta$, we
        have $W / \Delta [n] = W_{-, n}$.
        \item Let $h$ be a morphism in $\Sp\catCC$, and $K \in \Psh\catCC$.
        Then $K \backslash h = \langle (\emptyset \hookrightarrow K) \backslash
        h \rangle$. Similarly, if $L \in \PshDelta$, then $h / L = \langle h /
        (\emptyset \hookrightarrow L) \rangle$.
        \item Let $X \in \Sp\catCC$, and $f$ be a morphism in $\Psh\catCC$.
        Then $f \backslash X = \langle f \backslash (X \twoheadrightarrow 1)
        \rangle$. Likewise, if $g$ is a morphism in $\PshDelta$, then $X / g =
        \langle (X \twoheadrightarrow 1) / g \rangle$.
    \end{enumerate}
\end{lemma}

\begin{definition}
    [{Leibniz construction, \cite[definition 4.4]{Riehl2014}}]
    \label{def:leibniz-construction}
    Consider a functor $- \otimes - : \catAA \times \catBB \longrightarrow
    \catCC$, where $\catCC$ has pushouts. Its \emph{Leibniz
    construction}\index{Leibniz construction} is the functor $-
    \mathbin{\hat{\otimes}} - : \catAA^\rightarrow \times \catBB^\rightarrow
    \longrightarrow \catCC^\rightarrow$\index{$\mathbin{\hat{\otimes}}$|see
    {Leibniz construction}} which maps an arrow $f : A_1 \longrightarrow A_2$
    in $\catAA$ and $g : B_1 \longrightarrow B_2$ in $\catBB$ to the
    cocartesian gap map below:
    \[
        \begin{tikzcd}
            A_1 \otimes B_1
                \ar[r, "A_1 \otimes g"] \ar[d, "f \otimes B_1" left]
                \ar[dr, phantom, "\ulcorner", very near end] &
            A_1 \otimes B_2
                \ar[d, "f'"] \ar[ddr, bend left, "f \otimes B_2"] \\
            A_2 \otimes B_1
                \ar[r, "g'"] \ar[drr, bend right, "A_2 \otimes g"] &
            P
                \ar[dr, dashed, "f \mathbin{\hat{\otimes}} g" description] \\
            &
            &
            A_2 \otimes B_2 ,
        \end{tikzcd}
        \qqquad
        P = A_1 \otimes B_2 \coprod_{A_1 \otimes B_1}
            A_2 \otimes B_1 .
    \]
    The Leibniz construction $- \mathbin{\hat{\otimes}} -$ essentially has the
    same properties as $- \otimes -$, see \cite[section 4]{Riehl2014}.
\end{definition}

\begin{definition}
    [Leibniz box product]
    \label{def:leibniz-box-product}
    The \emph{Leibniz box product}\footnote{We follow the terminology and
    notation of \cite{Riehl2014,Riehl2019}. In \cite[section 2]{Joyal2007}, it
    is denoted by $\medsquare'$.}\index{Leibniz box product} $- \lboxproduct -
    : \Psh\catCC^\rightarrow \times \PshDelta^\rightarrow \longrightarrow
    \Sp\catCC^\rightarrow$\index{$\lboxproduct$|see {Leibniz box product}} is
    simply the Leibniz construction of \cref{def:leibniz-construction} applied
    to the box product of \cref{def:box-product}. If $\sfKK$ and $\sfLL$ are
    classes of morphisms of $\Psh\catCC$ and $\PshDelta$ respectively, let
    $\sfKK \lboxproduct \sfLL \eqdef \left\{ k \lboxproduct l \mid k \in \sfKK,
    l \in \sfLL \right\}$.

    Akin to the box product, the Leibniz box product is divisible on both
    sides. Specifically, if $h : W_1 \longrightarrow W_2$ is a morphism in
    $\Sp\catCC$, let $\langle f \backslash h \rangle$\index{$\langle -
    \backslash - \rangle$|see {Leibniz box product}}, be the morphism induced
    by the universal property of the pullback:
    \[
        \begin{tikzcd}
            X_2 \backslash W_1
                \ar[drr, bend left, "f \backslash W_1"]
                \ar[dr, "\langle f \backslash h \rangle"]
                \ar[ddr, bend right, "X_2 \backslash h"]
            & & \\
            &
            P
                \ar[r]
                \ar[d]
                \ar[dr, phantom, "\lrcorner" very near start] &
            X_1 \backslash W_1
                \ar[d, "X_1 \backslash h"] \\
            &
            X_2 \backslash W_2
                \ar[r, "f \backslash W_2"] &
            X_1 \backslash W_2 ,
        \end{tikzcd}
        \qqquad
        P = \left( X_1 \backslash W_1 \right) \prod_{X_1 \backslash W_2}
            \left( X_2 \backslash W_2 \right) .
    \]
    The morphism $\langle f \backslash h \rangle$ is also called the
    \emph{cartesian gap map}\index{cartesian gap} of the square above. Then the
    functor $\langle f \backslash - \rangle$ is right adjoint to $f
    \lboxproduct -$. Dually, let $\langle h / g \rangle$\index{$\langle - / -
    \rangle$|see {Leibniz box product}}, on the left, be the cartesian gap map
    of the square on the right
    \[
        \begin{tikzcd}
            W_1 / Y_2
                \ar[drr, bend left, "h / Y_2"]
                \ar[dr, "\langle h / g \rangle"]
                \ar[ddr, bend right, "W_1 / g"]
            & & \\
            &
            P
                \ar[r]
                \ar[d]
                \ar[dr, phantom, "\lrcorner" very near start] &
            W_2 / Y_2
                \ar[d, "W_2 / g"] \\
            &
            W_1 / Y_1
                \ar[r, "h / Y_1"] &
            W_2 / Y_1
        \end{tikzcd}
        \qqquad
        P = \left( W_1 / Y_1 \right) \prod_{W_2 / Y_1}
            \left( W_2 / Y_2 \right) .
    \]
    Then the functor $\langle - / g \rangle$ is right adjoint to $-
    \lboxproduct g$.
\end{definition}

\subsection{Application to simplicial presheaves}
\label{sec:simplicial-presheaves}

In this section, we lift some technical results of \cite[section 2]{Joyal2007}
from the settings of simplicial spaces to simplicial presheaves over a normal
skeletal category $\catCC$ (\cref{def:skeletal-category}). Recall that by
\cref{prop:cisinski:boundaries-generate-monomorphisms}, the set of boundary
inclusions $\sfBB_\catCC \eqdef \left\{ \sfB_c : \partial c \longhookrightarrow
c \mid c \in \catCC \right\}$ generates the class of monomorphisms of
$\Psh\catCC$, in the sense that it is $\Cell{\sfBB_\catCC}$.

\begin{proposition}
    \label{prop:spc:generating-monomorphisms}
    The class of monomorphisms of $\Sp\catCC$ is $\Cell {\sfBB_\catCC
    \lboxproduct \sfBB_\bbDelta}$.
\end{proposition}
\begin{proof}
    Observe that $\bbDelta$ is normal skeletal, and thus so is the product
    $\catCC \times \bbDelta$ in an evident way \cite[remark
    8.1.7]{Cisinski2006}. In particular, for $(c, [n]) \in \catCC \times
    \bbDelta$, maps $f \in (\catCC \times \bbDelta)_+$ with codomain $(c, [n])$
    are adequate pairs or morphisms $f = (f_\catCC, f_\bbDelta) \in \catCC_+
    \times \bbDelta_+$, and $f$ is not an isomorphism if and only if $f_\catCC$
    or $f_\bbDelta$ is not\footnote{This observation is called \emph{Leibniz's
    formula}\index{Leibniz's formula} in \cite[observation 4.2]{Riehl2014}.}.
    Thus it is easy to see that the boundary and boundary inclusion of $(c,
    [n])$ are given by
    \[
        \left( c \boxproduct \partial \Delta [n] \right)
        \coprod_{\partial c \boxproduct \partial \Delta [n]}
        \left( \partial c \boxproduct \Delta [n] \right)
        \xrightarrow{\sfB_c \lboxproduct \sfB_n}
        c \boxproduct \Delta [n] .
    \]
    We apply \cref{prop:cisinski:boundaries-generate-monomorphisms} to
    conclude.
\end{proof}

\begin{definition}
    [Trivial fibration]
    \label{def:trivial-fibration}
    We say that a morphism $f$ (in some category) is a \emph{trivial
    fibration}\index{trivial fibration} if it has the right lifting property
    with respect to all monomorphisms.
\end{definition}

In model category theory, a trivial fibration\index{trivial fibration} is
usually a fibration that is also a weak equivalence. In fact, the motivation
for the terminology of \cref{def:trivial-fibration} is that in the familiar
Quillen model structure on $\PshDelta$, both notions coincide. More generally,
they coincide in all Cisinski model category
(\cref{def:cisinski-model-structure}). In general however, there is no reason
for it to be the case, and both terminologies clash. To remedy this, we resort
to a classical alternative: fibrations that are weak equivalences will be
called \emph{acyclic fibrations}\index{acyclic fibration} throughout this
paper. Likewise, we will favor the name \emph{acyclic
cofibration}\index{acyclic cofibration} instead of trivial
cofibration\index{trivial cofibration}.

\begin{proposition}
    [{Generalization of \cite[proposition 2.3]{Joyal2007}}]
    \label{prop:joyal-tierney:2.3}
    Let $f : X \longrightarrow Y$ be a morphism in $\Sp\catCC$. The following
    are equivalent:
    \begin{enumerate}
        \item $f$ is a trivial fibration;
        \item the map $\langle \sfB_c \backslash f \rangle$ is a trivial
        fibration, for all $c \in \catCC$;
        \item the map $\langle u \backslash f \rangle$ is a acyclic fibtation,
        for all monomorphism $u$ in $\Psh\catCC$;
        \item the map $\langle f / \sfB_n \rangle$ is a trivial fibration, for
        all $n \in \bbNN$;
        \item the map $\langle f / v \rangle$ is a trivial fibration, for all
        monomorphism $v$ in $\PshDelta$.
    \end{enumerate}
\end{proposition}
\begin{proof}
    Simple consequence of \cref{prop:spc:generating-monomorphisms} and the
    adjunctions $u \lboxproduct - \dashv \langle u \backslash - \rangle$ and $-
    \lboxproduct v \dashv \langle - / v \rangle$ of
    \cref{sec:joyal-tierney-calculus}.
\end{proof}

\begin{lemma}
    \label{lemma:reedy-match}
    Let $\catDD$ be a normal skeletal category (particular, it is Reedy), and
    $\Psh\catEE$ be endowed with a model structure. Consider the Reedy model
    structure on $\Psh\catEE^{\catDD^\op}$.
    \begin{enumerate}
        \item For $X \in \Psh\catEE^{\catDD^\op}$ and $d \in \catDD$, the
        matching object of $X$ at $d$ is $\partial d \backslash X$.
        \item For $f : X \longrightarrow Y$ in $\Psh\catEE^{\catDD^\op}$ and $d
        \in \catDD$, the relative matching map of $f$ at $d$ is $\langle \sfB_d
        \backslash f \rangle$. In particular, $f$ is a Reedy fibration if and
        only if for all $d \in \catDD$, the map $\langle \sfB_d \backslash f
        \rangle$ is a fibration in $\Psh\catEE$.
        \item A map $f : X \longrightarrow Y$ in $\Psh\catEE^{\catDD^\op}$ is
        an acyclic Reedy fibration if and only for all $d \in \catDD$, the
        relative matching map $\langle \sfB_d \backslash f \rangle$ is an
        acyclic fibration in $\Psh\catEE$.
    \end{enumerate}
\end{lemma}
\begin{proof}
    For the fist point, observe that
    \[
        \partial d \backslash X
        = \left( \colim_{\substack{
            d' \rightarrow d \\ \text{in } \catDD_+, \text{not iso}
        }} d' \right) \backslash X
        \cong \lim_{\substack{
            d' \rightarrow d \\ \text{in } \catDD_+, \text{not iso}
        }} (d' \backslash X)
        \cong \lim_{\substack{
            d' \rightarrow d \\ \text{in } \catDD_+, \text{not iso}
        }} X_{d'} ,
    \]
    which is the matching object of $X$ at $d$. The second claim is by
    definition, and the third is \cite[theorem 15.3.15]{Hirschhorn2009}.
\end{proof}

For the the rest of this section, we assume that $\Psh\catCC$ is endowed with a
Cisinski model structure (\cref{def:cisinski-model-structure}).

\begin{definition}
    \label{def:vertical-horizontal-model-structures}
    Since $\catCC$ is normal skeletal, it is a Reedy category. Let
    $\Sp\catCC_v$, the \emph{vertical model structure}\index{vertical model
    structure} on $\Sp\catCC \cong \SpDelta^{\catCC^\op}$, be the Reedy
    structure induced by the Quillen model structure on $\PshDelta$. In this
    structure, a map $f : X \longrightarrow Y$ is a weak equivalence (also
    called \emph{column-wise weak equivalence}\index{column-wise weak
    equivalence|see {vertical model structure}}) if for all $c \in \catCC$, the
    map of simplicial sets $c \backslash f = f_c : X_c \longrightarrow Y_c$ is
    a weak equivalence. It is a fibration (also called \emph{vertical
    fibration}\index{vertical fibration|see {vertical model structure}}) if for
    all $c \in \catCC$, the relative matching map $\langle \sfB_c \backslash f
    \rangle$ is a Kan fibration, where $\sfB_c : \partial c \longrightarrow c$
    is the boundary inclusion of $c$. Fibrant spaces in $\Sp\catCC_v$ are also
    called \emph{vertically fibrant}\index{vertically fibrant|see {vertical
    model structure}}.

    Dually, let $\Sp\catCC_h$, the \emph{horizontal model
    structure}\index{horizontal model structure} on $\Sp\catCC \cong
    \Psh\catCC^{\bbDelta^\op}$, be the Reedy structure induced by the model
    structure on $\Psh\catCC$. The description of weak equivalence and
    fibrations transpose from the vertical model structure \emph{mutadis
    mutandis}\index{row-wise weak equivalence|see {vertical model
    structure}}\index{horizontal fibration|see {vertical model
    structure}}\index{horizontally fibrant|see {vertical model structure}}.
\end{definition}

\begin{proposition}
    \label{prop:vertical-horizontal-model-structures:cisinski}
    The model structures $\Sp\catCC_v$ and $\Sp\catCC_h$ are of Cisinsky type
    (\cref{def:cisinski-model-structure}).
\end{proposition}
\begin{proof}
    By \cite[theorem 15.6.27]{Hirschhorn2009}, both are cofibrantly generated.
    A map $f \in \Sp\catCC$ is a vertical (resp. horizontal) acyclic fibration
    if and only if for all $c \in \catCC$ (resp. $n \in \bbNN$), the matching
    map $\langle \sfB_c \backslash f \rangle$ (resp. $\langle f / \sfB_n
    \rangle$) is a acyclic fibration in $\PshDelta_{\mathrm{Quillen}}$ (resp.
    in the model structure on $\Psh\catCC$, which is assumed to be of Cisinsly
    type), i.e. a trivial fibration. By \cref{prop:joyal-tierney:2.3}, $f$ is a
    trivial fibration. Finally, $f$ is a vertical or a horizontal acyclic fibration if
    and only if it is a trivial fibration. Therefore, vertical and horizontal
    cofibrations are the monomorphisms.
\end{proof}

\begin{proposition}
    [{Generalization of \cite[proposition 2.5]{Joyal2007}}]
    \label{prop:joyal-tierney:2.5}
    Let $\catCC$ be a normal skeletal category, and $f : X \longrightarrow Y$
    be a morphism in $\Sp\catCC$. The following are equivalent:
    \begin{enumerate}
        \item $f$ is a vertical fibration, i.e. the map $\langle \sfB_c
        \backslash f \rangle$ is a Kan fibration, for $c \in \catCC$;
        \item the map $\langle u \backslash f \rangle$ is a Kan fibration, for
        all monomorphism $u \in \Psh\catCC$;
        \item the map $\langle f / \sfH_n^k \rangle$ is a trivial fibration, where
        $\sfH_n^k : \Lambda^k [n] \longrightarrow \Delta [n]$ is the $k$-th
        horn inclusion of $[n]$, for all $n \in \bbNN$ and $0 \leq k \leq n$;
        \item the map $\langle f / v \rangle$ is a trivial fibration, for all
        anodyne extension $v \in \PshDelta$.
    \end{enumerate}
\end{proposition}

\begin{definition}
    [{Homotopically constant space \cite[section 3]{Rezk2001a}}]
    \label{def:homotopically-constant}
    A space $X \in \Sp\catCC$ is \emph{homotopically
    constant}\index{homotopically constant} if for all map $f : [k]
    \longrightarrow [l]$ in $\bbDelta$, the structure map $X / f : X_{-, l}
    \longrightarrow X_{-, k}$ is a weak equivalence\footnote{A vertically
    fibrant and homotopically constant space is called a \emph{simplicial
    resolution}\index{simplicial resolution} in \cite[definition
    4.7]{Dugger2001a}.}.
\end{definition}

\begin{lemma}
    \label{lemma:homotopically-constant}
    Let $X \in \Sp\catCC$. The following are equivalent:
    \begin{enumerate}
        \item $X$ is homotopically constant;
        \item for all $k \in \bbNN$, writing $s : [k] \longrightarrow [0]$
        the terminal map in $\bbDelta$, the structure map $X / s : X_{-, 0}
        \longrightarrow X_{-, k}$ a weak equivalence;
        \item for all codegeneracy $s^i : [k] \longrightarrow [k-1]$ in
        $\bbDelta$, the structure map $X / s^i : X_{-, k-1} \longrightarrow X_{-,
        k}$ is a weak equivalence;
        \item for all $k \in \bbNN$ and all map $d : [0] \longrightarrow [k]$
        in $\bbDelta$, the structure map $X / d : X_{-, k} \longrightarrow
        X_{-, 0}$ is a weak equivalence;
        \item for all coface map $d^i : [k] \longrightarrow [k+1]$ in
        $\bbDelta$, the structure map $X / d^i : X_{-, k+1} \longrightarrow
        X_{-, k}$ is a weak equivalence.
    \end{enumerate}
\end{lemma}
\begin{proof}
    \begin{itemize}
        \item (1) $\implies$ (2) $\implies$ (3) and (1) $\implies$ (4)
        $\implies$ (5) are trivial.
        \item (2) $\implies$ (1) Take a map $f : [k] \longrightarrow [l]$ in
        $\bbDelta$. Clearly, $s = s f$, so $X / s = (X / f) (X / s)$. By
        3-for-2, $X / f$ is a weak equivalence.
        \item (3) $\implies$ (2) Note that the terminal map $s : [k]
        \longrightarrow [0]$ is a composite of codegeneracies $[k]
        \longrightarrow [k-1] \longrightarrow \cdots \longrightarrow [1]
        \longrightarrow [0]$.
        \item (4) $\implies$ (2) The terminal map $s : [k] \longrightarrow [0]$
        is a retraction of any map $d : [0] \longrightarrow [k]$, thus $X / s$
        is a section of $X / d$. By 3-for-2, $X / s$ is a weak equivalence.
        \item (5) $\implies$ (4) Note that all map $d : [0] \longrightarrow
        [k]$ is a composite of coface maps.
        \qedhere
    \end{itemize}
\end{proof}

\begin{proposition}
    [{Generalization of \cite[proposition 2.8]{Joyal2007}}]
    \label{prop:joyal-tierney:2.8}
    A vertically fibrant space $X \in \Sp\catCC$
    (\cref{def:vertical-horizontal-model-structures}) is homotopically
    constant.
\end{proposition}
\begin{proof}
    Take $d : [0] \longrightarrow [n]$ in $\bbDelta$. Since $d : \Delta [0]
    \longrightarrow \Delta [n]$ is a trivial cofibration in
    $\PshDelta_{\mathrm{Quillen}}$, by \cref{prop:joyal-tierney:2.5}, the map
    $X / d = \langle (X \rightarrow 1) / d \rangle$ is a trivial fibration.
    Apply \cref{lemma:homotopically-constant} to conclude.
\end{proof}

\begin{proposition}
    [{Generalization of \cite[proposition 2.9]{Joyal2007}}]
    \label{prop:joyal-tierney:2.9}
    A map $f : X \longrightarrow Y$ between vertically fibrant spaces is a weak
    equivalence in $\Sp\catCC_h$ if and only if $f_{-, 0} : X_{-, 0}
    \longrightarrow Y_{-, 0}$ is a weak equivalence.
\end{proposition}
\begin{proof}
    Let $n \in \bbNN$ and $s : [n] \longrightarrow [0]$ be the terminal map in
    $\bbDelta$. We have a commutative square
    \[
        \squarediagram
            {X_{-, 0}}{Y_{-, 0}}{X_{-, n}}{Y_{-, n}}
            {f_{-, 0}}{X / s}{Y / s}{f_{-, n}}
    \]
    where by \cref{prop:joyal-tierney:2.8}, the vertical morphisms are weak
    equivalences. The result follows by 3-for-2.
\end{proof}

\section{Opetopic algebras}
\label{sec:opetopic-algebras}

\subsection{Remiders}
\label{sec:opetopic-algebras:reminders}

In this section, we fix once and for all a parameter $n \geq 1$ and $k \geq 0$,
and write $\optPolyFun$ for $\optPolyFun^n$, $\Alg$ for $\Alg^k
(\optPolyFun^n)$, $\sfAA$ for $\sfAA_{k, n}$, etc. Recall the main construction
of \cite{HoThanh2019}, namely the reflective adjunction $h : \Psh{\bbOO_{\geq
n-k}} \adjunction \Alg : N$\index{$h$}\index{$N$} between the category of
truncated opetopic sets (i.e. trivial below dimension $n-k$), and $k$-colored
$n$-opetopic algebras. It exhibits $\Alg$ as the localization $\sfAA_{k,
n}^{-1} \PshO$, or equivalently, as the orthoginality class $\sfAA_{k,
n}^\perp$ (\cref{def:orthogonality}). Let $\bbLambda$ be the full subcategory
of $\Alg$ spanned by the image of $h$, i.e. the full subcategory of free
algebras. Taking $h_! : \Psh{\bbOO_{\geq n-k}} \longrightarrow
\PshLambda$\index{$h^*$} to be the left Kan extension of
$\bbOO_{\geq n-k} \stackrel{h}{\longrightarrow} \bbLambda \longrightarrow
\PshLambda$ along the Yoneda embedding, and $v : \PshLambda \longrightarrow
\Alg$\index{$v$}\index{$M$} to be the left Kan extension of the inclusion
$\bbLambda \longhookrightarrow \Alg$ along the Yoneda embedding, the reflection
$h$ factors as $h \cong v h_!$:
\begin{equation}
    \label{eq:triangle-adjunction}
    \begin{tikzcd} [column sep = small, row sep = large]
        &
        \Alg
        \arrow[dl, shift left = .4em, "N" below right, "\scriptstyle{\bot}" {sloped, above}]
        \arrow[dr, shift right = .4em, "M" below left, "\scriptstyle{\bot}" {sloped, above}] &
        \\
        \Psh{\bbOO_{\geq n-k}}
        \arrow[ur, shift left = .4em, "h" above left]
        \arrow[rr, shift left = .4em, "h_!" above] &
        &
        \PshLambda ,
        \arrow[ul, shift right = .4em, "v" above right]
        \arrow[ll, shift left = .4em, "h^*" below, "\scriptstyle{\bot}" {sloped, above}]
    \end{tikzcd}
\end{equation}

\begin{proposition}
    \label{prop:vM:reflective}
    The adjunction $v : \PshLambda \adjunction \Alg : M$ is reflective, i.e.
    $M$ is an embedding of categories.
\end{proposition}
\begin{proof}
    Recall that $N$ is an embedding, i.e. fully faithful and injective on
    objects. In particular, $M$ must be injective on objects as well. By
    \cite[theorem 4.10]{Weber2007}, $M$ is fully faithful.
\end{proof}

\begin{definition}
    [Spine]
    \label{def:spine-lambda}
    Let $\omega \in \bbOO_{n+1}$, let $\catSS_\omega \eqdef \bbOO_{n-k, n} / S
    [\omega]$\index{$\catSS_\omega$|see {spine}} be the category of elements of
    the spine $S [\omega] \in \Psh{\bbOO_{n-k, n}}$\index{$S [\omega]$|see
    {spine}} of $\omega$. Define the \emph{spine}\index{spine}\index{$S [h
    \omega]$|see {spine}} of $h \omega$ to be the colimit
    \[
        S [h \omega] \eqdef \colim \left(
            \catSS_\omega
            \longrightarrow \bbOO_{n-k,n}
            \stackrel{h}{\longrightarrow} \bbLambda
            \longrightarrow \PshLambda
        \right) .
    \]
    We have an inclusion $\sfS_{h \omega} : S [h \omega] \longrightarrow h
    \omega$\index{$\sfS_{h \omega}$|see {spine inclusion}} called the
    \emph{spine inclusion}\index{spine inclusion} of $h \omega$ and let $\sfSS
    \eqdef \left\{\sfS_{h \omega} : S [h \omega] \longrightarrow h \omega \mid
    \omega \in \bbOO_{n+1} \right\}$\index{$\sfSS$|see {spine inclusion}} be
    the set of spine inclusions of $\PshLambda$.
\end{definition}

\begin{lemma}
    \label{lemma:hsriek-spines}
    The functor $h_! : \PshO \longrightarrow \PshLambda$ maps $\sfSS_{n+1}
    \subseteq \PshO$ to $\sfSS \subseteq \PshLambda$, and morphisms in
    $\sfSS_{> n+1} \subseteq \PshO$ to $\sfSS$-local isomorphisms.
\end{lemma}
\begin{proof}
    Recall from \cite[lemma 4.5.1]{HoThanh2019} that $\dotH_! :
    \Psh{\bbOO_{n-k, n+2}} \longrightarrow \PshLambda$ maps $\sfSS_{n+1}$ to
    $\sfSS$ and $\sfSS_{n+2}$ to $\sfSS$-local isomorphisms. Thus, so does $h_!
    : \Sp{\bbOO_{\geq n-k}} \longrightarrow \SpLambda$. Take $\omega \in
    \bbOO_m$ with $m \geq n + 3$ and consider the following square in $\PshO$:
    \[
        \squarediagram
            {S [\tgt \omega]}{S [\omega]}
                {O [\tgt \omega]}{O [\omega] .}
            {i}{\sfS_{\tgt \omega}}{\sfS_\omega}{\tgt}
    \]
    By \cite[lemma 3.4.9]{HoThanh2019}, $i \in \Cell {\sfSS_{n+1}}$, thus a
    $\sfSS$-local isomorphism, and by induction, $\sfS_{\tgt \omega}$ is a
    $\sfSS$-local isomorphism. By \cite[corollary 3.4.10]{HoThanh2019} the
    bottom target embedding is a $\sfSS_{m-1, m}$-local isomorphism, thus a
    $\sfSS$-local isomorphism by induction. Consequently, $\sfS_{h \omega}$ is
    a $\sfSS$-local isomorphism.
\end{proof}

\begin{theorem}
    \label{th:algebraic-localization:lambda}
    The left adjoint $v : \PshLambda \longrightarrow \Alg$ is the
    Gabriel--Ulmer localization with respect to the set $\sfSS$ of spine
    inclusions. Equivalently, $M : \Alg \longrightarrow \PshLambda$ corestricts
    as an isomorphism $\Alg \longrightarrow \sfSS^\perp$.
\end{theorem}
\begin{proof}
    By \cite[theorem 4.10]{Weber2007}, a presheaf $X \in \PshDelta$ is in the
    essential image of $M$ if and only if $h^* X$ is in the essential image of
    $N$. In other words, $X$ is an algebra if and only if $\sfSS_{\geq n+1}
    \perp h^* X$, which under the adjunction $h_! \dashv h^*$, is equivalent to
    $\sfSS \perp X$.
\end{proof}

\subsection{Free opetopic algebras}
\label{sec:opetopic-algebras:free-opetopic-algebras}

Recall that $\bbLambda = \bbLambda_{1, n}$ is the full subcategory of $\Alg$
spanned by free algebras on opetopes in $\bbOO_{n-1, n+1}$. However, if $\phi
\in \bbOO_{n-1}$, then $h \phi = h \itree{\phi}$, and if $\psi \in \bbOO_n$,
then $h \psi = h \ytree{\psi}$. So we may consider $\bbLambda$ to be the full
subcategory of $\Alg$ spanned by free algebras over $(n+1)$-opetopes.

\begin{definition}
    [Diagram]
    \label{def:diagram}
    For $\omega, \omega' \in \bbOO_{n+1}$ and $f : h \omega
    \longrightarrow h \omega'$ in $\bbLambda$, a \emph{diagram}\index{diagram}
    \cite{HoThanh2019} of $f$ is the datum of a cospan
    \[
        \omega
        \stackrel{\src_{[p]}}{\longrightarrow}
        \xi
        \stackrel{\tgt}{\longleftarrow}
        \omega' ,
    \]
    with $\xi \in \bbOO_{n+2}$, such that $f = (h \tgt)^{-1} (h \src_{[p]})$.
    The situation is summarized as follows:
    \begin{equation}
        \label{eq:example-diagram}
        \frac{
            \begin{tikzcd} [ampersand replacement = \&]
                \&
                \xi \\
                \omega \arrow[ur, sloped, near end, "\src_{[p]}"] \&
                \omega' \arrow[u, "\tgt"]
            \end{tikzcd}
        }{
            \begin{tikzcd} [ampersand replacement = \&]
                h \omega \arrow[r, "f"] \& h \omega'
            \end{tikzcd}
        }
    \end{equation}
    The morphism $f$ is said to be \emph{diagrammatic}\index{diagrammatic
    morphism|see {diagram}} if it admits a diagram.
\end{definition}

\begin{lemma}
    [{Diagrammatic lemma, \cite{HoThanh2019}}]
    \label{lemma:diagrammatic-lemma}
    Let $f : h \omega \longrightarrow h \omega'$ be a morphism in $\bbLambda$.
    If $\omega$ is not degenerate, then $f$ is diagrammatic.
\end{lemma}

\begin{lemma}
    \label{lemma:h-subtrees}
    Let $\omega \in \bbOO_{n+1} = \tree \optPolyFun^{n-1}$ and $\psi \in
    \bbOO_n$. Then $h \omega_\psi = (N h \omega)_\psi$ is the set of
    sub-$\optPolyFun^{n-1}$-trees $\nu$ of $\omega$ such that $\tgt \nu =
    \psi$.
\end{lemma}
\begin{proof}
    By definition,
    \[
        h \omega_\psi
        = \sum_{\substack{\nu \in \bbOO_{n+1} \\ \tgt \nu = \psi}}
            \PshO (S [\nu], O [\omega]) ,
    \]
    and morphisms $S [\nu] \longrightarrow O [\omega]$ are precisely the
    $\optPolyFun^{n-1}$-tree embeddings of $\nu$ in $\omega$.
\end{proof}

\begin{lemma}
    \label{lemma:h-morphism-equality}
    Let $f_1, f_2 : h \omega \longrightarrow h \omega'$ be two morphism in
    $\bbLambda$, with $\omega, \omega' \in \bbOO_{n+1}$. Then $f_1 = f_2$ if
    and only if for all $[p] \in \omega^\nodesymbol$, $f_1 c_{\omega, [p]} =
    f_2 c_{\omega, [p]}$. In other words, $f_1 = f_2$ if and only if $f_1$ and
    $f_2$ agree on all one-node subtrees of $\omega$.
\end{lemma}
\begin{proof}
    Under the adjunction $\Psh{\bbOO_{\geq n-1}} \adjunction \Alg$, $f_i$
    corresponds to a morphism $\barF_i : O [\omega] \longrightarrow h \omega'$,
    and since $\sfSS_{\geq n+1} \perp h \omega'$, if is uniquely determined by
    its restriction $\barF_i : S [h \omega] \longrightarrow h \omega'$.
\end{proof}

\begin{notation}
    \label{notation:generators-h-omega}
    As a consequence of \cref{lemma:h-morphism-equality}, if $\omega \in
    \bbOO_{n+1}$ is not degenerate, the algebra $h \omega$ is freely generated
    by its one-node subtrees. We denote these generators by $c_{\omega, [p]}$,
    where $[p]$ ranges over $\omega^\nodesymbol$.
\end{notation}

\subsection{Quotients}
\label{sec:opetopic-algebras:quotient}

Let $X \in \Alg$. Recall that $X$ can be seen as an opetopic set $X \in \PshO$
such that $\sfAA_{k, n} \bot X$. We now present a notation that shows how
solutions of lifting problems $S [h \omega] \longrightarrow X$ can really be
understood as compositions of ``tree-shaped arities''.

For $\psi \in \bbOO_{n-1}$ and $x \in X_\psi$, write $\id_x \in
X_{\ytree{\psi}}$\index{$\id_x$} for the target of the solution of the lifting
problem
\[
    \diagramlines{}{}{--}{}
    \triangleULdiagram
        {S [\itree{\psi}] = O [\psi]}{X}{O [\itree{\psi}] .}
        {x}{}{\exists ! f}
\]
Explicitely, if $f$ is the solution, let $\id_x$ be the cell of
$X_{\ytree{\psi}}$ selected by
\[
    O [\ytree{\psi}]
    \stackrel{\tgt}{\longrightarrow} O [\itree{\psi}]
    \stackrel{f}{\longrightarrow} X .
\]

Let $\omega, \omega' \in \bbOO_n$, and $[p] \in \omega^\nodesymbol$ such that
$\src_{[p]} \omega = \tgt \omega'$, so that $\ytree{\omega} \graft_{[[p]]}
\ytree{\omega'}$ is a well defined $(n+1)$-opetope. For $x \in X_\omega$, $x'
\in X_{\omega'}$ such that $\src_{[p]} x = \tgt x'$, let $x \graft_{[p]}
x'$\index{$x \graft_{[p]} x'$} be the target of the solution of the following
lifting problem:
\[
    S [\ytree{\omega} \graft_{[[p]]} \ytree{\omega'}]
    \xrightarrow{[] \mapsto x,  [[p]] \mapsto x'} X .
\]
An iterated composition as on the right can be concisely written as on the
left:
\[
    x \biggraft_{[p_i]} y_i \eqdef
    ( \cdots (x \graft_{[p_1]} y_1) \graft_{[p_2]} y_2 \cdots)
        \graft_{[p_k]} y_k ,
\]
where $x \in X_\omega$, $\omega^\nodesymbol = \left\{ [p_1], \ldots, [p_k]
\right\}$, $y_i \in X_{\omega_i}$, and $\tgt \omega_i = \src_{[p_i]} \omega$.

Let $X \in \Alg$ be an opetopic algebra, $\omega \in \bbOO_n$, and $x, y \in
X_\omega$. We say that $x$ and $y$ are \emph{parallel}\index{parallel cells} if
the following two composites are equal:
\[
    \eqdiagram
        {\partial O [\omega]}{O [\omega]}{X .}
        {\sfB_\omega}{x}{y}
\]
In that case, let the \emph{quotient}\index{quotient} of $X$ by the equality $x
= y$ be the following coequalizer in $\Alg$:
\[
    \coeqdiagram
        {h \omega}{X}{X / \left\{ x = y \right\} .}
        {x}{y}{}
\]
For each $\omega \in \bbOO_n$, take a set $K_\omega \subseteq X_n \times X_n$
of pairs of parallel cells, and let $K \eqdef \sum_\omega K_\omega$. Then the
quotient of $X$ by $K$ is the following coequalizer in $\Alg$:
\[
    \coeqdiagram
        {\sum_{\omega \in \bbOO_n} \sum_{(x, y) \in K_\omega} h \omega}
            {X}{X / K ,}
        {p_1}{p_2}{}
\]
where $p_1 (x, y) \eqdef x$ and $p_2 (x, y) \eqdef y$.

\subsection{The folk model structure}

In this section, fix a parameter $n \geq 1$, and take $k = 1$. As usual, we
omit $n$ and $k$ from most notations, e.g. $\Alg = \Alg^1 (\optPolyFun^n)$. Let
$A \in \Alg$ be an algebra, $\psi \in \bbOO_n$, and $d : \partial h \psi
\longrightarrow A$, and consider the following pullback
\[
    \diagramarrows{c->}{}{}{c->}
    \pullbackdiagram
        {A_d}{A_{h \psi} = \PshLambda (h \psi, A)}
            {*}{\PshLambda (\partial h \psi, A) .}
        {}{}{\sfB_{h \psi}^*}{d}
\]
Explicitely, $A_d$ is the subset of $A_{h \psi}$ of all cell $a$ such that
$\src_{[q]} a = d (\src_{[q]} \psi)$ for all $[q] \in \psi^\nodesymbol$, and
$\tgt a = d (\tgt \psi)$.

\begin{example}
    Take $n = 2$, so that $\Alg$ is the category of planar colored operads. For
    $P \in \Alg$ and $\optInt{k} \in \bbOO_2$, a morphism $d : \partial h
    \optInt{k} \longrightarrow P$ is just a $(k+1)$-tuple of colors $(d_0,
    \ldots, d_{k-1}; d_k)$ of $P$, where $d_i = d \src_{[*^i]} \optInt{k}$ for
    $0 \leq i < k$, and $d_k = d \tgt \optInt{k}$. The set $P_d$ is simply $P
    (d_0, \ldots, d_{k-1}; d_k)$.
\end{example}

\begin{definition}
    [Internal isomorphism]
    \label{def:isomorphism}
    Let $f : A \longrightarrow B$ be a morphism of algebras. Let $\phi \in
    \bbOO_{n-1}$ and $a \in A_{h \ytree{\phi}}$. If $x = \src_{[]} a$ and $y =
    \tgt a$, we write $a : x \longrightarrow y$. We say that $a$ is an
    \emph{internal isomorphism}\index{internal isomorphism} (or just
    \emph{isomorphism}\index{isomorphism|see {internal isomorphism}}) if it is
    invertible, i.e. if there exists a cell $a^{-1} \in A_{h \ytree{\phi}}$
    such that
    \[
        a \graft_{[]} a^{-1} = \id_y,
        \qqquad
        a^{-1} \graft_{[]} a = \id_x .
    \]
    In this case, we also say that $x$ and $y$ are \emph{isomorphic}, and write
    $x \cong y$.
\end{definition}

\begin{definition}
    [Natural transformation]
    \label{def:natural-transformation}
    Let $f, g : A \longrightarrow B$ be two parallel morphisms of algebras. A
    \emph{natural transformation}\index{natural transformation} $\alpha : f
    \longrightarrow g$ is a collection of cells $\alpha_a : f(a)
    \longrightarrow g(a)$ (called \emph{components}\index{component|see
    {natural transformation}}, see \cref{def:isomorphism} for notation) such
    that for all $\psi \in \bbOO_n$ and $x \in A_{h \psi}$, the following
    relation holds:
    \[
        g(x) \biggraft_{[q]} \alpha_{\src_{[q]} a}
        =
        \alpha_{\tgt a} \graft_{[]} f(x) .
    \]
    Natural transformations can be composed in the obvious way, and the
    adequate exchange law holds. A \emph{natural isomorphism}\index{natural
    isomorphism} is an invertible natural transformation, or equivalently, on
    whose compoments are all isomorphisms.
\end{definition}

\begin{definition}
    [Algebraic equivalence]
    \label{def:algebraic-equivalence}
    Let $f : A \longrightarrow B$ be a morphism of algebras.
    \begin{enumerate}
        \item We say that $f$ is \emph{fully faithful}\index{fully faithful} if
        for all $\psi \in \bbOO_n$ and for all morphism $d : \partial h \psi
        \longrightarrow A$, the induced map $f_d : A_d \longrightarrow B_{fd}$
        is a bijection.
        \item We say that $f$ is \emph{essentially
        surjective}\index{essentially surjective} of for all $b \in B_{n-1}$,
        there exist an $a \in A_{n-1}$ such that $f(a) \cong b$.
        \item We say that $f$ is an \emph{algebraic
        equivalence}\index{algebraic equivalence} (or \emph{equivalence of
        algebras}\index{equivalence of algebras|see {algebraic equivalence}},
        or simply \emph{equivalence}) if it is fully faithful and essentially
        surjective. Clearly, if $n = 1$, we recover the notion of fully
        faithful functor between categories, whereas if $n = 2$, this matches
        the definition of operadic equivalence of \cite{Moerdijk2007}.
    \end{enumerate}
\end{definition}

\begin{proposition}
    \label{prop:algebraic-equivalence-weakly-invertible}
    A morphism $f : A \longrightarrow B$ of algebras is an equivalence if and
    only if it is invertible up to natural isomorphism, i.e. there exists $g :
    B \longrightarrow A$ and natural isomorphisms $\epsilon : gf
    \longrightarrow \id_A$ and $\eta : \id_B \longrightarrow fg$.
\end{proposition}
\begin{proof}
    This is similar to \cite[theorem IV.4.1]{MacLane1998}. Necessity is easy.
    Assume that $f$ is an equivalence. For $b \in B_{n-1}$, choose an $a = g
    (b) \in A_{n-1}$ and an isomorphism $\eta_b : b \longrightarrow f(g(b))$.
    Let $\psi \in \bbOO_n$ and $y \in B_{h \psi}$ be such that $\tgt y$ is in
    the image of $f$, say $\tgt y = f (a)$ for some $a \in A_{n-1}$. Then
    \[
        z \eqdef y \biggraft_{[q]} \eta_{\src_{[q]} y}^{-1}
    \]
    where $[q]$ ranges over $\psi^\nodesymbol$, has all its faces in the image
    of $f$, as $\src_{[q]} z = \src_{[]} \eta_{\src_{[q]} y}^{-1} = f (g
    (\src_{[q]} y))$. Since $f$ is fully faithful, there exist a unique $g (y)$
    in $A_{h \psi}$ such that $f (g (y)) = z$. This defines a morphism $g : B
    \longrightarrow A$, and a natural isomorphism $\eta : \id_B \longrightarrow
    fg$.

    Let $a \in A_{n-1}$, and consider the isomorphism $e \eqdef
    \eta_{f(a)}^{-1}$. Note that $e : fgf (a) \longrightarrow f (a)$, and since
    $f$ is fully faithful, there exists a unique $\epsilon_a : gf (a)
    \longrightarrow a$ such that $e = f(\epsilon_a)$. It is straightforward to
    check that the $\epsilon_a$s are the component of a natural isomorphism
    $\epsilon : gf \longrightarrow \id_A$.
\end{proof}

\begin{lemma}
    \label{lemma:acyclic-cofibration-weakly-retractible}
    Let $f : A \longrightarrow B$ be an equivalence that is injective on
    $(n-1)$-cells. Then $f$ admits a retract up to isomorphism, i.e. a weak
    inverse $g : B \longrightarrow A$ together with natural isomorphisms
    $\epsilon : gf \longrightarrow \id_A$ and $\eta : \id_B \longrightarrow fg$
    as in \cref{prop:algebraic-equivalence-weakly-invertible}, but where $gf =
    \id_A$ and $\epsilon$ is an identity (i.e. all its compoments are
    identities), and where $f = fgf$, and for $a \in A_{n-1}$, $\eta_{f(a)} =
    \id_{f(a)}$
\end{lemma}
\begin{proof}
    It suffices to amend the proof of
    \cref{prop:algebraic-equivalence-weakly-invertible} so that $g$, $\epsilon$
    and $\eta$ have the desired properties. For $a \in A_{n-1}$, since $f$ is
    injective on objects, we may choose $g(f(a))$ to be $a$, and $\eta_{f(a)}$
    to be $\id_{f(a)}$. It follows that, after extending $g$ to a morphism $B
    \longrightarrow A$, we have $gf = \id_A$. Further, for $a \in A_{n-1}$,
    $\epsilon_a : gf (a) \longrightarrow a$ is the only $n$-cell of $A$ such
    that $f (\epsilon_a) = \eta_{f(a)}^{-1} = \id_{f(a)}$, whence $\epsilon$ is
    the identity.
\end{proof}

\begin{definition}
    [Isofibration]
    \label{def:isofibration}
    A morphism $f : A \longrightarrow B$ of algebras is an
    \emph{isofibration}\index{isofibration} if for all $a \in A_{n-1}$, all
    isomorphism $g : f(a) \longrightarrow b$ in $B$, there exists an
    isomorphism $g' : a \longrightarrow a'$ in $A$ such that $f(g') = g$.
\end{definition}

\begin{theorem}
    [{Generalization of \cite[theorem 1.4]{Joyal2007}}]
    \label{th:folk-model-structure-algebras}
    The category $\Alg = \Alg^1 (\optPolyFun^n)$ of $1$-colored algebras admits
    a model structure where the weak equivalences are the algebraic
    equivalences (\cref{def:algebraic-equivalence}), the cofibrations are those
    morphism that are injective on $(n-1)$-cells, and where the fibrations are
    the isofibrations (\cref{def:isofibration}). We call this stucture the
    \emph{folk model structure}\index{folk model structure}, and denote it by
    $\Alg_{\mathrm{folk}}$\index{$\Alg_{\mathrm{folk}}$|see {folk model
    structure}}.

    Furthermore, acyclic fibrations are the algebraic equivalences that are
    surjective on $(n-1)$-cells, and every object is both fibrant and
    cofibrant.
\end{theorem}

The second claim can easily be checked once the model structure is established.
The rest of this section is dedicated to proving this. To that end, we verify
each of Quillen's axioms, cf \cite[definition 7.1.3]{Hirschhorn2009}.

\begin{proof}
    [Proof of \cref{th:folk-model-structure-algebras}, \condition{M1}: limit
    axiom] Since $\Alg$ is the category of models of a small projective sketch
    \cite{HoThanh2019}, it is locally presentable \cite[corollary
    1.52]{Adamek1980}, and therefore has all small limits and colimits.
\end{proof}

\begin{proof}
    [Proof of \cref{th:folk-model-structure-algebras}, \condition{M2}: 3-for-2
    axiom] This is clear.
\end{proof}

\begin{proof}
    [Proof of \cref{th:folk-model-structure-algebras}, \condition{M3}: retract
    axiom]
    Straightforward verifications.
\end{proof}

\begin{proof}
    [Proof of \cref{th:folk-model-structure-algebras}, \condition{M4}: lifting
    axiom]
    Consider a commutative square
    \[
        \diagramarrows{}{>->}{->>}{}
        \squarediagram
            {A}{B}{C}{D,}{f}{i}{p}{g}
    \]
    where $i$ is a cofibration and $p$ is a fibration. We show that a lift
    exists whenever $i$ or $p$ is a weak equivalence.
    \begin{enumerate}
        \item Assume that $i$ is a weak equivalence, and let $r : C
        \longrightarrow A$ be a weak retract of $i$ as in
        \cref{lemma:acyclic-cofibration-weakly-retractible}, together with the
        natural isomorphism $\eta : \id_C \longrightarrow ir$. Let $c \in
        C_{n-1}$, and consider $gir (c) = pfr (c) \in D_{n-1}$. Since $p$ is an
        isobibration, there exist an isomorphism $\beta_c : fr(c)
        \longrightarrow b$ in $B$ such that $p(\beta_c) = g(\eta_c^{-1}) : gir
        (c) = pfr (c) \longrightarrow g (c)$. In particular, $p (b) = g (c)$,
        and define $l (c) \eqdef b$. This defined a lift $l : C_{n-1}
        \longrightarrow B_{n-1}$.

        The construction above also provides a natural isomorphism $\beta : fr
        \longrightarrow l$. Without loss of generality, we choose $l$ and the
        components of $\beta$ such that for all $a \in A_{n-1}$, $li (a) =
        f(a)$, and $\beta_{i (a)} = \id_{f(a)}$. For $\psi \in \bbOO_n$ and $x
        \in C_{h \psi}$, let
        \[
            l (x) \eqdef \beta_{\tgt x} \graft_{[]} fr(x)
                \biggraft_{[q]} \beta_{\src_{[q]} x}^{-1} .
        \]
        It can easily be checked that $l$ is then a morphism of algebras $C
        \longrightarrow B$, and finally the desired lift.

        \item Assume that $p$ is a weak equivalence. In particular, on
        $(n-1)$-cells, $i$ is an injection, and $p$ a surjection, and so a lift
        $l : C_{n-1} \longrightarrow B_{n-1}$ can be found. We now extend $l$
        to a morphism of algebras $l : C \longrightarrow B$.

        Let $\psi \in \bbOO_n$, $x \in C_\psi$, and $d$ be the composite
        \[
            \partial h \psi
            \stackrel{\sfB_{h \psi}}{\longhookrightarrow} h \psi
            \stackrel{x}{\longrightarrow} C .
        \]
        Since $p$ is fully faithful, it induces a bijection
        \[
            p_d : C_{l d} \stackrel{\cong}{\longrightarrow} D_{g d} ,
        \]
        and letting $l (x) \eqdef p_d^{-1} g (x)$ extends $l$ to an algebra
        morphism which is the desired lift.
        \qedhere
    \end{enumerate}
\end{proof}

\begin{proof}
    [Proof of \cref{th:folk-model-structure-algebras}, \condition{M5}:
    factorization axiom]
    Let $f : A \longrightarrow B$ be a morphism of algebra.
    \begin{enumerate}
        \item We decompose $f$ as $f = pi$, where $i$ is an acyclic cofibation,
        and where $p$ is fibration. Define $C \in \PshLambda$ as follows. For
        $\phi \in \bbOO_{n-1}$, let
        \[
            C_{h \phi} \eqdef \left\{ (a, v, b) \mid
                a \in A_{h \phi}, b \in B_{h \phi}, v \in B_{\ytree{\phi}}
                \text{ isomorphism } f (a) \longrightarrow b
            \right\} .
        \]
        There is ab obvious projection $\proj : C_{n-1} \longrightarrow
        A_{n-1}$, and if $\psi \in \bbOO_n$ and $d : \partial h \psi
        \longrightarrow C_{n-1}$, let $C_d \eqdef A_{\proj d}$. At this stage,
        $C$ clearly extends as an algebra, essentially inheriting the same law
        as $A$.

        Let $i : A \longrightarrow C$ map an $(n-1)$-cell $a$ to $(a,
        \id_{f(a)}, f(a))$. This completely determines $i$, as indeed, for an
        $n$-cell $x \in A_{\geq n}$ we necessarily have $i (x) = x$. Clearly,
        $i$ is a fully faithful cofibration. It remains to show that it is
        essentially surjective. If $(a, v, b) \in C_{n-1}$, note that $\id_a$
        exhibits an isomorphism $i (a) = (a, \id_{f(a)}, f(a)) \longrightarrow
        (a, v, b)$ in $C$. Therefore, $i$ is an acyclic cofibration.

        Let $p : C \longrightarrow B$ be the obvious projection $(a, v, b)
        \longmapsto b$ on $(n-1)$-cells. Let $\psi \in \bbOO_n$ and $x \in C_{h
        \psi}$. Write $\tgt x = (a_{\tgt}, v_{\tgt}, b_{\tgt})$, for a source
        address $[q] \in \psi^\nodesymbol$, write $\src_{[q]} x = (a_{[q]},
        v_{[q]}, b_{[q]})$, and define
        \[
            p (x) \eqdef v_{\tgt} \graft_{[]} f(x) \biggraft_{[q]}
                v_{[q]}^{-1} .
        \]
        This defines a morphism $p : C \longrightarrow B$ which we claim to be
        a fibration. Indeed, if $(a, v, b) \in C_{n-1}$ and $w : b
        \longrightarrow b'$ is an isomorphism in $C$, we have an isomorphism
        $\id_a : (a, v, b) \longrightarrow (a, w \graft_{[]} v, b')$ in
        $C$, and
        \[
            p (\id_{f(a)}) = (w \graft_{[]} v) \graft_{[]} f(\id_a)
                \graft_{[]} v = w .
        \]

        Lastly, it is clear that $f = pi$, and thus $f$ decomposes as a acyclic
        cofibration followed by a fibration.

        \item We decompose $f$ as $f = pi$, where $i$ is an cofibation, and
        where $p$ is an acyclic fibration. Define $D \in \PshLambda$ as
        follows. On $(n-1)$-cells, it is given as on the left, and let $\barF$
        be defined as on the right:
        \[
            D_{n-1} \eqdef A_{n-1} + B_{n-1} ,
            \qqquad
            \barF \eqdef (f, \id_{B_{n-1}}) : D_{n-1} \longrightarrow B_{n-1} .
        \]
        Explicitely, $\barF$ maps $a \in A_{n-1}$ to $f(a)$, and $b \in
        B_{n-1}$ to $b$. For $\psi \in \bbOO_n$ and $d : \partial h \psi
        \longrightarrow D_{n-1}$, let the fiber $D_d$ be simply $B_{\barF d}$.
        At this stage, $D$ clearly extends as an algebra, essentially
        inheriting the same law as $B$.

        Let $i : A \longrightarrow D$ map an $(n-1)$-cell $a$ to $a$, and an
        $n$-cell $x$ to $f(x)$. Obviously, this is a cofibration. Let $p : D
        \longrightarrow A$ map an $(n-1)$-cell $d$ to $\barF (d)$, and an
        $n$-cell $x$ to $x$. This can easily be seen to be an acyclic
        fibration. Lastly, $f = pi$, so that $f$ can be decomposed into a
        cofibration followed by an acyclic fibration.
        \qedhere
    \end{enumerate}
\end{proof}

\begin{definition}
    [Rezk interval]
    \label{def:rezk-interval-algebra}
    Let $\phi \in \bbOO_{n-1}$. The \emph{Rezk interval of shape
    $\phi$}\index{Rezk interval}\index{algebraic Rezk interval} is the algebra
    $\frakJJ_\phi \in \Alg$\index{$\frakJJ_\phi$|see {algebraic Rezk interval}}
    generated by one invertible operation $j_\phi$ of shape $\ytree{\phi}$.
    Explicitely, $\frakJJ_\phi$ has two $(n-1)$-cells $0_\phi, 1_\phi \in
    (\frakJJ_\phi)_\phi$, and four $n$-cells $j_\phi, j_\phi^{-1}, \id_0, \id_1
    \in (\frakJJ_\phi)_{\ytree{\phi}}$, satisfying the following equalities
    \[
        \src_{[]} j_\phi = \tgt j_\phi^{-1} = 0_\phi ,
        \quad
        \tgt j_\phi = \src_{[]} j_\phi^{-1} = 1_\phi ,
        \quad
        j_\phi \graft_{[]} j_\phi^{-1} = \id_1 ,
        \quad
        j_\phi^{-1} \graft_{[]} j_\phi = \id_0 .
    \]

    Note that up to isomorphism, there is a unique endpoint inclusion
    $\sfJ_\phi : h \phi \longrightarrow \frakJJ_\phi$, and let $\sfEE_\frakJJ
    \eqdef \left\{ \sfJ_\phi \mid \phi \in \bbOO_{n-1} \right\}$.

    If $X = (X_\psi \mid \psi \in \bbOO_{n-1})$ is a set over $\bbOO_{n-1}$,
    let $\frakJJ_X \eqdef \sum_\psi \sum_{x \in X_\psi} \frakJJ_\psi$. In the
    $x \in X_\psi$ component, we write $j_x$ instead of $j_\psi$, and
    similarly, $j_x^{-1}$, $0_x$, and $1_x$. The \emph{Rezk
    interval}\index{Rezk interval}\index{algebraic Rezk interval} (without any
    mention of shape) is the sum $\frakJJ = \frakJJ_{\bbOO_{n-1}} \eqdef
    \sum_\psi \frakJJ_\psi$ in $\Alg$.
\end{definition}

\begin{theorem}
    \label{th:folk-model-structure-algebras:cofibrantly-generated}
    The model structure $\Alg_{\mathrm{folk}}$ is cofibrantly generated, and
    $\sfEE_\frakJJ$ can be taken as a set of generating acyclic cofibrations.
\end{theorem}
\begin{proof}
    Clearly, a morphism $f$ is an isofibration if and only if $\sfEE_\frakJJ
    \pitchfork f$. It is straightforward to check that $f$ is an acyclic
    fibration of and only if $\sfJJ \pitchfork f$, where $\sfJJ$ is
    \[
        \left\{
            \emptyset \longhookrightarrow h \phi \mid \phi \in \bbOO_{n-1}
        \right\}
        \cup v \sfBB \cup
        \left\{
            h \psi \coprod_{\partial h \psi} h \psi \longrightarrow h \psi
            \mid \psi \in \bbOO_n
        \right\} ,
    \]
    (see \cref{def:boundary-lambda} for the definition of $\sfBB$).
\end{proof}

\subsection{Homotopy equivalences}

\begin{definition}
    [Rezk cylinder of an algebra]
    \label{def:rezk-cylinder-algebra}
    Let $A \in \Alg$, and consider the following pushout in $\Alg$:
    \[
        \pushoutdiagram
            {h A_{n-1} + h A_{n-1}}{A + A}{\frakJJ_{A_{n-1}}}{A' .}
            {}{}{(i_0, i_1)}{}
    \]
    For $a$ a cell of $A$, write $a^{(e)} \eqdef i_e (a)$, for $e = 0, 1$.
    Explicitely, $A'$ is generated by two copies of $A$, and for each $a \in
    A_\psi$, $\psi \in \bbOO_{n-1}$, an additional cell $j_a$ of shape
    $\ytree{\psi}$ with $\src_{[]} j_a = a^{(0)}$ and $\tgt j_a = a^{(1)}$,
    which can informally be written $j_a : a^{(0)} \longrightarrow a^{(1)}$.

    The \emph{Rezk cylinder}\index{Rezk cylinder}\index{algebraic Rezk
    cylinder} of $A$ is the quotient $\frakJJ A \eqdef A' / K$\index{$\frakJJ
    A$|see {algebraig Rezk cylinder}}, where
    \begin{equation}
        \label{eq:rezk-cylinder:relations}
        K \eqdef \left\{
            a^{(1)} \biggraft_{[p_i]} j_{\src_{[p_i]} a}
                = j_{\tgt a} \graft_{[]} a^{(0)}
            \mid a \in A_\omega, \omega \in \bbOO_n,
                \omega^\nodesymbol = \left\{ [p_1], \ldots \right\}
        \right\} .
    \end{equation}
    See \cref{sec:opetopic-algebras:quotient} for the $\graft$ notation. It is
    a cylinder object in the sense of \cref{def:cylinder}, i.e. we have a
    canonical factorization of the codiagonal map
    \[
        \diagramarrows{}{>->}{<-}{}
        \triangleDdiagram
            {A + A}{\frakJJ A}{A .}
            {\nabla}{(i_0, i_1)}{\nabla}
    \]
    Explicitly, $\nabla : \frakJJ A \longrightarrow A$ maps $a^{(e)}$ to $a$,
    for a cell $a \in A$, and $j_a$ to $\id_a$, if $a$ is $(n-1)$-dimensional.
    Note that $(i_0, i_1) : A + A \longrightarrow \frakJJ A$ is a monomorphism,
    since the relation $K$ of \cref{eq:rezk-cylinder:relations} does not
    identify cells of $e A$, for $e = 0, 1$. We write
    $A^{(e)}$\index{$A^{(e)}$|see {algebraic Rezk cylinder}} the image of
    $i_e$, and by abuse of notation, $i_e : A^{(e)} \longhookrightarrow \frakJJ
    A$ the obvious inclusion.
\end{definition}

\begin{example}
    Let $\psi \in \bbOO_n$. Then $\frakJJ h \psi$ is generated by
    \begin{enumerate}
        \item $\psi^{(0)}, \psi^{(1)} \in (\frakJJ h \psi)_{\psi}$, i.e. two
        copies of $h \psi$,
        \item for each $[p] \in \psi^\nodesymbol$, and writing $\phi \eqdef
        \src_{[p]} \psi$, two cells $j_\phi, j_\phi^{-1} \in (\frakII h
        \psi)_{\ytree{\phi}}$, i.e. one copy of $\frakJJ_\phi$,
        \item $j_{\tgt \psi}$, $j_{\tgt \psi}$, i.e. one copy of $\frakJJ_{\tgt
        \psi}$,
    \end{enumerate}
    subject to the relation \cref{eq:rezk-cylinder:relations}. Likewise, for
    $\omega \in \bbOO_{n+1}$, the cylinder $\frakII h \omega$ is generated by
    two copies of $h \omega$, and one copy $\frakJJ_{\edg_{[q]} \omega}$ for
    all edge address $[q]$ of $\omega$.
\end{example}

\begin{remark}
    The cylinder $\frakJJ A$ of an algebra $A$ can be thought of the
    \emph{Boardman--Vogt tensor product} \cite{Weiss2011,Boardman1973,May1972}
    $\frakJJ \otimes_{\mathrm{BV}} A$. Unfortunately, in the planar case, a
    general such construction is not possible, as there is no way to ``shuffle
    the inputs''. Since this is the only obstruction, one could still define a
    $A \otimes_{\mathrm{BV}} B$ as long as either $A$ of $B$ only have unary
    (i.e. endotopic) operations.
\end{remark}

\begin{lemma}
    \label{lemma:j-homotopy-natural-isomorphism}
    Let $f, g : A \longrightarrow B$ be two parallel morphisms of algebras. The
    following are equivalent:
    \begin{enumerate}
        \item $f \simeq g$ (\cref{def:homotopy});
        \item there is an elementary $\frakJJ$-homotopy from $f$ to $g$;
        \item there exist a natural isomorphism $f \longrightarrow g$.
    \end{enumerate}
\end{lemma}
\begin{proof}
    \begin{itemize}
        \item (1) $\implies$ (3). A homotopy $H$ from $f$ to $g$ as in
        \cref{def:homotopy} induces a natural isomorphism $\hatHH$ with
        components $\hatHH_a = H (j_a)$ (see \cref{def:rezk-cylinder-algebra}
        for notations), for $a \in A_{n-1}$.
        \item (3) $\implies$ (2). A natural isomorphism $\alpha : f
        \longrightarrow g$ induces a homotopy $\chalpha : \frakJJ A
        \longrightarrow B$ from $f$ to $g$, where for $a \in A$, $\chalpha
        (a^{(0)}) \eqdef f(a)$ (see \cref{def:rezk-cylinder-algebra} for
        notations), $\chalpha (a^{(1)}) \eqdef g(a)$, $\chalpha (j_a) =
        \alpha_a$, and $\chalpha (j_a^{-1}) \eqdef \alpha_a^{-1}$.
        \item (2) $\implies$ (1). By definition.
        \qedhere
    \end{itemize}
\end{proof}

\begin{proposition}
    \label{prop:folk-weq-j-homotopy-eq}
    A morphism $f : A \longrightarrow B$ is a weak equivalence (for the folk
    structure of \cref{th:folk-model-structure-algebras}) if and only if it is
    an isomorphism in $\Ho\Alg$ (\cref{def:homotopy}). Therefore, the category
    $\Ho\Alg$ is the localization of $\Alg$ at the algebraic equivalences.
\end{proposition}
\begin{proof}
    Follows from
    \cref{prop:algebraic-equivalence-weakly-invertible,lemma:j-homotopy-natural-isomorphism}.
\end{proof}

\section{The homotopy theory of \texorpdfstring{$\infty$}{infinity}-opetopic algebra}
\label{sec:infty-algebra}

We fix a parameter $n \geq 1$, and assume that $k = 1$. In \cite{HoThanh2019},
we built a reflective adjunction
\[
    h : \Psh{\bbOO_{\geq n-1}} \adjunction \Alg : N .
\]
We consider $N$ to be the inclusion of a full subcategory $\Alg
\longhookrightarrow \Psh{\bbOO_{\geq n-1}}$, and omit it if unambiguous.

\subsection{A skeletal structure on \texorpdfstring{$\bbLambda$}{Lambda}}

In this section, we endow $\bbLambda = \bbLambda_{1, n}$ with the structure of
a skeletal category (cf. \cref{sec:cisinski-homotopy-theory}). We first need to
asign a notion of \emph{degree}\index{degree}\index{$\deg$|see {degree}} to any
objects $\lambda \in \bbLambda$, denoted by $\deg \lambda \in \bbNN$. Next, we
need to specify two casses of morphisms $\bbLambda_+$ and $\bbLambda_-$
satisfying axioms \condition{Sq0} to \condition{Sq3}.

\begin{definition}
    \label{def:algebraic-dimension}
    Recall that the objects of $\bbLambda$ are the free algebras on
    $(n+1)$-opetopes. For $\omega \in \bbOO_{n+1}$, let $\deg h \omega \eqdef
    \hash \omega^\nodesymbol$.
\end{definition}

\begin{definition}
    \label{def:lambda-+-}
    Let $f : h \omega \longrightarrow h \omega'$ be a morphism in $\bbLambda$,
    with $\omega, \omega' \in \bbOO_{n+1}$. In particular, it induces a set map
    between $(n-1)$-cells: $f_{n-1} : h \omega_{n-1} \longrightarrow h
    \omega'_{n-1}$. Let $\bbLambda_+$ (resp. $\bbLambda_-$) be the wide
    subcategory spanned by those morphisms $f$ such that $f_{n-1}$ is an
    monomorphism (resp. epimorphism).
\end{definition}

The rest of this section is dedicated to prove the following result.

\begin{theorem}
    \label{th:lambda-skeletal}
    With the data above, $\bbLambda$ is a skeletal category
    (\cref{def:skeletal-category}).
\end{theorem}

\begin{proof}
    [Proof of \cref{th:lambda-skeletal}, axiom \condition{Sq0}] Since
    $\bbLambda$ is rigid (i.e. has no isomorphisms beside the identities), this
    axiom holds trivially.
\end{proof}

\begin{lemma}
    \label{lemma:lambda+:injective}
    Let $f : h \omega \longrightarrow h \omega'$ be a morphism in $\bbLambda$,
    with $\omega, \omega' \in \bbOO_{n+1}$. The following are equivalent:
    \begin{enumerate}
        \item $f \in \bbLambda_+$ (resp. $\bbLambda_-$);
        \item the set map between $n$-cells $f_n : h \omega_n \longrightarrow h
        \omega'_n$ is an monomorphism (resp. epimorphism);
        \item $f$ is a monomorphism in $\Psh{\bbOO_{\geq n-1}}$ (resp.
        epimorphism).
    \end{enumerate}
\end{lemma}
\begin{proof}
    Recall that $h : \Psh{\bbOO_{\geq n-1}} \longrightarrow \Alg$ is the
    localization with respect to the set $\sfS_{\geq n+1}$. Therefore,
    $f_{\geq n}$ is an monomorphism (resp. epimorphism) if and only if $f_n$
    is. It remains to show that $f_{n-1}$ is a monomorphism (resp. epimorphism)
    if and only if $f_n$ is. But this is a direct consequence of
    \cref{lemma:h-subtrees}, stating that $h \omega_n$ is the set of
    sub-$\optPolyFun^{n-1}$-trees of $\omega$.
\end{proof}

\begin{corollary}
    \label{coroll:lambda+:iso}
    Let $f : h \omega \longrightarrow h \omega'$ be a morphism in $\bbLambda$,
    with $\omega, \omega' \in \bbOO_{n+1}$. The following are equivalent:
    \begin{enumerate}
        \item $f \in \bbLambda_+ \cap \bbLambda_-$, i.e. $f_{n-1}$ is an
        isomorphism;
        \item $f_n$ is an isomorphism;
        \item $f$ is an isomorphism;
    \end{enumerate}
\end{corollary}
\begin{proof}
    Recall that in any presheaf category, the epi-monos are exactly the
    isomorphisms.
\end{proof}

\begin{proof}
    [Proof of \cref{th:lambda-skeletal}, axiom \condition{Sq1}] Let $f : h
    \omega \longrightarrow h \omega'$ be a morphism in $\bbLambda$ that is not
    an isomorphism, with $\omega, \omega' \in \bbOO_{n+1}$. If $f \in
    \bbLambda_+$, then by \cref{lemma:lambda+:injective,coroll:lambda+:iso},
    $f_n$ is a monomorphism that is not an isomirphism, thus $\hash h \omega_n
    < \hash h \omega'_n$, i.e. the number of subtrees of $\omega$ is strictly
    less than that of $\omega'$. In particular, t he number of nodes of $\omega$
    is strictly less than that of $\omega'$, i.e. $\deg h \omega < \deg h
    \omega'$.

    The case where $f \in \bbLambda_-$ is treated similarly.
\end{proof}

\begin{proof}
    [Proof of \cref{th:lambda-skeletal}, axiom \condition{Sq2}] Let $f : h
    \omega \longrightarrow h \omega'$ be a morphism in $\bbLambda$, with
    $\omega, \omega' \in \bbOO_{n+1}$. It maps the maximal subtree $\omega
    \subseteq \omega$ to a subtree of $\omega'$, say $\nu \subseteq \omega'$.
    Then the following is the desired factorization:
    \[
        \diagramarrows{}{}{<-c}{}
        \triangleDdiagram
            {h \omega}{h \omega'}{h \nu .}
            {f}{f}{}
    \]
    Uniqueness comes from the fact that $\bbLambda_+$ only contains inclusions
    (of $\Psh{\bbOO_{\geq n-1}}$).
\end{proof}

\begin{proof}
    [Proof of \cref{th:lambda-skeletal}, axiom \condition{Sq3}] Let $f : h
    \omega \longrightarrow h \omega'$ be a morphism in $\bbLambda_-$. We define
    a section $g$ of $f$. By adjointness, this is equivalent to specifying a
    map $\barG : O [\omega'] \longrightarrow h \omega$ in $\Psh{\bbOO_{\geq
    n-1}}$. Since $\sfSS_{\geq n +1} \perp h \omega$, it is enough to define
    $\barG$ on the spine $S [\omega']$ of $\omega'$. For $[p] \in
    (\omega')^\nodesymbol$, let
    \[
        \barG (\src_{[p]} \omega') \eqdef \src_{[q]} \omega,
        \qqquad
        [q] \eqdef \min \left\{
            [r] \in \omega^\nodesymbol \mid
            f (\src_{[r]} \omega) = \src_{[p]} \omega'
        \right\} .
    \]
    In other words, $g$ maps the node $\src_{[p]} \omega'$ to the
    lexicographically minimal node in the fiber $f^{-1} (\src_{[p]} \omega')$.
    The composite
    \[
        S [\omega']
        \stackrel{\barG}{\longrightarrow} h \omega
        \stackrel{f}{\longrightarrow} h \omega'
    \]
    maps a node $\src_{[p]} \omega'$ to $\src_{[p]} \omega'$, thus $g$ is a
    section of $f$.

    Let now $f_1, f_2 : h \omega \longrightarrow h \omega'$ be a morphism in
    $\bbLambda_-$ having the same sections. In particular, for $g_i$ the
    section of $f_i$ constructed as above, where $i = 1, 2$, we have $f_1 g_2 =
    \id_{h \omega'}$. Thus for $[p] \in (\omega')^\nodesymbol$, we have $g_2
    (\src_{[p]} \omega') \preceq g_1 (\src_{[p]} \omega')$, meaning that the
    node $g_2 (\src_{[p]} \omega')$ is lexicographically inferior to (or
    ``below'') $g_1 (\src_{[p]} \omega')$ in $\omega$. Conversely, since $f_2
    g_1 = \id_{h \omega'}$, we have $g_1 (\src_{[p]} \omega') \preceq g_2
    (\src_{[p]} \omega')$, and finally, $g_1 (\src_{[p]} \omega') = g_2
    (\src_{[p]} \omega')$. By \cref{lemma:h-morphism-equality}, $g_1 = g_2$, so
    clearly, $f_1 = f_2$.
\end{proof}

\subsection{Anodyne extensions}
\label{sec:infty-alg:model-structure:anodyne-extensions}

\begin{definition}
    [{Boundary, \cite[8.1.30]{Cisinski2006}}]
    \label{def:boundary-lambda}
    Let $\lambda \in \bbLambda$. The \emph{boundary}\index{boundary} $\partial
    \lambda \in \PshLambda$\index{$\partial \lambda$|see {boundary}} of
    $\lambda$ (see \cref{def:boundary}) is the colimit in the middle, which can
    equivalently be defined as the union (in $\PshLambda$) on the right
    \[
        \partial \lambda
        \eqdef \colim_{\substack{
            f : \lambda' \rightarrow \lambda \text{ in } \bbLambda_+ , \\
            f \text{ not an iso.}
        }} \lambda'
        =
        \bigcup_{\substack{
            f : \lambda' \rightarrow \lambda \text{ in } \bbLambda_+ , \\
            \deg \lambda' = \deg \lambda - 1
        }} \im f .
    \]

    Explicitely, if $\lambda = h \omega$, for $\omega \in \bbOO_{n+1}$, then
    $\partial h \omega$ is the subpresheaf of $h \omega$ spanned by its
    $(n-1)$-cells.

    We write $\sfB_\lambda : \partial \lambda \longrightarrow
    \lambda$\index{$\sfB_\lambda$|see {boundary inclusion}} the \emph{boundary
    inclusion}\index{boundary inclusion} of $\lambda$, and $\sfBB \eqdef
    \left\{\sfB_\lambda \mid \lambda \in \bbLambda \right\}$\index{$\sfBB$|see
    {boundary inclusion}} the set of boundary inclusions.
\end{definition}

\begin{proposition}
    The class of monomorphisms if $\PshLambda$ is exactly the class of
    $\sfBB$-cell complexes $\Cell \sfBB$. Thus in the terminology of
    \cite[definition 1.2.26]{Cisinski2006}, $\sfBB$ is a \emph{cellular
    model}\index{cellular model} of $\PshLambda$.
\end{proposition}
\begin{proof}
    Since $\bbLambda$ is a normal skeletal category,
    \cref{prop:cisinski:boundaries-generate-monomorphisms} applies.
\end{proof}

\begin{definition}
    [Elementary face]
    \label{def:face}
    Let $\lambda \in \bbLambda$.
    \begin{enumerate}
        \item A \emph{elementary face}\index{elementary face} of $\lambda$ is a
        morphism $f : \lambda' \longrightarrow \lambda$ in $\bbLambda_+$, where
        $\deg \lambda' = \deg \lambda - 1$.
        \item Let $f : \lambda' \longrightarrow \lambda$ be an elementary face
        of $\lambda$, and write $\lambda = h \omega$ and $\lambda' = h
        \omega'$, with $\omega, \omega' \in \bbOO_{n+1}$. The face $f$ is
        \emph{inner}\index{inner face} if $f_{n-1}$ exhibits a bijection
        between the leaves of $\omega'$ (seen as $(n-1)$-cells of $h \omega'$)
        and the leaves of $\omega$, and if it maps the root edge $\edg_{[]}
        \omega$ to $\edg_{[]} \omega'$. In other words, $f_{n-1}$ is a
        bijection $h \partial O [\tgt \omega] \longrightarrow h \partial O
        [\tgt \omega']$.
    \end{enumerate}
\end{definition}

\begin{remark}
    \label{def:inner-face}
    If $f : h \omega' \longrightarrow h \omega$ is an inner face of $h \omega$,
    then a counting argument on the number of nodes of $\omega'$ and $\omega$,
    there exists a unique $[p] \in (\omega')^\nodesymbol$ such that $f
    c_{\omega', [p]}$ (see \cref{notation:generators-h-omega}) is a subtree of
    $\omega$ with two nodes. If $[p] \neq [p'] \in (\omega')^\nodesymbol$, then
    $f c_{\omega', [p']}$ is a generator of $h \omega$, i.e. a one-node
    subtree. Thus $f$ exhibits a subtree $\nu \eqdef f c_{\omega, [p]}$ of
    $\omega$ with two nodes, or equivalently, an inner edge.
\end{remark}

This remark motivates the following terminology:

\begin{definition}
    [Inner horn]
    \label{def:inner-horn}
    Let $\omega \in \bbOO_{n+1}$.
    \begin{enumerate}
        \item If $\omega$ is degenerate (resp. an endotope), say $\omega =
        \itree{\phi}$ for some $\phi \in \bbOO_{n+1}$ (resp. $\ytree{\phi}$ for
        some $\phi \in \bbOO_n$), then the inner horn of $h \omega$ is simply
        $\Lambda h \omega \eqdef h \phi$. Write $\sfH : \Lambda h \omega
        \longrightarrow h \omega$ for the horn inclusion of $h \omega$.

        \item Otherwise, for $f$ an inner face of $\lambda$, define $\Lambda^f
        \lambda$, the \emph{inner horn}\index{inner horn} of $\lambda$ at $f$,
        as the colimit on the left, or equivalently, as the union on the right
        \[
            \Lambda^f \lambda
            \eqdef \colim_{\substack{
                g : \lambda' \rightarrow \lambda \text{ elem. face} \\
                g \neq f
            }} \lambda'
            =
            \bigcup_{\substack{
                g : \lambda' \rightarrow \lambda \text{ elem. face} \\
                g \neq f
            }} \im g .
        \]
        Let $\sfH_\lambda^f : \Lambda^f \lambda \longrightarrow
        \lambda$\index{$\sfH_\lambda^f$|see {inner horn inclusion}} be the
        \emph{inner horn inclusion}\index{inner horn inclusion} $\lambda$ at
        $f$.
    \end{enumerate}
    Let $\sfHH_{\mathrm{inner}}$ be the set of inner horn inclusions.
\end{definition}

\begin{definition}
    [Inner anodyne extension]
    \label{def:inner-anodyne-extension}
    The class
    $\mathsf{An}_{\mathrm{inner}}$\index{$\mathsf{An}_{\mathrm{inner}}$|see
    {inner anodyne extension}} of \emph{inner anodyne extensions}\index{inner
    anodyne extension} if defined as $\mathsf{An}_{\mathrm{inner}} \eqdef
    {}^\pitchfork \left(\sfHH_{\mathrm{inner}}^\pitchfork \right)$. An
    \emph{inner fibration}\index{inner fibration} is a morphism $f \in
    \PshLambda$ such that $\sfHH_{\mathrm{inner}} \pitchfork f$, or
    equivalently, such that $\mathsf{An}_{\mathrm{inner}} \pitchfork f$.
\end{definition}

\begin{definition}
    \label{def:super-horn}
    Generalizing \cref{def:inner-horn}, if $\omega \in \bbOO_{n+1}$, and $I$ is
    a set of (not necessarily inner) faces of $h \omega$, we define
    \[
        \Lambda^I \lambda
        \eqdef \colim_{\substack{
            g : \lambda' \rightarrow \lambda \text{ elem. face} \\
            g \notin I
        }} \lambda'
        =
        \bigcup_{\substack{
            g : \lambda' \rightarrow \lambda \text{ elem. face} \\
            g \notin I
        }} \im g ,
    \]\index{$\Lambda^I \lambda$}\index{$\sfH_{h \omega}^I$}
    and write $\sfH_{h \omega}^I : \Lambda^I h \omega \longrightarrow h \omega$
    the canonical inclusion.
\end{definition}

\begin{lemma}
    \label{lemma:relative-super-horns-inclusions}
    Let $\omega \in \bbOO_{n+1}$ have $d \geq 2$ nodes, and $\emptyset \neq I
    \subseteq J$ be two non empty sets of inner faces of $h \omega$. Then the
    inclusion $\sfH_{h \omega}^I$ factors as
    \[
        \Lambda^J h \omega
        \stackrel{u}{\longhookrightarrow} \Lambda^I h \omega
        \stackrel{\sfH_{h \omega}^J}{\longhookrightarrow} h \omega ,
    \]
    and $u$ is a cell complex of inner horn inclusions of opetopes with at most
    $d - 1$ nodes.
\end{lemma}
\begin{proof}
    We proceed by induction on $m \eqdef \hash (J - I)$. If $m = 0$, then $u$
    is an identity, and the result holds trivially. Assume the lemma holds up
    to $m-1$. Take $f \in J - I$, and let $J' \eqdef J - \{ f \}$. The
    inclusion $u$ decomposes as
    \[
        \Lambda^J h \omega
        \stackrel{v}{\longhookrightarrow} \Lambda^{J'} h \omega
        \stackrel{w}{\longhookrightarrow} \Lambda^I h \omega
    \]
    and by induction, $w$ is a cell complex of inner horn inclusions of
    opetopes with at most $d - 1$ nodes. It remains to show that $v$ is too.

    The inner face $f$ of $h \omega$ exhibits a subtree $\nu \subseteq \omega$
    with two nodes, or equivalently, an inner edge of $\omega$, say at address
    $[e]$. Let $\omega_{/[e]}$ be $\omega$ where this inner edge has been
    contracted. Explicitely, decomposing $\omega$ on the left so as to exhibits
    the subtree $\nu$, the opetope $\omega_{/[e]}$ is defined on the right:
    \[
        \omega = \alpha \graft_{[p]} \nu \biggraft_{[l_i]} \beta_i ,
        \qqquad
        \omega_{/[e]} \eqdef \alpha \graft_{[p]} \ytree{\tgt \nu}
            \biggraft_{[\readdress_\nu [l_i]]} \beta_i ,
    \]
    where $[l_i]$ ranges over $\nu^\leafsymbol$. Then the inclusion $w :
    \Lambda^J h \omega \longrightarrow \Lambda^{J'} h \omega$ above is obtained
    as the following pushout
    \[
        \pushoutdiagram
            {\Lambda^{J'} h \omega_{/[e]}}{\Lambda^J h \omega}
                {h \omega_{/[e]}}{\Lambda^{J'} h \omega .}
            {}{\sfH_{\omega/[e]}^{J'}}{w}{}
    \]
\end{proof}

\begin{lemma}
    [{Generalization of \cite[lemma 5.1]{Moerdijk2009}}] Let $\omega \in
    \bbOO_{n+1}$, and $I$ be a non empty set of inner faces of $h \omega$. Then
    the inclusion $\sfH_{h \omega}^I : \Lambda^I h \omega \longhookrightarrow h
    \omega$ is an inner anodyne extension.
\end{lemma}
\begin{proof}
    We proceed by induction $d \eqdef \deg h \omega$ and on $m \eqdef \hash I$.
    Since $I$ is non empty, $d \geq 2$. If $d = 2$, then $h \omega$ has a
    unique inner face, and the corresponding inner horn inclusion is just the
    spine inclusion $S [h \omega] \longhookrightarrow h \omega$.

    Assume now that $d \geq 3$. If $m = 1$, then $\Lambda^I h \omega$ is just
    an inner horn, and the claim holds. If $m \geq 2$, take $f \in I$, and let
    $J \eqdef I - \{ f \}$. The inclusion $\sfH_{h \omega}^I$ decomposes as
    \[
        \Lambda^I h \omega
        \stackrel{u}{\longhookrightarrow} \Lambda^J h \omega
        \stackrel{\sfH_{h \omega}^J}{\longhookrightarrow} h \omega .
    \]
    By induction on $m$, $\sfH_{h \omega}^J$ is an inner anodyne extension. By
    \cref{lemma:relative-super-horns-inclusions}, $u$ is a cell complex
    of inner horn inclusions of opetopes of at most $d-1$ nodes, and so by
    induction on $d$, $u$ is an inner anodyne extension as well.
\end{proof}

\begin{lemma}
    \label{lemma:spine-super-horn}
    Let $\omega \in \bbOO_{n+1}$ have $d \geq 2$ nodes, and $I$ be the set
    of all inner faces of $\omega$. Note that $\Lambda^I h \omega$ contains all
    the generators $c_{\omega, [p]}$, and thus the spine inclusion
    $\sfS_{h \omega}$ decomposes as
    \[
        S [h \omega]
        \stackrel{u}{\longhookrightarrow} \Lambda^I h \omega
        \stackrel{\sfH_{h \omega}^I}{\longhookrightarrow} h \omega .
    \]
    Then the inclusion $u$ is a cell complex of spine inclusions of opetopes of
    at most $d - 1$ nodes.
\end{lemma}
\begin{proof}
    Note that an outer face (all elementary faces that are not inner) of $h
    \omega$ are inclusions $\omega' \subseteq \omega$ of subtrees of $\omega$
    of $d - 1$ nodes. Hence $\Lambda^I h \omega$ is the union on the left,
    while $S [h \omega]$ can be expressed as on the right
    \[
        \Lambda^I h \omega = \bigcup_{\substack{
            \omega' \subseteq \omega \\ \deg h \omega' = d - 1
        }} h \omega' ,
        \qqquad
        S [h \omega] = \bigcup_{\substack{
            \omega' \subseteq \omega \\ \deg h \omega' = d - 1
        }} S [h \omega'] .
    \]
    Thus
    \[
        u : \bigcup_{\substack{
            \omega' \subseteq \omega \\ \deg h \omega' = d - 1
        }} S [h \omega']
        \longhookrightarrow
        \bigcup_{\substack{
            \omega' \subseteq \omega \\ \deg h \omega' = d - 1
        }} h \omega'
    \]
    can be obtained as a composition of pushouts of the $\sfS_{\omega'}$'s.
    Explicitely, write $\omega_1, \ldots, \omega_m \subseteq \omega$ be the
    subtrees of $\omega$ with $d-1$ nodes, and $g_1, \ldots, g_m$ the
    associated outer face inclusions of $h \omega$. Define $X_0 \eqdef S
    [\omega]$, $\iota_0$ be the identity on $X_0$, and for $1 \leq i \leq m$,
    let $X_i$ be the pushout
    \[
        \diagramsize{}{4}
        \diagramarrows{}{c->}{c->}{}
        \pushoutdiagram
            {S [\omega_i]}{X_{i-1}}{h \omega_i}{X_i ,}
            {\iota_{i-1} g_i \sfS_{\omega_i}}{}{v}{}
    \]
    and $\iota_i \eqdef v \iota_{i-1}$. Then $\Lambda^I h \omega = X_m$.
\end{proof}

\begin{proposition}
    [{Generalization of \cite[proposition 2.4]{Cisinski2013}}]
    \label{prop:spine-complex-inner-anodyne}
    We have an inclusion ${}^\pitchfork \left( \sfSS^\pitchfork \right)
    \subseteq \mathsf{An}_{\mathrm{inner}}$.
\end{proposition}
\begin{proof}
    It is enough to show that spine inclusions $\sfS_{h \omega} : S [h \omega]
    \longhookrightarrow h \omega$ are inner anodyne extensions.
    \begin{enumerate}
        \item If $\deg h \omega = \hash \omega^\nodesymbol = 0, 1$, then
        $\sfS_{h \omega}$ is an identity, thus an inner anodyne extension.
        \item If $\deg h \omega = 2$, then $h \omega$ admits a unique inner
        face $f : h \ytree{\tgt \omega} \longrightarrow h \omega$, and note
        that $\Lambda^f h \omega = S [h \omega]$. Thus in this case, the spine
        inclusion is an inner horn inclusion.
        \item Assume $d = \deg h \omega \geq 3$, and let $I$ be the set of all
        inner faces of $h \omega$. Note that since $\Lambda^I h \omega$
        contains all generators of $h \omega$ (i.e. subtrees of one node), the
        spine inclusion $\sfS_{h \omega}$ decomposes as
        \[
            S [h \omega]
            \stackrel{u}{\longhookrightarrow} \Lambda^I h \omega
            \stackrel{\sfH_{h \omega}^f}{\longhookrightarrow} h \omega ,
        \]
        and in order to show that $\sfS_{h \omega}$ is an inner anodyne extension,
        it suffices to show that $u$ is. By
        \cref{lemma:spine-super-horn}, it is a cell complex of spine
        inclusions of opetopes of at most $d - 1$ nodes, thus by induction, it
        is an inner anodyne extension.
        \qedhere
    \end{enumerate}
\end{proof}

\begin{proposition}
    \label{prop:inner-horn-inclusions-S-local-isomorphisms}
    Inner anodyne extensions are $\sfSS$-local isomorphisms.
\end{proposition}
\begin{proof}
    It is enough to show that inner horn inclusions are $\sfSS$-local
    isomorphisms. Let $\omega \in \bbOO_{n+1}$ have $d \geq 2$ nodes. If $d =
    2$, the it has a unique inner face, and the corresponding inner horn
    inclusion is simply the spine inclusion $S [h \omega] \longhookrightarrow h
    \omega$.

    Assume $d \geq 3$, let $I$ be the set of all inner faces of $h \omega$, and
    $f \in I$. Then the spine inclusion $\sfS_{h \omega}$ decomposes as
    \[
        S [h \omega]
        \stackrel{u}{\longhookrightarrow} \Lambda^I h \omega
        \stackrel{v}{\longhookrightarrow} \Lambda^f h \omega
        \stackrel{\sfH_{h \omega}^f}{\longhookrightarrow} h \omega .
    \]
    By \cref{lemma:spine-super-horn}, $u$ is a spine complex, thus an $\sfSS$
    local-isomorphism. By \cref{lemma:relative-super-horns-inclusions}, $v$ is
    a cell complex of inner horn inclusions of opetopes with at most $d - 1$
    nodes. By induction, $v$ is an $\sfSS$ local-isomorphism. Thus $vu$ and
    $\sfS_{h \omega} = \sfH_{h \omega}^f \cdot (vu)$ are $\sfSS$
    local-isomorphisms, and by 3-for-2, so is $\sfH_{h \omega}^f$.
\end{proof}

\subsection{Cylinder objects}
\label{sec:cylinder}

The Resk cylinder construction of \cref{def:rezk-cylinder-algebra} extends to
$\PshLambda$.

\begin{definition}
    [Rezk interval]
    \label{def:rezk-interval-lambda}
    For $\phi \in \bbOO_{n-1}$, recall the definition of the Rezk interval
    $\frakJJ_\phi \in \Alg$. Define $\frakII_\phi \eqdef M
    \frakJJ_\phi$\index{$\frakII_\phi$}. Write $\sfI_\phi : h \phi
    \longrightarrow \frakII_\phi$\index{$\sfI_\phi$} for the endpoint
    inclusion, and $\sfEE_\frakII \eqdef \left\{ \sfI_\phi \mid \phi \in
    \bbOO_{n-1} \right\}$.
\end{definition}

\begin{definition}
    [Rezk cylinder]
    \label{def:rezk-cylinder-lambda}
    For $\omega \in \bbOO_{\geq n-1}$, let $\frakII h \omega \eqdef M (\frakJJ
    h \omega)$. Extend $\frakII$ by colimits to obtain a functor $\frakII :
    \PshLambda \longrightarrow \PshLambda$\index{$\frakII X$|see {Rezk
    cylinder}}. The \emph{Rezk cylinder}\index{Rezk cylinder} $\frakII X$ of a
    presheaf $X \in \PshLambda$ is a cylinder object in the sense of
    \cref{def:cylinder}, i.e. we have a canonical factorization of the
    codiagonal map
    \[
        \diagramarrows{}{>->}{<-}{}
        \triangleDdiagram
            {X + X}{\frakJJ X}{X .}
            {\nabla}{(i_0, i_1)}{\nabla}
    \]
\end{definition}

Explicitely, for $X \in \PshLambda$,
\[
    \frakII X \eqdef \colim_{h \omega \rightarrow X} M (\frakJJ h \omega) .
\]
In dimension $(n-1)$ and $n$, $\frakII X$ is the following pushout
\[
    \pushoutdiagram
        {X_{n-1} + X_{n-1}}{X_{n-1, n} + X_{n-1, n}}
            {M \frakJJ_{X_{n-1}}}{(\frakII X)_{n-1, n} ,}
        {}{b}{}{}
\]
where $b$ maps a $x \in X_\psi$ in the first (resp. second) component to $0_x$
(resp. $1_x$) in $M \frakJJ_x \subseteq M \frakJJ_{X_{n-1}}$. The $(n+1)$-cells
of $\frakII X$ are so that in $v \frakII X$, the following relation (analogous
to \cref{eq:rezk-cylinder:relations}) holds:
\[
    x^{(1)} \biggraft_{[p_i]} j_{\src_{[p_i]} x}
    = j_{\tgt x} \graft_{[]} x^{(0)}
\]
for a cell $x \in X_\omega$, $\omega \in \bbOO_n$, and $\omega^\nodesymbol =
\left\{ [p_1], \ldots \right\}$. We readily deduce the following:

\begin{lemma}
    \label{lemma:cynlinder:I-vs-J}
    For $X \in \PshLambda$ we have a canonical isomorphism $\frakJJ v X \cong v
    \frakII X$.
\end{lemma}

\begin{proposition}
    The functorial cylinder $\frakII$ is an elementary homotopical data
    (\cref{def:elementary-homotopical-data}).
\end{proposition}
\begin{proof}
    Straightforward unpacking of the definition of $\frakII$.
\end{proof}

\subsection{Homotopical structure}
\label{sec:homotopical-structure}

Recall from \cref{def:inner-horn} that the class $\mathsf{An}_{\mathrm{inner}}$
of inner anodyne extensions of $\PshLambda$ is the class ${}^\pitchfork
\left(\sfHH_{\mathrm{inner}}^\pitchfork \right)$ of retracts of cell complexes
of inner horn inclusions.

\begin{definition}
    [{$\frakII$-anodyne extension}]
    \label{def:i-anodyne-extension}
    Recall from \cref{def:rezk-interval-lambda} the set $\sfEE_\frakII$ of
    endpoint inclusions of the Rezk intervals in $\PshLambda$. Let
    $\mathsf{An}_\frakII$, the class of \emph{$\frakII$-anodyne
    extensions}\index{anodyne extension}, be the class ${}^\pitchfork
    \left(\left(\sfHH_{\mathrm{inner}} \cup \sfEE_\frakII \right)^\pitchfork
    \right)$.
\end{definition}

\begin{definition}
    [Lifting problem]
    \label{def:lifting-problem}
    Let $k : K \longhookrightarrow h \omega$ be a subpresheaf of a
    representable presheaf, and $f : K \longrightarrow X$ a morphism. We say
    that $f$ is a \emph{lifting problem}\index{lifting problem} of degree $d$
    (or \emph{$d$-lifting problem}), where $d = \deg h \omega$. We say that $f$
    is \emph{$k$-unsolved} (or just \emph{unsolved} if $k$ is clear from the
    context) if $f$ does not factor through $k$.
\end{definition}

\begin{lemma}
    \label{lemma:unit-mono:inner-anodyne}
    Let $X \in \PshDelta$, be such that the unit map $\eta_X : X
    \longrightarrow M v X$ is a monomorphism. Then $\eta_X$ is an inner horn
    complex, and in particular an inner anodyne extension.
\end{lemma}
\begin{proof}
    Let $X^{(0)} \eqdef X$ and $\iota^{(0)} \eqdef \eta_X$. If $x \in X^{(0)}$,
    we write $\deg x = 0$. Assume by induction that we have an inclusion
    $\iota^{(\alpha)} : X^{(\alpha)} \longhookrightarrow M v X$ for all ordinal
    $\alpha < \beta$, and a degree function $\deg : X^{(\alpha)}
    \longrightarrow \alpha + 1$.

    If $\beta$ is a limit ordinal, simply set $X^{(\beta)} \eqdef
    \bigcup_{\alpha < \beta} X^{(\alpha)}$, and $\iota^{(\beta)}$ to be the
    induced inclusion. Assume that $\beta$ is a successor ordinal, say $\beta =
    \alpha + 1$, and choose a horn lifting problem $l : \Lambda^f h \omega
    \longrightarrow X^{(\alpha)}$ such that
    \begin{enumerate}
        \item \condition{L1} $l$ is unsolved in $X^{(\alpha)}$,
        \item \condition{L2} $l$ has degree $\leq 2$ only if there is no
        unsolved horn lifting problem of degree $\geq 3$.
    \end{enumerate}
    If such an $l$ does not exist, simply set $X^{(\alpha + 1)} \eqdef
    X^{(\alpha)}$ and $\iota^{(\alpha + 1)} \eqdef \iota^{(\alpha)}$. If it
    does, set $X^{(\alpha + 1)}$ to be the following pushout
    \[
        \diagramarrows{}{c->}{c->}{}
        \pushoutdiagram
            {\Lambda^f h \omega}{X^{(\alpha)}}{h \omega}{X^{(\alpha + 1)} ,}
            {l}{}{}{u}
    \]
    and $\iota^{(\alpha + 1)}$ to be the induced map $X^{(\alpha + 1)}
    \longrightarrow M v X$. For $x \in X^{(\alpha + 1)} - X^{(\alpha)}$, set
    $\deg x \eqdef \alpha + 1$.

    We claim that $\iota^{(\alpha + 1)}$ is a monomorphism. Towards a
    contradiction, assume that it is not. For $\nu \subseteq \omega$ the
    two-nodes subtree corresponding to the inner face $f$, write $x \in
    X^{(\alpha + 1)}_{h \nu}$ the cell selected by the arrow
    \[
        h \nu \longhookrightarrow h \omega
        \stackrel{u}{\longrightarrow} X^{(\alpha + 1)}
    \]
    Then there exist $y \in X^{(\alpha)}_{h \nu}$ such that $\iota_{(\alpha +
    1)} (x) = \iota_{(\alpha + 1)} (y)$.
    \begin{enumerate}
        \item If $l$ has degree $\leq 2$, then the lifting problem $l$ was not
        unsolved in $X^{(\alpha)}$, a contradiction with condition
        \condition{L1}.
        \item If the degree of $l$ is $\geq 3$, then the cell $y$ fives a
        factorization of $l$ as
        \[
            \triangleULdiagram
                {\Lambda^f h \omega}{X^{(\alpha)}}{\partial h \omega ,}
                {l}{}{l'}
        \]
        with $l'$ unsolved in $X^{(\alpha)}$. But this contradicts condition
        \condition{L2}. Indeed, the step of the induction that created $y$ but
        not $u (h \omega)$ only considered a lifting problem of degree $2$,
        whereas $l$ could have been considered instead.
    \end{enumerate}
    Therefore, $\iota^{(\alpha + 1)}$ is a monomorphism.

    If $\kappa$ is the cardinal of the set of horn lifting problems of $M v X$.
    Then the sequence $X_\alpha$ stabilizes after $\kappa + 1$, as all lifting
    problems have been exhausted. Clearly,
    \[
        M v X \cong \colim_{\alpha < \kappa + 1} X^{(\alpha)} ,
    \]
    and by construction, $\eta_X$ is an inner horn complex.
\end{proof}

\begin{corollary}
    [Generation lemma]
    \label{coroll:generation-lemma}
    Take $X \in \PshLambda$, $A \in \Alg$, and $m : X \longhookrightarrow M A$
    be a monomorphism such that its transpose $\barM : v X \longrightarrow A$
    is an isomorphism. Then $m$ is an inner horn complex, and in particular, an
    inner anodyne extension.
\end{corollary}
\begin{proof}
    The condition states that up to isomorphism, $m$ is the unit map $\eta_X :
    X \longrightarrow M v X$. We can apply \cref{lemma:unit-mono:inner-anodyne}
    to conclude.
\end{proof}

\begin{proposition}
    \label{prop:cylinder:An1}
    The pair $(\frakII, \mathsf{An}_\frakII)$ satisfies condition
    \condition{An1} of \cref{def:anodyne-extension}.
\end{proposition}
\begin{proof}
    By \cite[proposition 1.1.16]{Cisinski2006}, it is enough to check the claim
    when $m$ is a boundary inclusion, say $m = \sfB_{h \omega} : X = \partial h
    \omega \longrightarrow Y = h \omega$, where $\omega \in \bbOO_{n+1}$.
    \begin{enumerate}
        \item Assume $\deg h \omega = 0$, i.e. $\omega$ is degenerate, say
        $\omega = \itree{\phi}$. Then $\partial h \omega = \emptyset$, and
        since $\frakII$ preserves colimits, $\frakII \partial h \omega =
        \emptyset$ as well. Thus $g$ is the inclusion $h \omega^{(e)} = h
        \itree{\phi} \longrightarrow \frakII h \omega = M \frakJJ_\phi$, i.e.
        an endpoint inclusion of the Rezk interval $\frakJJ_\phi$, which by
        definition is $\frakII$-anodyne

        \item Assume $\deg h \omega \geq 1$. We only treat the case $e = 0$,
        the other one being similar. Let $[p] \in \omega^\nodesymbol$, and
        consider the $n$-cell $a = \src_{[p]} \omega$ of $h \omega$. By
        \cref{eq:rezk-cylinder:relations}, in $\frakII h \omega$, we have
        \[
            a^{(1)} = j_{\tgt a} \graft_{[]} a^{(0)}
                \biggraft_{[q_i]} j_{\src_{[q_i]} a}^{-1}
        \]
        where $[q_i]$ ranges over $(\src_{[p]} \omega)^\nodesymbol$. Therefore,
        $\frakII h \omega$ is freely generated by $\frakII \partial h \omega
        \cup h \omega^{(0)}$. By \cref{coroll:generation-lemma},
        $g$ is an inner anodyne extension.
        \qedhere
    \end{enumerate}
\end{proof}

\begin{proposition}
    \label{prop:cylinder:An2}
    The pair $(\frakII, \mathsf{An}_\frakII)$ satisfies condition
    \condition{An2} of \cref{def:anodyne-extension}.
\end{proposition}
\begin{proof}
    By \cite[proposition 1.1.16]{Cisinski2006}, it is enough to check the claim
    when $m$ is an inner horn inclusion or an endpoint inclusion of a Rezk
    interval.
    \begin{enumerate}
        \item Assume $m$ is a inner horn inclusion, say $m = \sfH_{h \omega}^f
        : X = \Lambda^f h \omega \longrightarrow Y = h \omega$, where $\omega
        \in \bbOO_{n+1}$. Recall that by
        \cref{prop:inner-horn-inclusions-S-local-isomorphisms}, $v \Lambda^f h
        \omega = h \omega$. Thus applying $v$ to the diagram of
        \cref{def:elementary-homotopical-data} \condition{An2} yeilds
        \[
            \begin{tikzcd}
                h \omega + h \omega
                    \ar[r, equal] \ar[d]
                    \ar[dr, phantom, "\ulcorner", very near end] &
                    h \omega + h \omega
                    \ar[d] \ar[ddr, bend left] \\
                \frakJJ h \omega
                    \ar[r, equal] \ar[drr, bend right, equal] &
                \frakJJ h \omega
                    \ar[dr, dashed, "v g" description] \\
                &
                &
                \frakJJ h \omega .
            \end{tikzcd}
        \]
        Thus $g : \frakII \Lambda^f h \omega \cup (h \omega + h \omega)
        \longrightarrow \frakII h \omega = M \frakJJ h \omega$ is such that its
        transpose
        \[
            \frakJJ h \omega
            \stackrel{v g}{\longrightarrow} v M \frakJJ h \omega
            \cong \frakJJ h \omega
        \]
        is an isomorphism. By \cref{coroll:generation-lemma}, $g$ is an inner
        anodyne extension.

        \item Assume $m$ is an endpoint inclusion of a Rezk interval, say $m :
        X = M h \phi \longhookrightarrow Y = M \frakJJ_\phi$, for a $\phi \in
        \bbOO_{n-1}$. We thus have $\frakII X \cup (Y + Y) = \frakII M h\phi
        \cup (M \frakJJ_\phi + M \frakJJ_\phi) = M \frakJJ_\phi \cup (M
        \frakJJ_\phi + M \frakJJ_\phi)$.

        We have $\frakII Y = \frakII (M \frakJJ_\phi) = M (\frakJJ v (M
        \frakJJ_\phi)) \cong M (\frakJJ \frakJJ_\phi)$. The algebra $\frakJJ
        \frakJJ_\phi$ contains four $(n-1)$-cells $00_\phi$, $01_\phi$,
        $10_\phi$, and $11_\phi$ of shape $\phi$, and is generated by the
        $n$-cells $j_\phi^{(0*)} : 00_\phi \longrightarrow 01_\phi$,
        $j_\phi^{(1*)} : 10_\phi \longrightarrow 11_\phi$, $j_\phi^{(*0)} :
        00_\phi \longrightarrow 10_\phi$, $j_\phi^{(1*)} : 10_\phi
        \longrightarrow 11_\phi$ and their inverses. Further, the equality on
        the right holds, which can be depicted as a commutative square of
        invertible arrows on the right:
        \[
            j_\phi^{(*1)} \graft_{[]} j_\phi^{(0*)}
            = j_\phi^{(1*)} \graft_{[]} j_\phi^{(*0)} ,
            \qqquad
            \squarediagram
                {00_\phi}{01_\phi}{10_\phi}{11_\phi .}
                {j_\phi^{(0*)}}{j_\phi^{(*0)}}{j_\phi^{(*1)}}{j_\phi^{(1*)}}
        \]
        On the other hand, the pushout $M \frakJJ_\phi \cup (M \frakJJ_\phi + M
        \frakJJ_\phi)$ contains $j_\phi^{(0*)}$, $j_\phi^{(*0)}$,
        $j_\phi^{(*1)}$, and their inverses. Thus it generates $\frakJJ
        \frakJJ_\phi$, i.e. the cocartesian gap map $g$ satisfies the
        conditions of \cref{coroll:generation-lemma}. Consequently, it is an
        inner anodyne extension.
        \qedhere
    \end{enumerate}
\end{proof}

\begin{theorem}
    \label{th:pshlambda:infty-model-structure}
    The category $\bbLambda$ endowed with the functorial cylinder $\frakII :
    \PshLambda \longrightarrow \PshLambda$ and the class $\mathsf{An}_\frakII$
    of $\frakII$-anodyne extension forms a homotopical structure
    (\cref{def:homotopical-structure})

    Using \cref{th:cisinsli-model-structure}, we obtain the following
    \emph{model structure for $\infty$-algebras}\index{model structure for
    $\infty$-algebras} $\PshLambda_\infty$\index{$\PshLambda_\infty$|see {model
    structure for $\infty$-algebras}} on $\PshLambda$:
    \begin{enumerate}
        \item a morphism $f$ is a \emph{naive fibration} if
        $\mathsf{An}_\frakII \pitchfork f$ (\cref{def:i-anodyne-extension}); a
        presheaf $X \in \PshLambda$ is \emph{fibrant} if the terminal morphism
        $X \longrightarrow 1$ is a naive fibration;
        \item a morphism $f : X \longrightarrow Y$ is a \emph{weak equivalence}
        if for all fibrant object $P \in \PshLambda$, the induced map $f^* :
        \Ho\PshLambda (Y, P) \longrightarrow \Ho\PshLambda (Y, P)$ is a
        bijection, where $\simeq$ is the $\frakII$-homotopy relation of
        \cref{def:rezk-cylinder-lambda,def:homotopy};
        \item a morphism $f$ is a \emph{cofibrations} if it is a monomorphisms,
        it is a \emph{acyclic cofibrations} if it is a cofibration and a weak
        equivalence;
        \item a morphism $f$ is a fibration if it has the right lifting
        property with respect to acyclic cofibrations, it is an \emph{acyclic
        fibration} if it has the right lifting property with respect to all
        cofibrations.
    \end{enumerate}
    In particular, $\PshLambda_\infty$ is of Cisinski type
    (\cref{def:cisinski-model-structure}), cellular, and proper. Fibrant
    objects in $\PshLambda_\infty$ are called
    \emph{$\infty$-algebras}\index{$\infty$-algebra} (or \emph{inner Kan
    complexes}\index{inner Kan complex|see {$\infty$-algebra}}).
\end{theorem}
\begin{proof}
    By definition, $\mathsf{An}_\frakII$ is the class of cell complexes over a
    set of monomorphisms, thus it satisfies axiom \condition{An0}. Axioms
    \condition{An1} and \condition{An2} are checked by
    \cref{prop:cylinder:An1,prop:cylinder:An2} respectively.
\end{proof}

This construction is a direct generalization of the model structure
$\PshDelta_{\mathrm{Joyal}}$ for quasi-categories \cite[theorem
1.9]{Joyal2007}, and of $\Psh\bbOmega_{\mathrm{CM}}$ for planar
$\infty$-operads \cite[theorem 2.4]{Cisinski2011}.

\begin{lemma}
    \label{lemma:everyone-is-isofibrant}
    For all $X \in \PshLambda$ we have $\sfEE_\frakII \pitchfork X$. In
    particular, if $\sfHH_{\mathrm{inner}} \pitchfork X$ if and only $X$ is an
    $\infty$-algebra.
\end{lemma}
\begin{proof}
    Let $\phi \in \bbOO_{n-1}$, and consider a lifting problem $f : h \phi
    \longrightarrow X$. In particular, $f$ exhibits a cell $x \in X_{h \phi}$,
    but also a cell $\id_x \in X_{h \ytree{\phi}}$. Then a lift $\barF$ of $f$
    as in
    \[
        \triangleULdiagram
            {h \phi}{X}{\frakII_\phi}{f}{\sfI_\phi}{\barF}
    \]
    can be obtained by mapping $0_\phi$ and $1_\phi$ to $x$, and $j_\phi$ and
    $j_\phi^{-1}$ to $\id_x$ (see \cref{def:rezk-interval-algebra} for
    notations).
\end{proof}

The following results comes as a ``sanity-check'' for the model structure of
\cref{th:pshlambda:infty-model-structure}:

\begin{proposition}
    \label{prop:nerves-infty-algebras}
    For $A \in \Alg$, its nerve $M A$ is an $\infty$-algebra.
\end{proposition}
\begin{proof}
    By \cref{th:algebraic-localization:lambda}, $\sfSS \perp M A$, and thus by
    \cref{prop:inner-horn-inclusions-S-local-isomorphisms},
    $\sfHH_{\mathrm{inner}} \perp M A$. We apply
    \cref{lemma:everyone-is-isofibrant} to conclude.
\end{proof}

\begin{proposition}
    \label{prop:infty-folk-quillen-adjunction}
    We have a Quillen adjunction $v : \PshLambda_\infty \adjunction
    \Alg_{\mathrm{folk}} : M$.
\end{proposition}
\begin{proof}
    Trivially, if $f : X \longrightarrow Y$ is a monomorphism in $\PshLambda$,
    then $v f$ is injective on $(n-1)$-cells. Therefore, $v$ preserves
    cofibrations. To conclude, it suffices to show that $M$ preserves
    fibrations.

    First, note that $v$ maps $\sfHH_{\mathrm{inner}}$ to isomorphisms, and
    since the adjunction is reflective, it maps $\sfEE_\frakII = M
    \sfEE_\frakJJ$ (see \cref{def:rezk-interval-lambda}) to $\sfEE_\frakJJ$ up
    to isomorphism. Therefore, $v$ maps $\mathsf{An}_\frakII$ to acyclic
    cofibrations. Let now $f$ be a fibration in $\Alg_{\mathrm{folk}}$. Then
    $f$ has the right lifting property against all acyclic cofibrations, and in
    particular, $v \mathsf{An}_\frakII \pitchfork f$. By adjointness,
    $\mathsf{An}_\frakII \pitchfork M f$, and thus $Mf$ is a naive fibration.
    By \cref{prop:nerves-infty-algebras}, the codomain of $M f$ is an
    $\infty$-algebra, and we apply \cref{lemma:cisinski:1.3.36} to conclude
    that $M f$ is a fibration.
\end{proof}

\begin{definition}
    [Quillen model structure]
    \label{def:pshlambda:quillen-model-structure}
    Let $\sfHH$ be the set of all horn inclusions of $\bbLambda$, not only the
    inner ones. The \emph{Quillen model structure}\index{Quillen model
    structure}
    $\PshLambda_{\mathrm{Quillen}}$\index{$\PshLambda_{\mathrm{Quillen}}$|see
    {Quillen model structure}} on $\PshLambda$ is the Bousfield localization
    $\sfHH^{-1} \PshLambda_\infty$, which exists by \cite[theorem
    4.1.1]{Hirschhorn2009}.

    Note that $\PshLambda_{\mathrm{Quillen}}$ is still of Cisinski type,
    cellular, and proper. In this structure, fibrant objects are called
    \emph{Kan complexes}\index{Kan complex}, and fibrations are \emph{Kan
    fibrations}\index{Kan fibration}. For instance, in the case $(k, n) = (1,
    1)$, we recover the classical Quillen structure on simplicial sets.
\end{definition}

\subsection{Underlying categories}
\label{sec:underlying-categories}

This section is devoted to prove the technical
\cref{lemma:v-reflects-isofibrations}. If $n = 1$, i.e. $\bbLambda_{1, 1} =
\bbDelta$, then the result is stated in \cite[proposition 1.13]{Joyal2007} and
\cite[2.13]{Joyal}. The bulk of the work is to study the \emph{underlying
category}\index{underlying category} adjunction
\begin{equation}
    C_{\phi !} : \PshDelta \adjunction \PshLambda : C_\phi^* ,
\end{equation}
where $\phi \in \bbOO_{n-1}$, to reduce a similar statements for an arbitrary
$n \geq 1$ to $n = 1$.

In this section, we assume $n \geq 1$, and for once, do not omit if from
various notations, e.g. $\bbLambda_{1, n}$, $\Alg_{1, n}$, etc. Recall that
$\bbLambda_{1, 1} = \bbDelta$ and $\Alg_{1, 1} = \Cat$. In this case, the
adjunction $v_{1, n} : \Psh{\bbLambda_{1, n}} \adjunction \Alg_{1, n} : M_{1,
n}$ is more commonly denoted by $\tau : \PshDelta \adjunction \Cat : N$.

Pick $\phi \in \bbOO_{n-1}$, and define an functor $C_\phi : \bbDelta
\longrightarrow \bbLambda_{1, n}$ as
\[
    [0] \longmapsto h \phi ,
    \qquad
    [1] \longmapsto h \ytree{\phi} ,
    \qquad
    [i] \longmapsto h \left(
        \ytree{\phi} \graft_{[]} \ytree{\phi} \graft_{[]} \cdots
        \graft_{[]} \ytree{\phi}
    \right) ,
\]
where on the right, there are $i$ instances of $\ytree{\phi}$. In other words,
$C_\phi$ maps $[i]$ to $h \omega$, where $\omega \in \bbOO_{n+1}$ is the linear
tree with $i$ nodes, all decorated by $\ytree{\phi} \in \bbOO_n$. Note that
$C_\phi$ is an embedding that exhibits $\bbDelta$ as an
\emph{sieve}\index{sieve} of $\bbLambda_{1, n}$ in that of for all morphism
$\lambda \longrightarrow \lambda '$ in $\bbLambda_{1, n}$, if $\lambda' \in \im
C_\phi$, then so is $\lambda$.

This functor induces an adjunction
\begin{equation}
    \label{eq:dphi-presheaf-adjunction}
    C_{\phi !} : \PshDelta \adjunction \Psh{\bbLambda_{1, n}} : C_\phi^* ,
\end{equation}
where $C_{\phi !}$ is the left Kan extension of $\bbDelta \xrightarrow{C_\phi}
\bbLambda \longrightarrow \Psh{\bbLambda_{1, n}}$ along the Yoneda embedding,
and $C_\phi^*$ is the precomposition by $C_\phi$. Since $C_\phi$ is an
embedding, so is $C_{\phi !}$, and furthermore, it exhibits $\PshDelta$ as an
sieve of $\Psh{\bbLambda_{1, n}}$. The following result follows directly from
this observation:

\begin{lemma}
    \label{lemma:C-constructions}
    Let $k \in \bbNN$.
    \begin{enumerate}
        \item We have $C_{\phi !} S [k] \cong S [C_\phi [k]]$.
        \item If $\Lambda^l [k]$ is an inner horn of $[k]$, then $C_{\phi !}
        \Lambda^l [k]$ is an inner horn of $C_\phi [k]$, and all inner horns of
        $C_\phi [k]$ are obtained in this way.
    \end{enumerate}
\end{lemma}
\begin{proof}
    The spine and inner horns of $C_\phi [k]$ are obtained as colimits of
    presheaves over $C_\phi [k]$. Since $C_{\phi !}$ exhibits $\PshDelta$ as an
    sieve of $\Psh{\bbLambda_{1, n}}$, those colimits can be computed in
    $\PshDelta$.
\end{proof}

\begin{lemma}
    \label{lemma:C-adjunction-restricts}
    The adjunction \cref{eq:dphi-presheaf-adjunction} restrics and corestricts
    as an adjunction $C_{\phi !} : \Cat \adjunction \Alg_{1, n} : C_\phi^*$. We
    thus have three commutative squares, where the last one is obtained by
    adjunction from the second one:
    \[
        \diagramarrows{->}{<-}{<-}{->}
        \squarediagram
            {\PshDelta}{\Psh{\bbLambda_{1, n}}}{\Cat}{\Alg_{1, n} ,}
            {C_{\phi !}}{N}{M_{1, n}}{C_{\phi !}}
        \qqquad
        \diagramarrows{<-}{<-}{<-}{<-}
        \squarediagram
            {\PshDelta}{\Psh{\bbLambda_{1, n}}}{\Cat}{\Alg_{1, n} ,}
            {C_\phi^*}{N}{M_{1, n}}{C_\phi^*}
        \qqquad
        \diagramarrows{->}{->}{->}{->}
        \squarediagram
            {\PshDelta}{\Psh{\bbLambda_{1, n}}}{\Cat}{\Alg_{1, n} .}
            {C_{\phi !}}{\tau}{v_{1, n}}{C_{\phi !}}
    \]
\end{lemma}
\begin{proof}
    \begin{enumerate}
        \item Take $K \in \PshDelta$. Note that if $\phi' \in \bbOO_{n-1}$,
        $\phi' \neq \phi$, then $(C_{\phi !} K)_{\phi'} = \emptyset$, i.e.
        $C_{\phi !} K$ does not have cells of shape $\phi' \neq \phi$. If
        $\omega \in \bbOO_{\geq n-1}$, $\omega \nin \im C_\phi$, then $\omega$
        has a $(n-1)$-dimensional face different from $\phi$, whence
        $\Psh{\bbLambda_{1, n}} (S [h \omega], C_{\phi !} K) = \emptyset$. By
        \cref{lemma:C-constructions} and the fact that $C_{\phi !}$ is fully faithful,
        if $\sfSS_{1, 1} \perp K$, then $\sfSS_{1, n} \perp d_{\phi, !} K$, and
        $C_{\phi !}$ restrics and corestricts as a functor $\Cat
        \longrightarrow \Alg_{1, n}$.
        \item Take $A \in \Alg_{1, n}$, i.e. a presheaf $A \in
        \Psh{\bbLambda_{1, n}}$ such that $\sfSS_{1, n} \perp A$. By
        \cref{lemma:C-constructions}, $C_{\phi !} \sfSS_{1, 1} \subseteq \sfSS_{1,
        n}$, so in particular, $C_{\phi !} \sfSS_{1, 1} \perp A$, whence
        $\sfSS_{1, 1} \perp C_\phi^* A$, i.e. $C_\phi^* A \in \Cat$.
    \end{enumerate}
\end{proof}

\begin{lemma}
    \label{lemma:C-cylinders}
    \begin{enumerate}
        \item For $A \in \Cat$, we have $C_{\phi !} \frakJJ A \cong \frakJJ
        C_{\phi !} A$.
        \item For $K \in \PshDelta$, we have $C_{\phi !} \frakII K \cong
        \frakII C_{\phi !} K$.
        \item Let $f, g : K \longrightarrow L$ be two parallel maps in
        $\PshDelta$. If $f \simeq g$ (\cref{def:homotopy}), then $C_{\phi !} f
        \simeq C_{\phi !} g$.
    \end{enumerate}
\end{lemma}
\begin{proof}
    Point (1) is by definition, and (3) follows from (2). To prove (2), consider
    \begin{align*}
        C_{\phi !} \frakII K
        &= C_{\phi !} N \frakJJ \tau K
            & \text{\cref{def:rezk-cylinder-lambda}} \\
        &\cong M_{1, n} C_{\phi !} \frakJJ \tau K
            & \text{\cref{lemma:C-adjunction-restricts}} \\
        &\cong M_{1, n} \frakJJ C_{\phi !} \tau K
            & \text{\cref{lemma:C-adjunction-restricts}} \\
        &\cong M_{1, n} \frakJJ v_{1, n} C_{\phi !} K
            & \text{\cref{lemma:C-adjunction-restricts}} \\
        &= \frakII C_{\phi !} K
            & \text{\cref{def:rezk-cylinder-lambda}.}
    \end{align*}
\end{proof}

\begin{proposition}
    \label{prop:C-quillen}
    We have a Quillen adjunction $C_{\phi !} : \PshDelta_\infty \adjunction
    \Psh{\bbLambda_{1, n}}_\infty : C_\phi^*$.
\end{proposition}
\begin{proof}
    Clearly, $C_{\phi !}$ preserves monomorphisms, and by
    \cref{lemma:C-cylinders}, it preserves the homotopy relation, thus the weak
    equivalences.
\end{proof}

\begin{lemma}
    [{Generalization of \cite[proposition 1.13]{Joyal2007} and
    \cite[2.13]{Joyal}}]
    \label{lemma:v-reflects-isofibrations}
    Let $X, Y \in \PshLambda$ be $\infty$-algebras, and $f : X \longrightarrow
    Y$ be an inner fibration (\cref{def:inner-anodyne-extension}). If $v f$ is
    a fibration in $\Alg_{\mathrm{folk}}$ (i.e. $\sfEE_\frakJJ \pitchfork v
    f$), then $\sfEE_\frakII \pitchfork f$.
\end{lemma}
\begin{proof}
    If $n = 1$, i.e. $\bbLambda = \bbDelta$, then the result holds by
    \cite[proposition 1.13]{Joyal2007}. Assume $n > 1$, and take $\phi \in
    \bbOO_{n-1}$. By \cref{prop:C-quillen}, $C_\phi^*$ is a right Quillen
    functor, thus $C_\phi^* X$ and $C_\phi^* Y$ are quasi-categories. By
    \cref{lemma:C-constructions}, $C_\phi^* f$ is an inner fibration, since $f$
    is. By assumption, $\sfEE_\frakJJ \pitchfork v_{1, n} f$, and in
    particular, $C_{\phi !} \sfJ_\optOne = \sfJ_\phi \pitchfork v_{1, n} f$. By
    adjunction and \cref{lemma:C-adjunction-restricts}, $\sfJ_\optOne
    \pitchfork C_\phi^* v_{1, n} f = \tau C_\phi^* f$, i.e. $C_\phi^* f$ is a
    fibration in $\Cat_{\mathrm{folk}}$. Finally, $C_\phi^* f$ satisfies the
    conditions of \cite[proposition 1.13]{Joyal2007}, hence $\sfI_\optOne
    \pitchfork C_\phi^* f$. By adjunction, $\sfI_\phi \pitchfork f$
\end{proof}

\subsection{Simplicial tensor and cotensor}


\begin{definition}
    \label{def:pshlambda:simplicial-tensor}
    For $k \in \bbNN$ and $\lambda \in \bbLambda$, let $\Delta [k] \otimes
    \lambda$ be the nerve
    \[
        \Delta [k] \otimes \lambda
        \eqdef
        M \left(
            \frakJJ \lambda
            \coprod_\lambda \frakJJ \lambda
            \coprod_\lambda \frakJJ \lambda
            \coprod_\lambda \cdots
            \coprod_\lambda \frakJJ \lambda
        \right) ,
    \]
    where there is $k$ instances of $\frakJJ \lambda$. In other words, it is
    the nerve of $k$ instances of the cylinder $\frakJJ \lambda$ ``glued
    end-to-end''. Extending in both variables by colimits yields a tensor
    product $- \otimes - : \PshDelta \times \PshLambda \longrightarrow
    \PshLambda$.

    Let us unfold the definition a little bit. Take $X \in \PshLambda$ and $K
    \in \PshDelta$. Then for each $x \in X_{h \phi}$ with $\phi \in
    \bbOO_{n-1}$, and $k \in K_0$, there is a cell $k \otimes x \in (K \otimes
    X)_{h \phi}$. For all edge $e \in K_1$, there is an isomorphism $e \otimes
    x \in (K \otimes X)_{h \ytree{\phi}}$ with source $d_1 e \otimes x$ and
    target $d_0 e \otimes x$. More generally, for every $m$-cell $k \in K_m$,
    writing $k_0, \ldots, k_m \in K_0$ its vertices, and $k_{i, j}$ its edge
    from $k_i$ to $k_j$, where $0 \leq i < j \leq m$, we have a cell
    \[
        k \otimes x \in (K \otimes X)_{h \omega} ,
        \qqquad
        \omega \eqdef C_\phi [m] = \ytree{\phi} \graft_{[]} \ytree{\phi}
            \graft_{[]} \cdots \graft_{[]} \ytree{\phi} ,
    \]
    such that $\src_{[*^i]} (k \otimes x) = k_{i, i+1} \otimes x$.

    With this description, it is clear that for $\frakII_\optOne \in \PshDelta$
    the nerve of the groupoid generated by one isomorphism
    (\cref{def:rezk-interval-lambda}), we have $\frakII X \cong \frakII_\optOne
    \otimes X$ for all $X \in \PshLambda$.
\end{definition}

\begin{definition}
    \label{def:pshlambda:simplicial-cotensor-mapping-space}
    A mapping space and cotensor can be constructed from the tensor product
    $\otimes$ of \cref{def:pshlambda:simplicial-tensor} so as to make
    $\PshLambda$ tensored and cotensored over $\PshDelta$:
    \[
        \Map (X, Y)_k \eqdef \PshLambda (\Delta [k] \otimes X, Y) ,
        \qqquad
        (Y^K)_\lambda \eqdef \PshLambda (K \otimes \lambda, Y) ,
    \]
    where $X, Y \in \PshLambda$, $K \in \PshDelta$, and $k \in \bbNN$.
\end{definition}


\begin{lemma}
    \label{lemma:tensor-cotensor-transpose-preserves-homotopy}
    For $K \in \PshDelta$ and $X, Y \in \PshLambda$, consider the natural
    $\hom$-set isomorphism
    \[
        \Phi : \PshLambda (K \otimes X, Y) \longrightarrow \PshLambda (X, Y^K)
    \]
    of the adjunction $K \otimes - : \PshLambda \adjunction \PshLambda :
    (-)^K$. The map $\Phi$ preserves and reflects the $\frakII$-homotopy
    relation (\cref{def:homotopy,def:rezk-cylinder-lambda}), i.e. it induces an
    isomorphism
    \[
        \Phi : \PshLambda (K \otimes X, Y) / {\simeq}
            \longrightarrow \PshLambda (X, Y^K) / {\simeq} .
    \]
\end{lemma}
\begin{proof}
    It is enough to show that $\Phi$ preserves and reflects the elementary
    $\frakII$-homotopy relation. Let $f, g : K \otimes X \longrightarrow Y$ be
    elementary homotopic maps, i.e. such that there exist a homotopy $H :
    \frakII (K \otimes X) \longrightarrow Y$ making the following triangle
    commute:
    \[
        \triangleDLdiagram
            {(K \otimes X) + (K \otimes X)}{\frakII (K \otimes X)}{Y.}
            {(i_0, i_1)}{f + g}{H}
    \]
    Note that $(K \otimes X) + (K \otimes X) \cong K \otimes (X + X)$, and
    $\frakII (K \otimes X) \cong \frakII_\optOne \otimes K \otimes X \cong
    (\frakII_\optOne \times K) \otimes X \cong (K \times \frakII_\optOne)
    \otimes X \cong K \otimes \frakII X$. Under the adjunction, the triangle
    above transposes as
    \[
        \triangleDLdiagram
            {X + X}{\frakII X}{Y^K ,}
            {(i_0, i_1)}{\Phi f + \Phi g}{\Phi H}
    \]
    exhibiting a homotopy from $\Phi f$ to $\Phi g$. Reflection of homotopies
    is done similarly.
\end{proof}

\begin{lemma}
    \label{lemma:pshlambda:simplicial-tensor:anodyne-extensions}
    Let $K \in \PshDelta$.
    \begin{enumerate}
        \item For $\omega \in \bbOO_{\geq n-1}$, and $e$ an inner face of $h
        \omega$, the map $K \otimes \sfH_{h \omega}^e : K \otimes \Lambda^e h
        \omega \longrightarrow K \otimes h \omega$ is an anodyne extension.
        \item For $\phi \in \bbOO_{n-1}$, the map $K \otimes \sfI_\phi : K
        \otimes h \phi \longrightarrow K \otimes \frakII_\phi$ is an anodyne
        extension.
    \end{enumerate}
\end{lemma}
\begin{proof}
    In both cases, it is enough to check the claim where $K$ is representable,
    say $K = \Delta [m]$ for some $m \in \bbNN$.
    \begin{enumerate}
        \item Clearly, the inclusion $\Delta [m] \otimes \sfH_{h \omega}^e :
        \Delta [m] \otimes \Lambda^e h \omega \longrightarrow \Delta [m]
        \otimes h \omega$ satisfies the condition of
        \cref{coroll:generation-lemma}, i.e. $\Delta [m] \otimes \Lambda^e h
        \omega$ generates $v (\Delta [m] \otimes h \omega)$.

        \item Write $*$ the unique element of $h \phi_{h \phi}$ (corresponding
        the the identity map $h \phi \longrightarrow h \phi$ in $\bbLambda$),
        and $0_m, \ldots, m_m$ the vertices of $\Delta [m]$. In particular,
        $(\Delta [m] \otimes h \phi)_{h \phi} = \left\{ 0_m \otimes *, \ldots,
        m_m \otimes * \right\}$.

        Let $T \in \PshLambda$ be the sum of $m+1$ copies of $h \phi$, and
        write $T_{h \phi} = \left\{ *_0, \ldots, *_m \right\}$. Clearly,
        $\frakII T$ is the sum of $m+1$ copies of $\frakII_\phi$, and the
        inclusion $i_0 : T \longrightarrow \frakII T$ is the sum of $m+1$
        copies of the endpoint inclusion $\sfI_\phi : h \phi \longrightarrow
        \frakII_\phi$. There is an obvious inclusion $t : T \longrightarrow
        \Delta [m] \otimes h \phi$ mapping the $*_i$ to $i_m \otimes *$. Let $X
        \in \PshLambda$ be defined by the following pushout:
        \[
            \pushoutdiagram
                {T}{\Delta [m] \otimes h \phi}{\frakII T}{X .}
                {t}{i_0}{u}{}
        \]
        In other words, $X$ is $\Delta [m] \otimes h \phi$ where an instance of
        $\frakII_\phi$ has been glued to each cell $i_m \otimes *$. The map
        $\Delta [m] \otimes \sfI_\phi$ factors as
        \[
            \Delta [m] \otimes h \phi
            \stackrel{u}{\longhookrightarrow} X
            \stackrel{v}{\longhookrightarrow} \Delta [m] \otimes \frakII_\phi ,
        \]
        and by construction, $u$ is an anodyne extension. On the other hand,
        $v$ satisfies the conditions of \cref{coroll:generation-lemma}, as the
        cells not in its image can be obtained as composites of cells in $X$.
        Therefore, $v$ is an anodyne extension, and so is $\Delta [m] \otimes
        \sfI_\phi$.
        \qedhere
    \end{enumerate}
\end{proof}

\begin{corollary}
    \label{coroll:pshlambda:simplicial-cotensor:naive-fibrations}
    For $K \in \PshDelta$, the functor $(-)^K$ preserves naive fibration
    (\cref{def:cisinski-model-structure}), and in particular,
    $\infty$-algebras.
\end{corollary}
\begin{proof}
    Let $\sfLL \eqdef \sfHH_{\mathrm{inner}} \cup \sfEE_\frakII$, so that
    $\mathsf{An}_\frakII = {}^\pitchfork \left( \sfLL^\pitchfork \right)$. Let
    $p$ be a naive fibration, i.e. a map such that $\mathsf{An}_\frakII
    \pitchfork p$, or equivalently, such that $\sfLL \pitchfork p$. By
    \cref{lemma:pshlambda:simplicial-tensor:anodyne-extensions}, we have $K
    \otimes \sfLL \pitchfork p$, and by adjunction, $\sfLL \pitchfork p^K$.
    Consequently, $p^K$ is a naive fibration.
\end{proof}

\begin{lemma}
    \label{lemma:tensor-preserves-cofs-weq}
    \begin{enumerate}
        \item For $K \in \PshDelta$, the tensor $K \otimes - :
        \PshLambda_\infty \longrightarrow \PshLambda_\infty$ preserves
        cofibrations and weak equivalences.
        \item Let $X \in \PshLambda$ be an $\infty$-algebra. Then $- \otimes X
        : \PshDelta_{\mathrm{Quillen}} \longrightarrow \PshLambda_\infty$
        preserves cofibrations, and weak equivalences between Kan complexes.
    \end{enumerate}
\end{lemma}
\begin{proof}
    \begin{enumerate}
        \item Clearly, $K \otimes -$ preserves monomorphisms. Let $u : X
        \longrightarrow Y$ be a weak equivalence, and $P \in \PshLambda$ be an
        $\infty$-algebra. Recall from
        \cref{lemma:tensor-cotensor-transpose-preserves-homotopy} that we have
        a natural isomorphism
        \[
            \Phi : \PshLambda (K \otimes X, Y) / {\simeq}
                \longrightarrow \PshLambda (X, Y^K) / {\simeq} ,
        \]
        where $\Phi$ is the natural $\hom$-set isomorphism of the adjunction $K
        \otimes - \dashv (-)^K$. Therefore, we have the following naturality
        square:
        \[
            \squarediagram
                {\PshLambda (K \otimes X, P) / {\simeq}}
                    {\PshLambda (K \otimes Y, P) / {\simeq}}
                    {\PshLambda (X, P^K) / {\simeq}}
                    {\PshLambda (Y, P^K) / {\simeq} .}
                {(K \otimes u)^*}{\Phi}{\Phi}{u^*}
        \]
        The vertical maps are bijections. By
        \cref{coroll:pshlambda:simplicial-cotensor:naive-fibrations}, $P^K$ is
        an $\infty$-algebra, thus $u^*$ is a bijection as well. Therefore, $(K
        \otimes u)^*$ is a bijection for all $\infty$-algebra $P$. By
        definition, $K \otimes u$ is a weak equivalence.

        \item Clearly, $- \otimes X$ preserves monomorphisms. Let $w : K
        \longrightarrow L$ be a weak equivalence between Kan complexes. By
        \cite[theorem 1.2.10]{Hovey1999}, it is a homotopy equivalence, meaning
        that it admits an inverse $w^{-1} : L \longrightarrow K$ up to
        homotopy, and $w^{-1} \otimes X$ is an inverse of $w \otimes X$ up to
        homotopy.
        \qedhere
    \end{enumerate}
\end{proof}

\begin{corollary}
    \label{coroll:leibniz-tensor-preserves-cof-weq}
    Let $u : K \longrightarrow L$ be a cofibration between Kan complexes, $v :
    X \longrightarrow Y$ be a cofibration in $\PshLambda$, and consider the
    Leibniz tensor $u \mathbin{\hat{\otimes}} v$
    (\cref{def:leibniz-construction}). Then $u \mathbin{\hat{\otimes}} v$ is a
    cofibration. If either $u$ or $v$ is an acyclic cofibration, then so is $u
    \mathbin{\hat{\otimes}} v$.
\end{corollary}
\begin{proof}
    Surely, since $u$ and $v$ are monomorphisms, $u \mathbin{\hat{\otimes}} v$
    is too. Assume that $u$ is an acyclic cofibration. By
    \cref{lemma:tensor-preserves-cofs-weq}, $u \otimes X$ and $u \otimes Y$ are
    too, and so is the pushout $u'$ of $u \otimes X$ along $K \otimes v$. By
    3-for-2, $u \mathbin{\hat{\otimes}} v$ is a weak equivalence. The case
    where $v$ is an acyclic cofibration instead of $u$ is done similarly.
\end{proof}

\begin{proposition}
    \label{prop:fibrant-cotensor}
    Let $X \in \PshLambda$ be an $\infty$-algebra, and $v : K \longrightarrow
    K$ be a cofibration (resp. acyclic cofibration) between Kan complexes. Then
    $X^v : X^L \longrightarrow X^K$ is a fibration (resp. acyclic fibration).
\end{proposition}
\begin{proof}
    Assume that $v$ is an cofibration (resp. acyclic cofibration). In order to
    show that $X^v$ is an fibration (resp. acyclic fibration), we must show
    that $u \pitchfork X^v$ for all acyclic cofibration (resp. cofibration) $u$
    in $\PshLambda_\infty$. This is equivalent to $u \mathbin{\hat{\otimes}} v
    \pitchfork X$. By \cref{coroll:leibniz-tensor-preserves-cof-weq}, $u
    \mathbin{\hat{\otimes}} v$ is an acyclic cofibration, and since $X$ is an
    $\infty$-algebra, the result holds.
\end{proof}

\begin{remark}
    Unfortunately, the model structure $\PshLambda_\infty$, together with the
    tensor and cotensor of
    \cref{def:pshlambda:simplicial-tensor,def:pshlambda:simplicial-cotensor-mapping-space}
    cannot be promoted into a simplicial model category. In fact, this already
    fails if $n = 1$ \cite[section 6]{Joyal2007}.
\end{remark}

\section{\texorpdfstring{$\infty$}{Infinity}-algebras vs. complete Segal spaces}
\label{sec:infty-algebras-vs-complete-segal-spaces}

In \cite{HoThanh2020}, we introduced the Segal and Rezk model structure on
$\SpLambda$ as Bousfield localization of the projective structure. Here, we
take a slightly different approach by starting with the \emph{vertical} (or
\emph{Reedy}) structure. This corresponds to the theory we generalize
\cite{Joyal2007,Cisinski2013}, while still being equivalent to the notions of
\cite{HoThanh2020} as there is a Quillen equivalence $\id :
\SpLambda_{\mathrm{proj}} \adjunction \SpLambda_v : \id$.

\subsection{Segal spaces}
\label{sec:segal-spaces}

\begin{definition}
    \label{def:segal-model-structure}
    Recall from \cref{def:vertical-horizontal-model-structures} that
    $\SpLambda_v$ is the Reedy structure on $\SpLambda =
    \PshDelta^{\bbLambda^\op}$ induced by $\PshDelta_{\mathrm{Quillen}}$. Let
    $\SpLambda_{\mathrm{Segal}}$\index{$\SpLambda_{\mathrm{Segal}}$|see {Segal
    model structure}}, the \emph{Segal model structure}\index{Segal model
    structure} on $\SpLambda$, be the left Bousfield localization of
    $\SpLambda_v$ at the set $\sfSS$ of spine inclusions
    (\cref{def:spine-lambda}), which exists by \cite[theorem
    4.1.1]{Hirschhorn2009}.

    Fibrant objects (resp. weak equivalences) in $\SpLambda_{\mathrm{Segal}}$
    are called \emph{Segal spaces}\index{Segal space} (resp. \emph{Segal weak
    equivalences}\index{Segal weak equivalence|see {Segal model structure}}).
    Explicitely, a Segal space $X \in \PshLambda$ is a vertically fibrant space
    such that for all $\omega \in \bbOO_{n-1}$, the map $\sfS_{h \omega}
    \backslash X : h \omega \backslash X = X_{h \omega} \longrightarrow S [h
    \omega] \backslash X$ is a weak equivalence.
\end{definition}

\begin{lemma}
    \label{lemma:saturation-rcp}
    Let $\sfKK$ be a saturated class of monomorphism of $\PshLambda$ having the
    \emph{right cancellation property}\index{right cancellation property}, i.e.
    such that for all composable pair or morphisms $f, g \in \PshLambda$, if
    $fg, g \in \sfKK$, then $f \in \sfKK$.
    \begin{enumerate}
        \item (Generalization of \cite[lemma 3.5]{Joyal2007} and
        \cite[proposition 2.5]{Cisinski2013}) If $\sfSS \subseteq \sfKK$, then
        $\mathsf{An}_{\mathrm{inner}} \subseteq \sfKK$.
        \item (Generalization of \cite[lemma 3.7]{Joyal2007}) If $\sfKK$
        contains all elementary face embeddings, then $\mathsf{An} \subseteq
        \sfKK$.
    \end{enumerate}
\end{lemma}
\begin{proof}
    \begin{enumerate}
        \item Since $\sfKK$ is saturated, we have $\Cell\sfSS \subseteq \sfKK$,
        so by \cref{lemma:spine-super-horn}, $\sfHH_{\mathrm{inner}} \subseteq
        \sfKK$. By saturation again, $\mathsf{An}_{\mathrm{inner}} \subseteq
        \sfKK$.

        \item For $\omega \in \bbOO_{n+1}$, let $F (\omega)$ be the set of
        elementary faces of $h \omega$. It suffices to show that for all
        $\omega \in \bbOO_{n+1}$ and non empty set $I \subsetneq F (\omega)$,
        the inclusion $\sfH_{h \omega}^I : \Lambda^I h \omega \longrightarrow$
        (cf. \cref{sec:infty-alg:model-structure:anodyne-extensions} for
        notations) is in $\sfKK$.

        Take a set $I \subsetneq F (\omega)$ and $f \in F (\omega) - I$, say $f
        : h \omega' \longrightarrow h \omega$. If $I = F (\omega) - \{ f \}$,
        then clearly, $\Lambda^I h \omega = \im f \cong h \omega'$, and
        $\sfH_{h \omega}^I \cong f \in \sfKK$.

        Otherwise, let $I \subsetneq J \subsetneq F (\omega)$ be a set of
        elementary faces of $h \omega$ not containing $f$. Then $f$ factors as
        \[
            h \omega'
            \stackrel{u}{\longhookrightarrow} \Lambda^J h \omega
            \stackrel{v}{\longhookrightarrow} \Lambda^I h \omega
            \stackrel{\sfH_{h \omega}^I}{\longhookrightarrow} h \omega .
        \]
        Since $f \in \sfKK$ be assumption, and since $\sfKK$ has the right
        cancellation property, in order to prove that $\sfH_{h \omega}^I$, it
        suffices to show that $u, v \in \sfKK$.
        \begin{enumerate}
            \item Assuming $v \in \sfKK$ for all $J \supsetneq I$ not
            containing $f$, it is enough to show that $u \in \sfKK$ in the case
            where $J = F (\omega) - \{ f \}$. But in this case, as before,
            $\Lambda^J h \omega = \im f \cong h \omega'$, and $u$ is an
            isomorphism.
            \item It is enough to consider the case $J = I + \{ g \}$, for some
            elementary face $g \in F (\omega)$ different from $f$. Since
            $\Lambda^I h \omega = \im f \cup \Lambda^J h \omega$, we have the
            following pushout square:
            \[
                \diagramarrows{c->}{c->}{c->}{c->}
                \pushoutdiagram
                    {\im f \cap \Lambda^J h \omega}{\Lambda^J h \omega}
                        {\im f}{\Lambda^I h \omega .}
                    {}{w}{v}{}
            \]
            Recall that $\im f \cong h \omega'$, and it is easy to see that the
            inclusion $w : \im f \cap \Lambda^J h \omega \longhookrightarrow
            \im f$ is isomorphic to an elementary face of $h \omega'$.
            Consequently, $v$ is the pushout of $w$, which is in $\sfKK$ by
            assumption, and thus $v \in \sfKK$.
            \qedhere
        \end{enumerate}
    \end{enumerate}
\end{proof}

\begin{proposition}
    [{Generalization of \cite[proposition 3.4]{Joyal2007}}]
    \label{prop:joyal-tierney:3.4}
    Let $X \in \SpLambda$ be vertically fibrant. The following are equivalent:
    \begin{enumerate}
        \item $X$ is a Segal space;
        \item the map $\sfH_{h \omega}^f \backslash X$ is an acyclic Kan
        fibration, for all $\omega \in \bbOO_{n+1}$ and all inner horn
        inclusion $\sfH_{h \omega}^f : \Lambda^f h \omega \longrightarrow h
        \omega$;
        \item the map $u \backslash X$ is an acyclic Kan fibration, for all
        inner anodyne extension $u \in \PshLambda$;
        \item the map $X / \sfB_n$ is a inner fibration
        (\cref{def:inner-anodyne-extension}), for all $n \in \bbNN$;
        \item the map $X / v$ is a inner fibration, for all monomorphism $v \in
        \PshDelta$.
    \end{enumerate}
\end{proposition}
\begin{proof}
    All equivalences are straightforward, except for (1) $\iff$ (3).
    \begin{itemize}
        \item (1) $\implies$ (3). Let $\sfKK$ be the class of morphisms $u \in
        \PshLambda$ such that $u \backslash X$ is an acyclic Kan fibration.
        Since $X$ is a segal space, $\sfSS \subseteq \sfKK$, and clearly,
        $\sfKK$ has the right cancellation property. We now show that it is
        saturated. By definition, $u \in \sfKK$ if and only if $u \backslash X$
        is an acyclic Kan fibration, i.e. $\sfB_n \pitchfork u \backslash X$
        for all $n \in \bbNN$. But this is equivalent to $u \pitchfork X /
        \sfB_n$, and thus $\sfKK = {}^\pitchfork \left\{ X / \sfB_n \mid n \in
        \bbNN \right\}$, and in particular, $\sfKK$ is saturated. Finally, by
        \cref{lemma:saturation-rcp}, $\sfKK$ contains all inner anodyne
        extensions.
        \item (3) $\implies$ (1). Recall that spine inclusions are inner
        anodyne extensions by \cref{prop:spine-complex-inner-anodyne}.
        \qedhere
    \end{itemize}
\end{proof}

\begin{corollary}
    [{Generalization of \cite[corollary 3.6]{Joyal2007}}]
    \label{coroll:joyal-tierney:3.6}
    Take $X \in \SpLambda$ be a Segal space, and $K \in \PshDelta$. Then $X /
    K$ is an inner Kan complex. In particular, $X_{-, n} = X / \Delta [n]$ is
    an inner Kan complex for all $n \in \bbNN$.
\end{corollary}
\begin{proof}
    For $k : \emptyset \longhookrightarrow K$ the initial map, the map $X / k :
    X / K \longrightarrow X / \emptyset = 1$ is an inner fibration by
    \cref{prop:joyal-tierney:3.4}.
\end{proof}

\begin{lemma}
    [{Generalization of \cite[lemma 3.8]{Joyal2007}}]
    \label{lemma:joyal-tierney:3.8}
    Let $f : X \longrightarrow Y$ be an inner fibration between
    $\infty$-algebras. It is an acyclic fibration if and only if it is a weak
    equivalence surjective on $(n-1)$-cells.
\end{lemma}
\begin{proof}
    \begin{enumerate}
        \item ($\implies$) Surely, if $f$ is an acyclic fibration, it is a weak
        equivalence. Moreover, $f$ has the right lifting property against all
        monomorphisms (recall \cref{th:pshlambda:infty-model-structure}), and
        in particular, against inclusions of the form $\emptyset
        \longhookrightarrow h \phi$, for $\phi \in \bbOO_{n-1}$. Therefore, $f$
        is surjective on $(n-1)$-cells.

        \item ($\impliedby$) By
        \cref{th:pshlambda:infty-model-structure,lemma:cisinski:1.3.36} $f$ is
        a fibration if and only if it is an inner fibration, and $\sfEE_\frakII
        \pitchfork f$. By \cref{lemma:v-reflects-isofibrations}, it suffices to
        show that $\sfEE_\frakJJ \pitchfork vf$. Note that $f$ is a weak
        equivalence between cofibrant objects. By Ken Brown's lemma \cite[lemma
        1.1.12]{Hovey1999}, $v f$ is a weak equivalence. Besides, it is also
        surjective on $(n-1)$-cells, so by
        \cref{th:folk-model-structure-algebras}, it is an acyclic fibration. In
        particular, it is a fibration, and we apply
        \cref{lemma:v-reflects-isofibrations} to conclude.
        \qedhere
    \end{enumerate}
\end{proof}

\begin{proposition}
    [{Generalization of \cite[proposition 3.9]{Joyal2007}}]
    \label{prop:joyal-tierney:3.9}
    A presheaf $X \in \SpLambda$ is a Segal space if and only if the following
    conditions are satisfied:
    \begin{enumerate}
        \item for all $n \in \bbNN$, the map $X / \sfB_n$ is an inner fibration;
        \item $X_{h \phi}$ is a Kan complex for all $\phi \in \bbOO_{n-1}$;
        \item $X$ is homotopically constant (see
        \cref{def:homotopically-constant}).
    \end{enumerate}
\end{proposition}
\begin{proof}
    \begin{itemize}
        \item ($\implies$) Assume that $X \in \SpLambda$ is a Segal space. In
        particular, it is vertically fibrant.
        \begin{enumerate}
            \item This is \cref{prop:joyal-tierney:3.4} (4).
            \item The terminal map $! : X \longrightarrow 1$ is a vertical
            fibration, so by \cref{def:vertical-horizontal-model-structures}, for $\phi \in
            \bbOO_{n-1}$, the map $\langle \sfB_{h \phi} \backslash ! \rangle :
            X_{h \phi} \longrightarrow 1$ is a Kan fibration.
            \item This is \cref{prop:joyal-tierney:2.8}.
        \end{enumerate}

        \item ($\impliedby$) Let $v : K \longrightarrow L$ be an anodyne
        extension in $\PshDelta$ such that $X / v$ is a weak equivalence in
        $\PshLambda_\infty$. We claim that $X / v$ is an acyclic fibration. By
        \cref{lemma:joyal-tierney:3.8}, it suffices to show that $X / v : X / L
        \longrightarrow X / K$ is surjective on $(n-1)$-cells. Note that for
        $\phi \in \bbOO_{n-1}$,
        \[
            (X / L)_{h \phi}
            = \PshLambda (h \phi, X / L)
            \cong \SpLambda (h \phi \boxproduct L, X)
            \cong \PshDelta (L, X_{h \phi}) ,
        \]
        and likewise, $(X / K)_{h \phi} \cong \PshDelta (K, X_{h \phi})$. By
        assumption, $X_{h \phi}$ is a Kan complex, and since $v$ is an anodyne
        extension, the precomposition map on top is surjective
        \[
            \squarediagram
                {\PshDelta (L, X_{h \phi})}{\PshDelta (K, X_{h \phi})}
                    {(X / L)_{h \phi}}{(X / K)_{h \phi}.}
                {v^*}{\cong}{\cong}{(X / v)_{h \phi}}
        \]
        Therefore, $X / v$ is surjective on $(n-1)$-cells, and thus an acyclic
        fibration as claimed.

        We now show that $X$ is vertically fibrant. By
        \cref{prop:joyal-tierney:2.5}, this is equivalent to $X / v$ being an
        acyclic fibration for all anodyne extension $v \in \PshDelta$. Let
        $\sfKK$ be the class of anodyne extensions $v$ such that $X / v$ is an
        acyclic fibration. Note that by the first claim, this is equivalent to
        $X / v$ being a weak equivalence. Using \cref{lemma:saturation-rcp},
        it suffices to show that $\sfKK$ has the right cancellation property,
        is saturated, and that it contains all simplicial face maps.
        \begin{enumerate}
            \item Right cancellation follows from 3-for-2.
            \item For staturation, note that $v \in \sfKK$ if and only if
            $\sfB_{h \omega} \pitchfork X / v$ for all $\omega \in \bbOO_{\geq
            n-1}$ by definition. This is equivalent to $v \pitchfork \sfB_{h
            \omega} \backslash X$, thus $\sfKK = {}^\pitchfork \left\{ \sfB_{h
            \omega} \backslash X) \mid \omega \in \bbOO_{\geq n-1} \right\}$ is
            saturated.
            \item Lastly, let us show that the simplicial face maps $d^i :
            \Delta [n-1] \longrightarrow \Delta [n]$ belong to $\sfKK$. Since
            $X$ is homotopically constant, the simplicial map $X_{-, n}
            \longrightarrow X_{-, 0}$ is a weak equivalence for all $n \in
            \bbNN$, so by 3-for-2, $X / d^i$ is too, as displayed by the
            following triangle
            \[
                \triangleDdiagram
                    {X_{-, n}}{X_{-, n-1}}{X_{-, 0}.}
                    {X / d^i}{\sim}{\sim}
            \]
        \end{enumerate}
        Therefore, by \cref{prop:joyal-tierney:2.5}, $\sfKK$ contains all
        anodyne extensions, and by \cref{prop:joyal-tierney:2.5}, $X$ is
        vertically fibrant.

        Lastly, we show that $X$ satisfies the Segal condition
        \cref{def:segal-model-structure}. It suffices to show that for all
        $\omega \in \bbOO_{\geq n-1}$, the map $\sfS_{h \omega} \backslash X$
        is an acyclic fibration, i.e. that for all $n \in \bbNN$, we have
        $\sfB_n \pitchfork \sfS_{h \omega} \backslash X$. This is equivalent to
        $\sfS_{h \omega} \pitchfork X / \sfB_n$, which holds since $\sfS_{h
        \omega}$ is an inner anodyne extension by
        \cref{prop:spine-complex-inner-anodyne}.
        \qedhere
    \end{itemize}
\end{proof}

\begin{proposition}
    [{Generalization of \cite[proposition 3.10]{Joyal2007}}]
    \label{prop:joyal-tierney:3.10}
    Let $f : X \longrightarrow Y$ be a vertical fibration between two Segal
    spaces.
    \begin{enumerate}
        \item If $u : A \longrightarrow B$ is an inner anodyne extension in
        $\PshLambda$ (\cref{def:inner-anodyne-extension}), then $\langle u
        \backslash f \rangle : B \backslash X \longrightarrow B \backslash Y
        \prod_{A \backslash Y} A \backslash X$ is an acyclic fibration.
        \item If $v : K \longrightarrow L$ is a monomorphism in $\PshDelta$,
        then $\langle f / v \rangle : X / L \longrightarrow Y / L \prod_{Y /
        K} X / K$ is an inner fibration between $\infty$-algebras.
    \end{enumerate}
\end{proposition}
\begin{proof}
    \begin{enumerate}
        \item By \cref{prop:joyal-tierney:2.5}, the map $\langle v \backslash f
        \rangle$ is a Kan fibration. It remains to show that it is a weak
        equivalence. Consider the following diagram:
        \[
            \begin{tikzcd}
                B \backslash X
                    \ar[drr, bend left, "u \backslash X"]
                    \ar[dr, "\langle u \backslash f \rangle"]
                    \ar[ddr, bend right, "B \backslash f"]
                & & \\
                &
                \cdot
                    \ar[r, "p_2"]
                    \ar[d, "p_1"]
                    \ar[dr, phantom, "\lrcorner" very near start] &
                A \backslash X
                    \ar[d, "A \backslash f"] \\
                &
                B \backslash Y
                    \ar[r, "u \backslash Y"] &
                A \backslash Y .
            \end{tikzcd}
        \]
        By \cref{prop:joyal-tierney:3.4}, $u \backslash X$ and $u \backslash Y$
        are trivial fibrations, and so is the pullback map $p_2$. By 3-for-2,
        $\langle u \backslash f \rangle$ is a weak equivalence.

        \item By \cref{prop:joyal-tierney:3.4}, $X / S$ and $X / L$ are
        $\infty$-algebras. In the pullback square
        \[
            \pullbackdiagram
                {Y / L \prod_{Y / K} X / K}{X / K}{Y / L}{Y / K ,}
                {p_2}{}{}{Y / v}
        \]
        the bottom map $Y / v$ is an inner fibration by
        \cref{prop:joyal-tierney:3.4}, and thus $p_2$ is too. Since $X / K$ is
        an $\infty$-algebra, so is $Y / L \prod_{Y / K} X / K$.

        We now show that $\langle f / v \rangle$ is an inner fibration, i.e.
        that $u \pitchfork \langle f / v \rangle$ for all $u \in
        \mathsf{An}_{\mathrm{inner}}$. By (1), $\langle u \backslash f \rangle$
        is an acyclic fibration, so $v \pitchfork \langle u \backslash f
        \rangle$, and by adjunction, $u \pitchfork \langle f / v \rangle$ as
        desired.
        \qedhere
    \end{enumerate}
\end{proof}

\subsection{Complete Segal spaces}
\label{sec:complete-segal-spaces}

\begin{definition}
    [Rezk map]
    \label{def:rezk-map}
    For $\phi \in \bbOO_{n-1}$, recall from \cref{def:rezk-interval-lambda} the
    definition of the Rezk interval $\frakII_\phi$, and the endpoint inclusion
    $\sfI_\phi : h \phi \longrightarrow \frakII_\phi$. There is a canonical
    morphism, called the \emph{Rezk map}\index{Rezk map} at $\phi$,
    \[
        \sfR_\phi : \frakII_\phi \longrightarrow h \phi ,
    \]\index{$\sfR_\phi$|see {Rezk map}}
    mapping $j_\phi$ and $j_\phi^{-1}$ to $\id_\phi$ (see
    \cref{def:rezk-interval-algebra} for notations), and let $\sfRR =
    \left\{\sfR_\phi \mid \phi \in \bbOO_{n-1} \right\}$ be the set of Rezk
    maps.
\end{definition}

\begin{definition}
    \label{def:rezk-model-structure}
    Let $\SpLambda_{\mathrm{Rezk}}$\index{$\SpLambda_{\mathrm{Rezk}}$|see {Rezk
    model structure}}, the \emph{Rezk model structure}\index{Rezk model
    structure} on $\SpLambda$, be the left Bousfield localization of
    $\SpLambda_{\mathrm{Segal}}$ (\cref{def:segal-model-structure}) at the set
    of Rezk maps $\sfRR$, which exists by \cite[theorem 4.1.1]{Hirschhorn2009}.

    Fibrant objects (resp. weak equivalences) in $\SpLambda_{\mathrm{Rezk}}$
    are called \emph{complete Segal spaces}\index{complete Segal space} (resp.
    \emph{Rezk weak equivalences}\index{Rezk weak equivalence|see {Rezk model
    structure}}). Explicitely, a Segal space $X \in \PshLambda$ is complete if
    for all $\phi \in \bbOO_{n-1}$, the map $\sfR_\phi \backslash X : X_{h
    \phi} \longrightarrow \frakII_\phi \backslash X$ is a weak equivalence.
\end{definition}

\begin{lemma}
    [{Generalization of \cite[lemma 4.2]{Joyal2007}}]
    \label{lemma:joyal-tierney:4.2}
    A Segal space $X \in \SpLambda$ is complete if and only if for all $\phi
    \in \bbOO_{n-1}$, the map $\sfI_\phi \backslash X$ is a trivial fibration, where
    $\sfI_\phi : h \phi \longrightarrow \frakII_\phi$ is the endpoint inclusion
    of the Rezk interval $\frakII_\phi$ (\cref{def:rezk-interval-lambda}).
\end{lemma}
\begin{proof}
    By definition, $X$ is complete if and only if for all $\phi \in
    \bbOO_{n-1}$, the map $\sfR_\phi \backslash X$ is a weak equivalence. Since
    $\sfR_\phi \sfI_\phi = \id_{h \phi}$, we have $(\sfI_\phi \backslash X)
    (\sfR_\phi \backslash X) = \id_{X_{h \phi}}$, and by 3-for-2, $\sfI_\phi
    \backslash X$ is a weak equivalence if and only if $\sfR_\phi \backslash X$
    is. On the other hand, $\sfI_\phi \backslash X$ is always a Kan fibration by
    \cref{prop:joyal-tierney:2.5}. Hence it is a trivial fibration if and only
    if it is a weak equivalence.
\end{proof}

\begin{lemma}
    [{Generalization of \cite[lemma 4.3]{Joyal2007}}]
    \label{lemma:joyal-tierney:4.3}
    Let $f : X \longrightarrow Y$ be a Rezk fibration (i.e. a fibration in
    $\SpLambda_{\mathrm{Rezk}}$) between two complete Segal spaces, and $u : K
    \longhookrightarrow L$ be a monomorphism in $\PshDelta$. Then the map
    $\langle f / u \rangle : X / L \longrightarrow (Y / L) \prod_{Y / K} (X /
    K)$ is a fibration.
\end{lemma}
\begin{proof}
    By \cref{prop:joyal-tierney:3.10}, $\langle f / u \rangle$ is an inner
    fibration between $\infty$-algebras. By \cref{lemma:cisinski:1.3.36},
    $\langle f / u \rangle$ us a fibration if and only if it is a naive
    fibration, so it remains to show that $\sfEE_\frakII \pitchfork \langle f /
    u \rangle$.

    By adjunction, for $\phi \in \bbOO_{n-1}$, we have $\sfI_\phi \pitchfork
    \langle f / u \rangle$ if and only if $u \pitchfork \langle \sfI_\phi
    \backslash f \rangle$. Thus, we must show that $\langle \sfI_\phi
    \backslash f \rangle$ is an acyclic fibration. By
    \cref{prop:joyal-tierney:2.5}, it is a fibration. Consider the following
    commutative diagram:
    \[
        \begin{tikzcd}
            \frakII_\phi \backslash X
                \ar[drr, bend left, "\sfI_\phi \backslash X"]
                \ar[dr, "\langle \sfI_\phi \backslash f \rangle"]
                \ar[ddr, bend right, "\frakII_\phi \backslash f"]
            & & \\
            &
            \cdot
                \ar[r, "p_2"]
                \ar[d, "p_1"]
                \ar[dr, phantom, "\lrcorner" very near start] &
            h \phi \backslash Y
                \ar[d, "h \phi \backslash f"] \\
            &
            \frakII_\phi \backslash Y
                \ar[r, "\sfI_\phi \backslash Y"] &
            h \phi \backslash Y .
        \end{tikzcd}
    \]
    By \cref{lemma:joyal-tierney:4.2}, $\sfI_\phi \backslash X$ and $\sfI_\phi
    \backslash Y$ are acyclic fibrations. Since $f$ is a Rezk fibration, it is
    a vertical fibration, and its matching map $\langle (\partial h \phi
    \hookrightarrow h \phi) \backslash f \rangle$ is a fibration by
    \cref{prop:joyal-tierney:2.5}. Since $\phi \in \bbOO_{n-1}$, $\partial h
    \phi = \emptyset$, thus $h \phi \backslash f = \langle (\partial h \phi
    \hookrightarrow h \phi) \backslash f \rangle$ is a fibration. The pullback
    map $p_2$ is an acyclic fibration, and by 3-for-2, so is $\langle \sfI_\phi
    \backslash f \rangle$.
\end{proof}

\begin{proposition}
    [{Generalization of \cite[proposition 4.4]{Joyal2007}}]
    \label{prop:joyal-tierney:4.4}
    A presheaf $X \in \SpLambda$ is a complete Segal space if and only if the
    following conditions are satisfied\footnote{In particular, complete Segal
    spaces are exactly the simplicial resolutions \cite[definition
    4.7]{Dugger2001a}.}:
    \begin{enumerate}
        \item $X / \sfB_n$ is a fibration for all $n \in \bbNN$, i.e. $X$ is
        horizontally fibrant (\cref{def:vertical-horizontal-model-structures});
        \item $X$ is homotopically constant
        (\cref{def:homotopically-constant}).
    \end{enumerate}
\end{proposition}
\begin{proof}
    \begin{itemize}
        \item ($\implies$) By \cref{prop:joyal-tierney:3.9}, $X$ is
        homotopically constant. By \cref{lemma:joyal-tierney:4.3}, $X / \sfB_n
        = \langle (X \rightarrow 1) / \sfB_n \rangle$ is a fibration as $X
        \longrightarrow 1$ is a Rezk fibration.

        \item ($\impliedby$) We first show that $X$ is vertically fibrant. By
        \cref{prop:joyal-tierney:2.5}, this is equivalent to $X / u$ being an
        acyclic fibration for all anodyne extension $u \in \PshDelta$. Let
        $\sfKK$ be the class of monomorphisms $u$ such that $X / u$ is an
        acyclic fibration. Using \cref{lemma:saturation-rcp}, it suffices to
        show that $\sfKK$ has the right cancellation property, is saturated,
        and that it contains all simplicial face maps.
        \begin{enumerate}
            \item By condition (1), $X / u$ is a fibration for any monomorphism
            $u \in \PshDelta$. Thus, $X / u$ is an acyclic fibration if and if
            it is a weak equivalence. The right cancellation property of
            $\sfKK$ then follows from 3-for-2.
            \item For staturation, note that $u \in \sfKK$ if and only if
            $\sfB_{h \omega} \pitchfork X / u$ for all $\omega \in \bbOO_{\geq
            n-1}$ by definition. This is equivalent to $v \pitchfork \sfB_{h
            \omega} \backslash X$, thus $\sfKK = {}^\pitchfork \left\{ \sfB_{h
            \omega} \mid \omega \in \bbOO_{\geq n-1} \right\}$ is saturated.
            \item Lastly, let us show that the simplicial face maps $d^i :
            \Delta [n-1] \longrightarrow \Delta [n]$ belong to $\sfKK$. Since
            $X$ is homotopically constant, the structure map $X_{-, n}
            \longrightarrow X_{-, 0}$ is a weak equivalence for all $n \in
            \bbNN$, so by 3-for-2, $X / d^i$ is too, as displayed by the
            following triangle
            \[
                \triangleDdiagram
                    {X_{-, n}}{X_{-, n-1}}{X_{-, 0}.}
                    {X / d^i}{\sim}{\sim}
            \]
            Since $d^i$ is a monomorphism, $X / d^i$ is a fibration by
            condition (1). Finally, $X / d^i$ is an acyclic fibration.
        \end{enumerate}
        Therefore, by \cref{prop:joyal-tierney:2.5}, $\sfKK$ contains all
        anodyne extensions, and by \cref{prop:joyal-tierney:2.5}, $X$ is
        vertically fibrant.

        Since the terminal map $! : X \longrightarrow 1$ is a vertical
        fibration, by \cref{def:vertical-horizontal-model-structures}, for
        $\phi \in \bbOO_{n-1}$, the map $\langle \sfB_{h \phi} \backslash !
        \rangle : X_{h \phi} \longrightarrow 1$ is a Kan fibration, and $X_{h
        \phi}$ is a Kan complex. By \cref{prop:joyal-tierney:3.9}, $X$ is a
        Segal space.

        Lastly, let us show that $X$ is complete. By
        \cref{lemma:joyal-tierney:4.2}, it suffices to show that $\sfI_\phi
        \backslash X$ is an acyclic fibration, for all $\phi \in \bbOO_n$
        Clearly, the enpoint inclusion $\sfI_\phi : h \phi \longrightarrow
        \frakII_\phi$ is a weak equivalence, with the Rezklemma:saturation-rcp-rcpp to homotopy. Since
        $\sfI_\phi$ is a monomorphism, it is a an acyclic cofibration. By
        condition (1), $\sfI_\phi \pitchfork X / \sfB_n$, for all $n \in
        \bbNN$, so by adjunction, $\sfB_n \pitchfork \sfI_\phi \backslash X$,
        and $\sfI_\phi \backslash X$ is a trivial fibration.
        \qedhere
    \end{itemize}
\end{proof}

\begin{theorem}
    [{Generalization of \cite[theorem 4.5]{Joyal2007}}]
    \label{th:joyal-tierney:4.5}
    \begin{enumerate}
        \item The model structure $\SpLambda_{\mathrm{Rezk}}$ is a Bousfield
        localization of the horizontal model structure $\SpLambda_h$
        (\cref{def:vertical-horizontal-model-structures}). In particular, a
        weak equivalence in $\SpLambda_h$ is a Rezk weak equivalence.
        \item An horizontally fibrant space is a complete Segal space if and
        only if it is homotopically constant
        (\cref{def:homotopically-constant}).
    \end{enumerate}
\end{theorem}
\begin{proof}
    \begin{enumerate}
        \item By \cref{prop:vertical-horizontal-model-structures:cisinski},
        $\SpLambda_v$ and $\SpLambda_h$ are both of Cisinski type. Since
        $\SpLambda_{\mathrm{Rezk}}$ is a left Bousfield localization of
        $\SpLambda_v$, it is also of Cisinski type. In particular,
        $\SpLambda_{\mathrm{Rezk}}$ and $\SpLambda_h$ have the same
        cofibrations, namely the monomorphisms. Thus, in order to prove the
        claim, it is enough to show that the identity functor induces a Quillen
        adjunction $\id : \SpLambda_h \adjunction \SpLambda_{\mathrm{Rezk}} :
        \id$. By a result of Dugger \cite[corollary A.2]{Dugger2001a} (also
        stated in \cite[proposition 8.5.4]{Hirschhorn2009}), it suffices to
        show that $\id$ preserves cofibrations, and Rezk fibrations between
        complete Segal spaces. In both structures, cofibrations are the
        monomorphisms. By \cref{lemma:joyal-tierney:4.3}, for $f : X
        \longrightarrow Y$ a Rezk fibration between complete Segal spaces, the
        matching map $\langle f / \sfB_n \rangle$ is a fibration for all $n \in
        \bbNN$.
        \item Follows from \cref{prop:joyal-tierney:4.4}
        \qedhere
    \end{enumerate}
\end{proof}

\begin{definition}
    Recall from \cite[proposition 1.16]{Joyal2007} the functor $J : \qCat
    \longrightarrow \Kan$ from quasi-categories to Kan complexes, that maps a
    quasi-category to its maximal sub Kan complex. It is right adjoint and
    preserves cofibrations, fibrations, and weak equivalences (from
    $\PshDelta_\infty$).

    For $X \in \PshLambda$, let $\Gamma X \in \SpLambda$ be given by $\Gamma
    X_{-, k} \eqdef X^{J (\Delta [k])}$, where the cotensor is given in
    \cref{def:pshlambda:simplicial-cotensor-mapping-space}.
\end{definition}

\begin{lemma}
    [{Generalization of \cite[lemma 4.8]{Joyal2007}}]
    \label{lemma:joyal-tierney:4.8}
    Let $\catCC$ be a small category, and $F : \PshDelta^\op \longrightarrow
    \Psh\catCC$ be a continuous functor, i.e. mapping colimits in $\PshDelta$
    to limits in $\Psh\catCC$. Then $F \cong G / -$, where $G \in \Sp\catCC$ is
    the restriction of $F$ to $\bbDelta$, i.e. $G_{-, k} \eqdef F \Delta [k]$,
    for $k \in \bbNN$.
\end{lemma}
\begin{proof}
    For $k \in \bbNN$, we have $G / \Delta [k] = G_{-, k} = F \Delta [k]$, thus
    $G / -$ and $F$ coincide on $\bbDelta^\op$. Since $\bbDelta^\op$ freely
    generates $\PshDelta^\op$ under small limits, we are done.
\end{proof}

\begin{proposition}
    [{Generalization of \cite[proposition 4.10]{Joyal2007}}]
    \label{prop:joyal-tierney:4.10}
    Let $X \in \SpLambda_\infty$ be an $\infty$-algebra. Then $\Gamma X$ is a
    complete Segal space, and there is a canonical acyclic cofibration
    $X^{\mathrm{disc}} \longrightarrow \Gamma X$, thus exhibiting $\Gamma X$ as
    a fibrant replacement of $X^{\mathrm{disc}}$ in
    $\SpLambda_{\mathrm{Rezk}}$.
\end{proposition}
\begin{proof}
    \begin{enumerate}
        \item Note that $X^{J (-)} : \PshDelta^\op \longrightarrow \PshLambda$
        is a continuous functor, thus by \cref{lemma:joyal-tierney:4.8}, there
        exists $G \in \SpLambda$ such that $X^{J (-)} \cong G / -$. In
        particular, for $k \in \bbNN$, we have $G_{-, k} \cong G / \Delta [k]
        \cong \Gamma X_{-, k}$, and therefore, $G \cong \Gamma X$, and $X^{J
        (-)} \cong \Gamma X / -$.

        \item As a consequence, and by \cref{prop:fibrant-cotensor}, $\Gamma X
        / \sfB_k$ is a fibration, for all $k \in \bbNN$. Likewise, the map $d :
        \Delta [0] \longrightarrow \Delta [n]$ is an acyclic cofibration in
        $\PshDelta_{\mathrm{Quillen}}$. Thereofre $\Gamma X / d : \Gamma X_{-,
        n} \longrightarrow \Gamma X_{-, 0}$ is a weak equivalence, and by
        \cref{lemma:homotopically-constant}, $\Gamma X$ is homotopically
        constant. By \cref{prop:joyal-tierney:4.4}, it is a complete Segal
        space.

        \item Note that $\Gamma X_{-, 0} \cong \Gamma X / \Delta [0] \cong X^{J
        (\Delta [0])} = X$, so we have an isomorphism $f : X \longrightarrow
        \Gamma X_{-, 0}$. We how extend $f$ to a morphism $\gamma :
        X^{\mathrm{disc}} \longrightarrow \Gamma X$. First, let $\gamma_0
        \eqdef f : X^{\mathrm{disc}}_{-, 0} = X \longrightarrow \Gamma X_{-,
        0}$. Next, note that in $\bbDelta$, the terminal map $s_k : [k]
        \longrightarrow [0]$ is a retraction of any iterated coface map $d :
        [0] \longrightarrow [k]$. Thus, in addition to being a weak equivalence
        (as $\Gamma X$ is homotopically constant), the structure map $s_k :
        \Gamma X_{-, 0} \longrightarrow \Gamma X_{-, k}$ is a section of $d :
        \Gamma X_{-, k} \longrightarrow \Gamma X_{-, 0}$, thus injective.
        Letting $\gamma_k \eqdef s_k f : X^{\mathrm{disc}}_{-, k} = X
        \longrightarrow \Gamma X_{-, k}$ gives rise to the desired map $\gamma
        : X^{\mathrm{disc}} \longrightarrow \Gamma X$ which by construction is
        a horizontal acyclic cofibration. By \cref{th:joyal-tierney:4.5}, it is
        a Rezk acyclic cofibration.
        \qedhere
    \end{enumerate}
\end{proof}

\begin{theorem}
    [{Generalization of \cite[proposition 4.7 and theorem 4.11]{Joyal2007}}]
    \label{th:joyal-tierney:4.11}
    We have a Quillen equivalence $(-)^\mathrm{disc} : \PshLambda_\infty
    \adjunction \SpLambda_{\mathrm{Rezk}} : (-)_{-, 0}$ (see
    \cref{sec:joyal-tierney-calculus}).
\end{theorem}
\begin{proof}
    \begin{enumerate}
        \item We first shows that the adjunction is Quillen. Clearly,
        $(-)^{\mathrm{disc}}$ preserves monomorphisms and maps weak
        equivalences to horizontal weak equivalences, which by
        \cref{th:joyal-tierney:4.5} are Rezk weak equivalences. Therefore,
        $(-)^{\mathrm{disac}}$ is left Quillen functor.

        \item We show that for a complete Segal space $X \in \SpLambda$, the map
        \[
            (Q X_{-, 0})^{\mathrm{disc}}
            \xrightarrow{q^{\mathrm{disc}}} (X_{-, 0})^{\mathrm{disc}}
            \stackrel{\epsilon}{\longrightarrow} X
        \]
        is a Rezk weak equivalence, where $q : Q X_{-, 0} \longrightarrow X_{-,
        0}$ a cofibrant replacement of $X_{-, 0}$. First, since all objects are
        cofibrant in $\PshLambda_\infty$, whoose $q$ to be an identity. Next,
        by \cref{th:joyal-tierney:4.5}, it suffices to show that $\epsilon$ is
        a weak equivalence in $\SpLambda_h$, i.e. that $\epsilon_{-, n} :
        (X_{-, 0})^{\mathrm{disc}}_{-, n} = X_{-, 0} \longrightarrow X_{-, n}$
        is a weak equivalence. Clearly, $\epsilon_{-, n}$ is induced by the
        terminal map $[n] \longrightarrow [0]$ in $\bbDelta$, and thus is a
        weak equivalence since $X$ is homotopically constant by
        \cref{prop:joyal-tierney:4.4}.

        \item We show that for an $\infty$-algebra $X \in \PshLambda_\infty$,
        the map
        \[
            X
            \stackrel{\cong}{\longrightarrow} (X^{\mathrm{disc}})_{-, 0}
            \xrightarrow{r_{-, 0}} (R X^{\mathrm{disc}})_{-, 0}
        \]
        is a Rezk weak equivalence, where $r : X^{\mathrm{disc}}
        \longrightarrow R X^{\mathrm{disc}}$ is a fibrant replacement of
        $X^{\mathrm{disc}}$. Choosing it to be $\gamma : X^{\mathrm{disc}}
        \longrightarrow \Gamma X$ (\cref{prop:joyal-tierney:4.10}) concludes
        the proof.
        \qedhere
    \end{enumerate}
\end{proof}

\printindex

\pagebreak

\bibliographystyle{alpha}
\bibliography{bibliography}

\end{document}